\documentclass[11pt]{article}

\usepackage[margin=1.25in]{geometry}

\usepackage{amssymb,amsmath}
\usepackage{amsthm}
\usepackage{hyperref}
\usepackage{mathrsfs}
\usepackage{textcomp}
\usepackage{enumerate}
\usepackage{framed}
\usepackage{ulem}
\usepackage{color,xcolor}
\usepackage{verbatim}
\usepackage{authblk}
\usepackage[title]{appendix}

\usepackage{soul}
\soulregister\cite7 
\soulregister\citep7 
\soulregister\citet7 
\soulregister\ref7 
\soulregister\pageref7 
\soulregister\eqref7

\hypersetup{
	colorlinks,%
	citecolor=blue,%
	filecolor=blue,%
	linkcolor=blue,%
	urlcolor=black
}

\newtheorem{theorem}{Theorem}[section]
\newtheorem{definition}{Definition}[section]
\newtheorem{lemma}{Lemma}[section]
\newtheorem{remark}{Remark}[section]
\newtheorem{proposition}{Proposition}[section]
\newtheorem{corollary}{Corollary}[section]

\numberwithin{equation}{section}

\makeatletter

\newdimen\bibspace
\renewenvironment{thebibliography}[1]{%
	\section*{\refname 
		\@mkboth{\MakeUppercase\refname}{\MakeUppercase\refname}}%
	\list{\@biblabel{\@arabic\c@enumiv}}%
	{\settowidth\labelwidth{\@biblabel{#1}}%
		\leftmargin\labelwidth
		\advance\leftmargin\labelsep
		\itemsep\bibspace
		\parsep\z@skip     %
		\@openbib@code
		\usecounter{enumiv}%
		\let\p@enumiv\@empty
		\renewcommand\theenumiv{\@arabic\c@enumiv}}%
	\sloppy\clubpenalty4000\widowpenalty4000%
	\sfcode`\.\@m}
{\def\@noitemerr
	{\@latex@warning{Empty `thebibliography' environment}}%
	\endlist}

\makeatother

\newcommand{\be}{\begin{equation}}      \newcommand{\ee}{\end{equation}}

\begin{document}

	\title{Compactness of Solutions to Sub-Elliptic Equations with Potential on the Heisenberg Group}
	
	\author[a]{Jiechen Qiang}
	\author[a]{Zhongwei Tang}
	\author[b]{Yichen Zhang}
	\author[c]{Ning Zhou}
	\affil[a]{School of Mathematical Sciences, Laboratory of Mathematics and Complex Systems, MOE, 
		Beijing Normal University, Beijing 100875, P. R. China}
	\affil[b]{School of Mathematical Sciences, Beihang University, Beijing 102206, P. R. China}
	\affil[c]{School of Mathematical Sciences and LPMC, Nankai University,
		Tianjin 300071, P. R. China}
	\renewcommand*{\Affilfont}{\small\it}
	\renewcommand\Authands{ and }
	\date{}

	\maketitle

	\begin{abstract}{
In this paper, we investigate the compactness of nonnegative solutions to a critical sub-elliptic equation with a nonnegative potential on the Heisenberg group. We establish compactness in all dimensions. 
Moreover, we show that if a sequence of solutions blows up, both the potential and its sub-Laplacian must vanish at the blow-up point. Our analysis overcomes the inherent geometric and analytical challenges posed by the Heisenberg group, including the degeneracy of the sub-Laplacian, its non-commutative structure, and the anisotropic dilation symmetry.

		}

	\end{abstract}
	

\section{Introduction}

Let $\mathbb{H}^n$ be the Heisenberg group with homogeneous dimension $Q=2n+2$. 
In this paper, we investigate the compactness of nonnegative solutions to the following critical sub-elliptic equation with a potential:
	\be\label{equation}
	\begin{cases}
		-\Delta_{\mathbb{H}^{n}} u = au+ u^{p}&\text { in } \Omega , \\
		u  \geq 0  & \text{ in } \Omega,
	\end{cases}
	\ee
where $\Omega \subset \mathbb{H}^n$ is a bounded domain, $p=\frac{Q+2}{Q-2}$ is the critical Sobolev exponent, $\Delta_{\mathbb{H}^n}$ denotes the sub-Laplacian on $\mathbb{H}^n$,
and $a(\xi)$ is assumed to be nonnegative and smooth.

In Euclidean and Riemannian settings, critical semilinear elliptic equations analogous to \eqref{equation} have been studied extensively. A prototypical example is the Yamabe equation on Riemannian manifolds, where the potential is given by the scalar curvature up to a dimensional constant:
\begin{equation}\label{eq:Yamabe}
-\Delta_g u+c(n) R_g u=\lambda u^{\frac{n+2}{n-2}},\quad u>0,
\end{equation}
where $\Delta_g$ is the Laplace-Beltrami operator on $(M,g)$, $c(n)=\frac{n-2}{4(n-1)}$, $R_g$ is the scalar curvature of $(M, g)$, and $\lambda$ is a constant.

When $\lambda<0$, the Yamabe equation \eqref{eq:Yamabe} admits a unique positive solution in the normalized Yamabe class; when $\lambda=0$, the Yamabe equation \eqref{eq:Yamabe} reduces to a linear equation, and its solution exists and is unique up to a constant factor; whereas for $\lambda>0$, Schoen \cite{Schoenunpublished} constructed examples of multiple high-energy and high Morse index solutions on $\mathbb{S}^1 \times \mathbb{S}^{n-1}$. Thus, a natural question arises: what can be concluded about the solution set of the Yamabe equation \eqref{eq:Yamabe}? Schoen \cite{Schoenunpublished} proved that the standard sphere is the only compact Riemannian manifold that admits a non-compact conformal diffeomorphism group action, and proposed the following compactness conjecture: if $(M, g)$ is not conformally equivalent to the standard sphere, then for $\lambda>0$, the solution set of the Yamabe equation \eqref{eq:Yamabe} is compact in the $C^2$ topology. For the case of locally conformally flat manifolds of dimension $n \geq 3$, Schoen \cite{Schoenunpublished} provided a proof of the compactness conjecture. Schoen's proof offers a strategy for addressing such compactness problems, namely the blow-up analysis method. For the non-locally conformally flat case, Li-Zhu \cite{Li Zhu CCM} gave a proof of the compactness conjecture for $n=3$; Druet \cite{DruetCompactness04} provided proofs for $n=4,5$; Li-Zhang \cite{Li Zhang CV} and Marques \cite{Marques} independently established the proof for $n=6,7$; Khuri-Marques-Schoen \cite{KMS} proved that the compactness conjecture holds for all $3 \leq n \leq 24$. On the other hand, Brendle \cite{B} and Brendle-Marques \cite{BM} constructed sequences of blowing-up solutions to the Yamabe equation on $\mathbb{S}^n$ with smooth non-conformally flat metrics for $n \geq 52$ and $25 \leq n \leq 51$, respectively. Hence, for $n \geq 25$, the compactness conjecture no longer holds.

Inspired by these geometric developments, there has been significant interest in extending compactness theories to Yamabe-type equations with non-geometric potentials.

Consider the critical Schr\"odinger-type equation, which is more general than the Yamabe equation:
\begin{equation}\label{eq:Schrodinger}
-\Delta_g u+h u=u^{\frac{n+2}{n-2}},\quad u>0,
\end{equation}
where $(M, g)$ is an $n$-dimensional $(n \geq 3)$ compact Riemannian manifold, and $h \in C^1(M)$ is such that the operator $-\Delta_g+h$ is coercive. When $h=c(n) R_g$, equation \eqref{eq:Schrodinger} reduces to the Yamabe equation \eqref{eq:Yamabe}. Assuming that $h<c(n) R_g$ holds everywhere on the manifold $M$, Li-Zhu \cite{Li Zhu CCM} and Druet \cite{DruetCompactness04} proved the {compactness of solutions} to equation \eqref{eq:Schrodinger} for the cases $n=3$ and arbitrary $n$, respectively.

Consider equations of the form
$$
(-\Delta)^\sigma u-a(x) u=u^{\frac{n+2 \sigma}{n-2 \sigma}},
$$
where the potential $a(x)$ plays a role analogous to the scalar curvature.
Niu-Peng-Xiong \cite{Niu Peng Xiong JFA} investigated the critical equation involving the fractional Laplacian. They proved that if the nonnegative potential possesses only non-degenerate zeros, the set of solutions is compact. Niu-Tang-Zhou \cite{Niu Tang Zhou IMRN} successfully extended this framework to higher-order critical elliptic equations. By employing blow-up analysis for local integral equations, they confirmed that the non-degeneracy of the potential's zeros serves as a sufficient condition for compactness in the higher-order case as well.

As pointed out by Jerison-Lee in \cite{CR yamabe}, the clear parallels between conformal geometry and the geometry of CR manifolds which serve as abstract models of real hypersurfaces in 
complex manifolds naturally motivated researchers to investigate Yamabe-type problems within the CR setting (see the book by Dragomir-Tomassini \cite{DG}). Motivated by these striking parallels between the Euclidean and CR settings, and the recent analytical progress for equations with non-geometric potentials, 
the primary objective of this paper is to extend the compactness results of \cite{Niu Peng Xiong JFA,Niu Tang Zhou IMRN} to the Heisenberg group $\mathbb{H}^n$. 

The first main result of the paper is as follows.

	\begin{theorem}\label{main thm 3}
		Let $u \in {C^{2 }(D_3)}$ be a nonnegative solution of
		$$
		-\Delta_{\mathbb{H}^n} u-a(\xi) u=u^{p} \quad \text { in } D_3,
		$$
		where $a(\xi)$ is a nonnegative smooth function in $D_3$. 
		If either 
		\begin{itemize}
			\item[$(i)$] $a>0$ in $D_2$, or
			\item[$(ii)$] $\Delta_{\mathbb{H}^n} a>0$ on $\{\xi: a(\xi)=0\} \cap D_2$ and $n\geq 2$,
		\end{itemize}
		holds, then
		$$
		\|u\|_{\Gamma^{2,\alpha}(D_1)} \leq C,
		$$
        where \(C>0\) depends only on \(n\), \(\|a\|_{C^4(D_3)}\) and $\inf_{D_2}a$ if $(i)$ holds; otherwise, it 
        depends only on \(n\), \(\|a\|_{C^4(D_3)}\) and $\inf_{\{\xi:a(\xi)=0\}\cap D_2}\Delta_{\mathbb{H}^n} a$.
	\end{theorem}
	
The proof of Theorem \ref{main thm 3} relies essentially on the classification of solutions to the 
critical equation $-\Delta_{\mathbb{H}^n} u = u^{p}$ on the whole space. Thanks to Li-Liu-Ma \cite{LLM} and Li \cite{LiJungangArxiv}, the complete classification results required to implement our 
blow-up arguments are established without any additional assumptions only in dimensions.

Furthermore, in the spirit of the results obtained in the Euclidean setting, we seek to derive a 
quantitative ``vanishing rate" for the potential at blow-up points. We prove that if a sequence of solutions blows up, the sub-Laplacian of the potential, 
$\Delta_{\mathbb{H}^n} a$, must vanish at the blow-up point. This result provides a CR-geometric counterpart to the analytical form of Schoen's Weyl tensor conjecture. The second main result in this paper is as follows.

	\begin{theorem}\label{main thm 2}
		 Let $u_i \in {C^{2 }(D_3)}, i=1,2, \cdots$, be nonnegative solutions of
		$$
		-\Delta_{\mathbb{H}^n} u_i-a_i(\xi) u_i=u_i^{p} \quad \text { in } D_3,
		$$
		where $a_i \geq 0,\left\|a_i\right\|_{{C^{4 }(D_3)}} \leq A_0$ for some $A_0>0$ and $a_i \rightarrow a$ in ${C^{4 }(D_3)}$. 
Suppose that $\Delta_{\mathbb{H}^n} a_i\geq 0$ in $\left\{\xi: a_i(\xi)<\varepsilon\right\} \cap D_2$ for some $\varepsilon>0$ independent of $i$ and $n \geq 2$. If $\xi_i \rightarrow \bar{\xi} \in D_1$ and $u_i\left(\xi_i\right) \rightarrow \infty$ as $i \rightarrow \infty$, then $a(\bar{\xi})=\Delta_{\mathbb{H}^n} a(\bar{\xi})=0$. Furthermore,
	
		(i) if $ n=2$, we can find $\xi_i^{\prime} \rightarrow \bar{\xi}$ such that
		$$
		a_i\left(\xi_i^{\prime}\right) u_i\left(\xi_i^{\prime}\right)\left(\ln u_i\left(\xi_i^{\prime}\right)\right)^{-1}+\Delta_{\mathbb{H}^n} a_i\left(\xi_i^{\prime}\right) \leq C\left(\ln u_i\left(\xi_i^{\prime}\right)\right)^{-1},
		$$
		where $C>0$ depends only on $\varepsilon$ and $A_0$.
		
		(ii) If $n \geq 3$, assume that
		$$
		\xi_i \text { is a local maximum point of } u_i, \quad\max _{D_{\bar{d}}\left(\xi_i\right)} u_i(\xi) \leq \bar{b} u_i\left(\xi_i\right)
		$$
		for some positive constants $\bar{b}$ and $\bar{d}$. Then,
		$$
		a_i\left(\xi_i\right) u_i\left(\xi_i\right)^{\frac{4}{Q-2}}+\Delta_{\mathbb{H}^n} a_i\left(\xi_i\right) \leq C \begin{cases}u_i\left(\xi_i\right)^{-\frac{4}{Q-2}} \ln u_i\left(\xi_i\right) & \text { for } n=3, \\ u_i\left(\xi_i\right)^{-\frac{4}{Q-2 }} & \text { for } n>3,\end{cases}
		$$
		where $C>0$ depends only on $n$, $\varepsilon$, $A_0$, as well as constants $\bar{b}$ and $\bar{d}$.
	\end{theorem}

We emphasize that the analysis on 
$\mathbb{H}^n$ presents distinct technical challenges compared to the Euclidean case, primarily due to the sub-elliptic  nature of the operator $\Delta_{\mathbb{H}^n}$ and the characteristic anisotropic dilation structure of the group. 
To derive our main results, we must overcome several significant difficulties arising from the degeneracy of the sub-Laplacian 
and the intrinsic geometric structure of the Heisenberg group. It is worth noting that the non-commutative structure, anisotropy, and sub-Riemannian geometry 
of the Heisenberg group make many classical methods in Euclidean space (or Riemannian manifolds) no longer applicable.

Since the sub-Laplacian on the Heisenberg group is degenerate elliptic, standard distance-based barriers make no sense in blow-up  analysis. While such barriers are indispensable in the Euclidean setting for constructing supersolutions, they fail in the present context 
because it is impossible to construct a strict supersolution along the characteristic directions relying only on the distance function. Consequently, 
adopting a strategy analogous to that used for fractional equations \cite{Jin Li Xiong} and inspired by Uguzzoni \cite{Uguzzoni Nodea}, we construct useful auxiliary functions  
based on cylindrically symmetric functions. This approach enables us to obtain the necessary control that the standard distance function cannot provide. (See proof of Lemma \ref{lemma 2.2 in jde}
for more details.)

Another difficulty arises from the non-commutativity of the derivative operators. Specifically, when deriving the local Pohozaev identity, the lack of commutativity generates 
extra terms that are difficult to handle in integral estimates. To overcome this, inspired by the work of Folland-Stein \cite{Follandsub}, we employ right-invariant vector fields to 
perform the integral estimates. A crucial property of these fields is that they commute with the standard left-invariant derivatives on the Heisenberg group. This commutativity 
allows us to eliminate these extra terms and successfully derive the crucial estimates. (See proof of Lemma \ref{lem 2.6 in jde} for more details.)

Besides these issues, the anisotropy of the Heisenberg group also poses significant challenges to the analysis. The distinct scaling properties of the horizontal and vertical directions necessitate a 
departure from classical Euclidean techniques. For instance, Taylor expansions must be performed with respect to the homogeneous degree rather than the standard algebraic 
degree. This structural difference significantly complicates the integral estimates, as the lack of full rotational symmetry introduces mixed terms that are difficult to control, 
posing a substantial challenge to our quantitative analysis.

The organization of the paper is as follows. In Section \ref{section 2}, we present some preliminaries and fix the notations regarding the Heisenberg group. 
In Section \ref{section 3}, we derive the Pohozaev identity and establish a B\"ocher-type theorem. In Section \ref{section 4}, we establish basic results concerning isolated simple 
blow-up points. Compared with previous works, several new ingredients are introduced to handle the geometric difficulties. In Section \ref{section 5}, we carry out 
the refined quantitative asymptotic analysis. In Section \ref{section 6}, we estimate the Pohozaev integral of blow-up solutions. Finally, the proofs of the main theorems are 
completed in Section \ref{section 7}.

\section{Preliminaries}\label{section 2}
	The Heisenberg group $\mathbb{H}^n$ is the set $\mathbb{R}^{2 n} \times \mathbb{R} $ endowed with the group action $\circ$ defined by
	\be\label{groupplus}
	\hat{\xi} \circ \xi:=(x+\hat{x}, y+\hat{y}, t+\hat{t}+2 \sum_{i=1}^n (x_i \hat{y}_i-y_i \hat{x}_i)),
	\ee
	for any $\xi=(x, y, t)$, $\hat{\xi}=(\hat{x}, \hat{y}, \hat{t})$ in ${\mathbb{H}^n}$, with $x=\left(x_1, \cdots, x_n\right)$, $\hat{x}=(\hat{x}_1, \cdots, \hat{x}_n)$, $y=\left(y_1, \cdots, y_n\right)$ and $\hat{y}=(\hat{y}_1, \cdots, \hat{y}_n)$ denoting elements of $\mathbb{R}^n$. We also use the notation $\xi=(z, t)$ with $z=x+i y$, $z \in \mathbb{C}^n \simeq \mathbb{R}^n \times \mathbb{R}^n$. Haar measure on $\mathbb{H}^n$ is the usual Lebesgue measure $\mathrm{d} \xi=\mathrm{d} z \mathrm{d} t$. 
	And $Q=2 n+2$ denotes the homogeneous dimension of ${\mathbb{H}^n}$ (see \cite{Follandstein hardyspace}).

	A basis for the Lie algebra of left-invariant vector fields on $\mathbb{H}^n$ is given by
	\be\label{left invariant vf}
	X_j=\frac{\partial}{\partial x_j}+2 y_j \frac{\partial}{\partial t}, \quad Y_j=\frac{\partial}{\partial y_j}-2 x_j \frac{\partial}{\partial t}, \quad  T=\frac{\partial}{\partial t}
	\ee
	for $j=1, \cdots, n$. From this, we obtain the following commutation relations for $j, k=1, \cdots, n$:
	$$
	\begin{gathered}
		{\left[X_j, X_k\right]=\left[Y_j, Y_k\right]=\left[X_j, \frac{\partial}{\partial t}\right]=\left[Y_j, \frac{\partial}{\partial t}\right]=0}, \quad 
		{\left[X_j, Y_k\right]=-4 \delta_{j k} T},
	\end{gathered}
	$$
where $\delta_{jk}$ denotes the Kronecker symbol. The Heisenberg gradient, or horizontal gradient, of a regular function $u$ is then defined by $\nabla_{\mathbb{H}} u:= \left(X_1 u, \ldots, X_n u, Y_1 u, \ldots, Y_n u\right)$. 

The homogeneous norm on ${\mathbb{H}^n}$ is defined by
	\be\label{groupnorm}
	\|\xi\|_{\mathbb{H}^n}:=\big(\big(\sum_{i=1}^n x_i^2+y_i^2 \big)^2+t^2 \big)^{\frac{1}{4}}=(|z|^4+t^2)^{\frac{1}{4}}.
	\ee
	Given $\xi_0 \in \mathbb{H}^n$, by \eqref{groupplus}, we have $\xi_0^{-1}=-\xi_0$. Then the corresponding distance on ${\mathbb{H}^n}$ is defined by
	$$
	d_{\mathbb{H}^n} (\xi, \hat{\xi}):=\|\hat{\xi}^{-1} \circ \xi\|_{\mathbb{H}^n}.
	$$
	For every $\xi \in {\mathbb{H}^n}$ and $R>0$, we define the following notations
	$$
	D_R(\xi):=\{\eta \in {\mathbb{H}^n} : d_{\mathbb{H}^n} (\xi, \eta)<R\},$$ $$ \partial D_R(\xi):=\{\eta \in {\mathbb{H}^n} : d_{\mathbb{H}^n} (\xi, \eta)=R\},$$
and call these sets respectively the Kor\'anyi ball and sphere centred at $\xi$ with radius $R$. For convenience, we also write $D_R(0):=D_R$.  

We also denote by $\tau_{\hat{\xi}}: {\mathbb{H}^n} \rightarrow {\mathbb{H}^n}$ the left translation by $\hat{\xi}$ on ${\mathbb{H}^n}$, defined by
	$$
	\tau_{\hat{\xi}}(\xi)=\hat{\xi} \circ \xi,
	$$
while for any $\lambda>0$ we will denote by $\delta_\lambda: {\mathbb{H}^n} \rightarrow {\mathbb{H}^n}$ the dilation defined by
	\begin{equation}\label{eq:dilation}
		\delta_\lambda(\xi):=(\lambda z, \lambda^2 t),
	\end{equation}
which satisfies
	$
	\delta_\lambda(\hat{\xi} \circ \xi)=\delta_\lambda(\hat{\xi}) \circ \delta_\lambda(\xi)
	$
for every $\xi, \hat{\xi} \in {\mathbb{H}^n}$ and every $\lambda>0$.

The sub-Laplacian on ${\mathbb{H}^n}$ is the differential operator
	$$
	\Delta_{{\mathbb{H}^n}}:=\sum_{j=1}^n \big(X_j^2+Y_j^2 \big),
	$$
where $X_j$ and $Y_j$ are defined in \eqref{left invariant vf}. 
For convenience, we can write
	\be\label{divergence form}
	\Delta_{\mathbb{H}^n} u=\operatorname{div}(A \nabla u),
	\ee
where 
	$$A=A(z):=\left(\begin{array}{ccc}
		\mathbb{I}_n & 0_n & 2 y \\
		0_n & \mathbb{I}_n & -2 x \\
		2 y & -2 x & 4|z|^2
	\end{array}\right),$$
and $\mathbb{I}_n$ denotes the $n\times n$ identity matrix.

In order to study analytical problems on the Heisenberg group, it is
natural to introduce function spaces adapted to its sub-Riemannian
structure. Following Folland--Stein \cite{Follandstein estimates complex}
and Jerison--Lee \cite{CR yamabe}, let
\[
Z_1=X_1,\cdots,Z_n=X_n,\qquad
Z_{n+1}=Y_1,\cdots,Z_{2n}=Y_n
\]
denote the standard left-invariant horizontal vector fields on
$\mathbb H^n$. For an ordered multi-index
\[
\mathcal I=(i_1,\cdots,i_m),\qquad i_j\in\{1,\cdots,2n\},
\]
we set
\[
|\mathcal I|=m,\qquad Z^{\mathcal I}=Z_{i_1}\cdots Z_{i_m},
\]
with $Z^\emptyset$ understood as the identity operator.

For $k\in\mathbb N_0$ and $1\leq p<\infty$, the Folland--Stein
Sobolev space $S^{k,p}(\mathbb H^n)$ is defined by
\[
S^{k,p}(\mathbb H^n)
=
\left\{
f\in L^p(\mathbb H^n):
Z^{\mathcal I} f\in L^p(\mathbb H^n)
\text{ for every } |\mathcal I|\leq k
\right\},
\]
where the derivatives are understood in the distributional sense,
and is equipped with the norm
\[
\|f\|_{S^{k,p}}
=
\sum_{|I|\leq k}\|Z^I f\|_{L^p}.
\]
Equivalently, for $1\leq p<\infty$, $S^{k,p}(\mathbb H^n)$ is the
completion of $C_c^\infty(\mathbb H^n)$ with respect to this norm.
For convenience, we write $S^k:=S^{k,2}$.

Analogously, using the Carnot--Carath\'eodory distance $d_{CC}$,
one defines the nonisotropic H\"older spaces
$\Gamma^{k,\alpha}(\mathbb H^n)$, $0<\alpha<1$; see
\cite{Follandstein estimates complex,Follandsub}.
We use $C^{k,\alpha}$ to denote the standard H\"older spaces.
The comparison between the nonisotropic and standard H\"older
spaces implies that sufficiently high-order
$\Gamma^{m,\alpha}$ estimates yield standard $C^{k,\alpha}$
estimates; see \cite[Theorem 2.3]{Afeltra}. Consequently, whenever
uniform $\Gamma^{m,\alpha}$ estimates are available for every
$m\in\mathbb N$, the corresponding compactness statements may
equivalently be formulated in the standard $C^{k,\alpha}$ topology.

In our proof, the classification of solutions to
\be\label{critical equation}
-\Delta_{\mathbb H^n}u=u^{p},\qquad p=\frac{Q+2}{Q-2}
\ee
is important. We use the following peak-normalized family of
Jerison--Lee bubbles. Let
\[
\bar c_n:=\frac{1}{(Q-2)^2}=\frac{1}{4n^2}
\]
and define
\be\label{talanti bubble}
\Lambda_{0,1}(z,t)
:=
\left[
\left(1+\bar c_n|z|^2\right)^2
+\bar c_n^2t^2
\right]^{\frac{2-Q}{4}}.
\ee
For $\xi_0\in\mathbb H^n$ and $\lambda>0$, set
\[
\Lambda_{\xi_0,\lambda}(\xi)
:=
\lambda^{\frac{Q-2}{2}}
\Lambda_{0,1}
\bigl(\delta_\lambda(\xi_0^{-1}\circ\xi)\bigr).
\]
Then
\[
-\Delta_{\mathbb H^n}\Lambda_{\xi_0,\lambda}
=
\Lambda_{\xi_0,\lambda}^{p}
\qquad\text{in }\mathbb H^n,
\]
and
\[
\Lambda_{0,1}(0)=1,
\qquad
\Lambda_{\xi_0,\lambda}(\xi_0)
=
\lambda^{\frac{Q-2}{2}}.
\]
Moreover,
\[
\|\Lambda_{\xi_0,\lambda}\|_{L^{Q^*}(\mathbb H^n)}
=
\|\Lambda_{0,1}\|_{L^{Q^*}(\mathbb H^n)},
\qquad
Q^*=\frac{2Q}{Q-2},
\]
which is a fixed positive constant depending only on $n$.

The classification of solutions to \eqref{critical equation} has been a central topic. While Jerison-Lee \cite{Jerison Lee JAMS} originally classified solutions 
in $L^{\frac{2Q}{Q-2}}(\mathbb{H}^n)$, recent works have significantly relaxed these energy assumptions. We summarize the key classification results 
relevant to our study as follows:

\begin{itemize}

\item \textbf{(Catino et al. \cite{Li Arxiv}):}
For $n=1$, every positive solution $u$ of \eqref{critical equation}
is a Jerison--Lee bubble.

\item \textbf{(Catino et al. \cite{Li Arxiv}):}
For $n\geq 2$, if a positive solution $u$ of \eqref{critical equation}
satisfies the decay estimate
\[
    u(\xi)\leq C(1+|\xi|_{\mathbb H^n})^{-\frac{Q-2}{2}},
\]
then $u$ is a Jerison--Lee bubble.

\item \textbf{(Flynn--V\'etois \cite{Flynn}):}
For $n\geq 2$, the classification holds under the weaker pointwise
condition
\[
    u(z,t)\leq C(|z|^2+|t|)^{-p}
    \quad \text{ for }p\geq \frac{n-2}{2}.
\]

\item \textbf{(Li-Liu-Ma \cite{LLM}, Li \cite{LiJungangArxiv}):}
For $n\geq 2$, every positive entire solution $u$ of
\eqref{critical equation} is a Jerison--Lee bubble, without any
a priori integrability, decay, boundedness, or symmetry assumption.

\end{itemize}

\section{Pohozaev identity and B\"ocher type theorem}\label{section 3}

	We denote by $\mathcal{X}$ the smooth vector fields
	\begin{align}\label{eq:vector field}
		\mathcal{X}:=\sum^n_{j=1}\Big( x_jX_j+y_jY_j\Big)+2tT.
	\end{align}
	We can observe that $\mathcal{X}$ is the generator of the group of dilations \eqref{eq:dilation} on $\mathbb{H}^n$. 

	Using this vector field, we can derive a Pohozaev type 
	identity as follows (see \cite{Ex and Nonex} for the proof).
	
	\begin{proposition}
		Let $\Omega \subset \mathbb{H}^{n}$ be a bounded, piecewise $C^1$ open set and let $u \in$ $C^2(\overline{\Omega})$. Then
		\begin{equation}\label{eq:pohozaev identity 1}
			\begin{gathered}
				2 \int_{\partial \Omega}(A \nabla u \cdot N) \mathcal{X} u \ \mathrm{d} \mathcal{H}_{Q-2}-\int_{\partial \Omega}|\nabla_{\mathbb{H}^n} u|^2 \mathcal{X} I \cdot N \ \mathrm{d} \mathcal{H}_{Q-2}\\=(2-Q) \int_\Omega|\nabla_{\mathbb{H}^{n}} u|^2\ \mathrm{d} z \mathrm{d} t+2 \int_\Omega \mathcal{X} u \Delta_{\mathbb{H}^{n}} u \ \mathrm{d} z \mathrm{d} t,
			\end{gathered}
		\end{equation}
		where $\mathrm{d}\mathcal{H}_{Q-2}$ denotes $(Q-2)$-dimensional Hausdorff measures on $\mathbb{H}^n$, $N$ is the outer unit normal to $\partial \Omega$.
	\end{proposition}
	Now let $u$ be a {$C^2$} positive solution of 
	\begin{equation}\label{eq:function of u}
		-\Delta_{\mathbb{H}^n}u=au+u^p \quad \text{ in }  D_R.
	\end{equation}
	Then multiplying \eqref{eq:function of u} by $u$ and integrating by parts, we have
	\begin{equation}\label{eq:3.4}
		-\int_{\partial D_R}A \nabla u\cdot Nu\ \mathrm{d}\mathcal{H}_{Q-2}+\int_{D_R}|\nabla_{\mathbb{H}^n}u|^2\ \mathrm{d}z\mathrm{d}t=\int_{D_R}( au^2+u^{p+1})\ \mathrm{d}z\mathrm{d}t.
	\end{equation} 
	Using \eqref{eq:pohozaev identity 1} and \eqref{eq:3.4}, we have
	\begin{equation*}
		\begin{aligned}
			& \quad2 \int_{\partial D_R}(A \nabla u \cdot N) \mathcal{X} u\  \mathrm{d} \mathcal{H}_{Q-2}-\int_{\partial D_R}|\nabla_{\mathbb{H}^{n}} u|^2 \mathcal{X} I\cdot N\  \mathrm{d} \mathcal{H}_{Q-2} \\
			&=(2-Q) \int_{\partial D_R}(A \nabla u \cdot N) u\ \mathrm{d} \mathcal{H}_{Q-2}+(2-Q) \int_{D_R} (au^2+u^{p+1})\ \mathrm{d} z \mathrm{d} t \\
			&\quad -2 \int_{D_R} \mathcal{X} u \big(au+u^{p}\big)\ \mathrm{d} z \mathrm{d} t .
		\end{aligned}
	\end{equation*}
	Using the definition of $\mathcal{X}$ and integrating the last term above by parts, we have
	\begin{equation}\label{eq:pohozaev ideneity 3}
		\begin{aligned}
			&\quad \int_{D_R} \mathcal{X} u \big(au+u^{p}\big)\ \mathrm{d} z \mathrm{d} t\\
			&=\int_{D_R} (x,y,2t) \cdot\nabla u \big(au+u^{p}\big)\ \mathrm{d} z \mathrm{d} t\\
			&=\frac{1}{2}\int_{\partial D_R} au^2 \mathcal{X}I\cdot N \ \mathrm{d} \mathcal{H}_{Q-2}+\frac{1}{p+1}\int_{\partial D_R} u^{p+1}\mathcal{X}I\cdot N \ \mathrm{d} \mathcal{H}_{Q-2}-\frac{Q}{2}\int_{D_R}au^2\ \mathrm{d}z\mathrm{d}t\\
			&\quad-\frac{Q}{p+1}\int_{D_R} u^{p+1}\ \mathrm{d}z\mathrm{d}t-\frac{1}{2}\int_{D_R}  \mathcal{X}(a) u^2\ \mathrm{d}z\mathrm{d}t.
		\end{aligned}
	\end{equation}
	Thus,
	\begin{equation}\label{eq:pohozaev ideneity 2}
		\begin{gathered}
			\frac{Q-2}{2} \int_{\partial D_R}(A \nabla u \cdot N) u\ \mathrm{d} \mathcal{H}_{Q-2}-\frac{1}{2} \int_{\partial D_R}|\nabla_{\mathbb{H}^{n}} u|^2 \mathcal{X}I \cdot N \ \mathrm{d} \mathcal{H}_{Q-2} \\
			+\int_{\partial D_R}(A \nabla u \cdot N) \mathcal{X} u\ \mathrm{d} \mathcal{H}_{Q-2} \\
			= \int_{D_R} au^2\ \mathrm{d} z \mathrm{d} t+\Big(\frac{Q}{p+1}-\frac{Q-2}{2}\Big) \int_{D_R} u^{p+1}\ \mathrm{d} z \mathrm{d} t +\frac{1}{2} \int_{D_R} \mathcal{X}(a) u^{2}\ \mathrm{d} z \mathrm{d} t \\ -\frac{1}{2} \int_{\partial D_R} au^2 \mathcal{X}I \cdot N\ \mathrm{d} \mathcal{H}_{Q-2} -\frac{1}{p+1} \int_{\partial D_R} u^{p+1} \mathcal{X}I \cdot N\ \mathrm{d} \mathcal{H}_{Q-2}.
		\end{gathered}
	\end{equation}
	Denote the boundary terms on the l.h.s. of \eqref{eq:pohozaev ideneity 2} by
	\begin{equation}\label{eq:B}
		\begin{aligned}
			\mathcal{D}( \xi, u, \nabla_{\mathbb{H}^{n}} u)=\frac{Q-2}{2} & (A \nabla u \cdot N) u-\frac{1}{2}\left|\nabla_{\mathbb{H}^{n}} u\right|^2 \mathcal{X}I \cdot N +(A \nabla u \cdot N) \mathcal{X} u.
		\end{aligned}
	\end{equation}
	Hence we have the following corollary.
	
	\begin{corollary}\label{pohozaev on heisenberg}
		If $u$ is a $C^2$, positive solution of \eqref{eq:function of u}, then
		\begin{equation}
			\begin{aligned}
				&\int_{\partial D_R}\mathcal{D}( \xi, u, \nabla_{\mathbb{H}^{n}} u)=\int_{D_R} au^2\ \mathrm{d} z \mathrm{d} t+\Big(\frac{Q}{p+1}-\frac{Q-2}{2}\Big) \int_{D_R} u^{p+1} \ \mathrm{d} z \mathrm{d} t\\
				& +\frac{1}{2} \int_{D_R} \mathcal{X}(a) u^{2}\ \mathrm{d} z \mathrm{d} t
				-\frac{1}{2} \int_{\partial D_R} au^2 \mathcal{X}I \cdot N\ \mathrm{d} \mathcal{H}_{Q-2} -\frac{1}{p+1} \int_{\partial D_R} u^{p+1} \mathcal{X}I \cdot N \ \mathrm{d} \mathcal{H}_{Q-2}.
			\end{aligned}
		\end{equation}
	\end{corollary}

	\begin{proposition}\label{pro 4.3 in asan}
		(i) For $u(\xi)=\|\xi\|_{\mathbb{H}^n}^{2-Q}$ and $\sigma>0$, we have
		$$
		\mathcal{D}(\xi, u, \nabla_{\mathbb{H}^n} u)=0 \quad \text { for } \xi \in \partial D_\sigma .
		$$
		(ii) If $u(\xi)=\|\xi\|_{\mathbb{H}^n}^{2-Q}+A+h(\xi)$, where $A>0$ is a positive constant and $h(\xi)$ is some function differentiable near the origin with $h(0)=0$, then there exists $\sigma_0>0$ such that for any $0<\sigma<\sigma_0$, we have
		$$
		\int_{\partial D_\sigma} \mathcal{D}(\xi, u, \nabla_{\mathbb{H}^n} u)<0.
		$$
		Furthermore,
		$$
		\lim _{\sigma \rightarrow 0} \int_{\partial D_\sigma} \mathcal{D}( \xi, u, \nabla_{\mathbb{H}^n} u)=-(Q-2)^2 A|D^{2 n-1}| \int_0^{\pi / 2} \cos ^{n-1} \alpha \ \mathrm{d} \alpha,
		$$
		where $|D^{2 n-1}|$ denotes the surface measure of the Kor\'anyi sphere.
	\end{proposition}
	
	\begin{proof}
		See \cite[Proposition 4.3]{A san} for more details.
	\end{proof}

	In the subsequent blow-up analysis, we need the B\"ocher type theorem for degenerate elliptic equations with isolated singularities.
	Since the sub-Laplacian on the Heisenberg group is a hypoelliptic operator, we can obtain the following results on $\mathbb{H}^n$ according to the Appendix in \cite{Li Zhu CCM} by using Bony's maximum principle in \cite{Bony} and $L^p$ estimates for the sub-Laplacian.
In the following, we provide some descriptions on singular behaviors of positive solutions to some linear elliptic equations in punctured balls. For nonnegative $\mathcal{L}$-harmonic functions in punctured open sets, one can see \cite{Bon Bocher} for more details.

				\begin{lemma}\label{lem 9.1 in ccm}
					Suppose $u \in C^2\left(D_1 \backslash\{0\}\right)$ satisfies
					\begin{equation}\label{eq:A1}
					-\Delta_{\mathbb{H}^n} u-a(\xi) u=0 \quad \text { in } D_1 \backslash\{0\},
					\end{equation}
					and $u(\xi)=o(\|\xi\|_{\mathbb{H}^n}^{2-Q})$ as $\|\xi\|_{\mathbb{H}^n} \rightarrow 0$, then $u \in C^{2, \alpha}\left(D_{1 / 2}\right)$ for any $0<\alpha<1$. 
				\end{lemma}
				
				\begin{proof}
					We first show that $-\Delta_{\mathbb{H}^n} u(\xi)-a(\xi) u(\xi)=0$ in $D_1$ in the distribution sense. For any $\varepsilon>0$, let $\zeta_\varepsilon $ be some cutoff function:
					$$
					\zeta_\varepsilon(\xi)= \begin{cases}1 & \text { for }\|\xi\|_{\mathbb{H}^n} \leq \varepsilon, \\ 0 & \text { for }\|\xi\|_{\mathbb{H}^n} \geq 2 \varepsilon, \\ \left|\nabla_{\mathbb{H}^n} \zeta_\varepsilon\right|<\frac{C}{\varepsilon}, & \left|\nabla_{\mathbb{H}^n}^2 \zeta_\varepsilon\right|<\frac{C}{\varepsilon^2} .\end{cases}
					$$
					Then for any $\phi \in C_c^{\infty}\left(D_1\right)$ we have
					\begin{equation}\label{eq:A2}
					-\int_{D_1} u\Delta_{\mathbb{H}^n} (\phi(1-\zeta_\varepsilon))-\int_{D_1} a u \phi(1-\zeta_\varepsilon)=0 .
					\end{equation}
					In fact, integrating \eqref{eq:A1} by $\phi(1-\zeta_\varepsilon)$, we have
					$$
					-\int_{D_1} (\phi(1-\zeta_\varepsilon))\Delta_{\mathbb{H}^n}u- \int_{D_1} a u \phi(1-\zeta_\varepsilon)=0.
					$$
					Given a domain $D_1 \backslash D_\delta$ with $0<\delta<\varepsilon$ such that $$ \begin{cases}
						\operatorname{supp} (\phi(1-\zeta_\varepsilon)) \subset D_1 \backslash D_\delta,\\
						\phi(1-\zeta_\varepsilon)=0 \quad \quad \quad \ \text{ on } \partial(D_1 \backslash D_\delta),\\
						\nabla_{\mathbb{H}^n} (\phi(1-\zeta_\varepsilon))=0 \quad \text{ on } \partial(D_1 \backslash D_\delta).
					\end{cases}    
					$$ 
					By the divergence theorem, we have
					$$
-\int_{D_1 \backslash D_\delta}\Delta_{\mathbb{H}^n}u (\phi(1-\zeta_\varepsilon))=\int_{D_1 \backslash D_\delta}\nabla_{\mathbb{H}^n} u\cdot \nabla_{\mathbb{H}^n} (\phi(1-\zeta_\varepsilon)) -\int_{\partial(D_1 \backslash D_\delta)} A\nabla u\cdot N(\phi(1-\zeta_\varepsilon)),
$$
$$
-\int_{D_1 \backslash D_\delta}u\Delta_{\mathbb{H}^n}(\phi(1-\zeta_\varepsilon)) =\int_{D_1 \backslash D_\delta}\nabla_{\mathbb{H}^n} u\cdot \nabla_{\mathbb{H}^n} (\phi(1-\zeta_\varepsilon)) -\int_{\partial(D_1 \backslash D_\delta)} A \nabla (\phi(1-\zeta_\varepsilon))\cdot Nu.
$$
Then, we obtain \eqref{eq:A2} after letting $\delta\rightarrow 0$.					
					
					Notice that
					$$\Delta_{\mathbb{H}^n} (fg)=g\Delta_{\mathbb{H}^n}f+f\Delta_{\mathbb{H}^n}g+2\langle \nabla_{\mathbb{H}^n}f, \nabla_{\mathbb{H}^n}g \rangle,$$
					then we have, 
					$$
\Delta_{\mathbb{H}^n} (\phi(1-\zeta_\varepsilon))=(1-\zeta_\varepsilon) \Delta_{\mathbb{H}^n} \phi-2\nabla_{\mathbb{H}^n} \phi\cdot \nabla_{\mathbb{H}^n} \zeta_\varepsilon-\phi \Delta_{\mathbb{H}^n} \zeta_\varepsilon.
$$
	Using \eqref{eq:A2}, it follows that
	$$
-\int_{D_1} u(1-\zeta_\varepsilon) \Delta_{\mathbb{H}^n} \phi  +2 \int_{D_1}u\nabla_{\mathbb{H}^n} \phi \cdot \nabla_{\mathbb{H}^n} \zeta_\varepsilon  +\int_{D_1} u\phi \Delta_{\mathbb{H}^n} \zeta_\varepsilon  -\int_{D_1} a u \phi(1-\zeta_\varepsilon) =0.
$$				
Therefore,
	$$
	\begin{aligned}
		-\int_{D_1} (\Delta_{\mathbb{H}^n} \phi) u-\int_{D_1} a u \phi &=-\int_{D_1} u\zeta_\varepsilon \Delta_{\mathbb{H}^n} \phi  -2 \int_{D_1}u\nabla_{\mathbb{H}^n} \phi \cdot\nabla_{\mathbb{H}^n} \zeta_\varepsilon \\
		&\quad  -\int_{D_1} u\phi \Delta_{\mathbb{H}^n} \zeta_\varepsilon  -\int_{D_1} a u \phi \zeta_\varepsilon .
\end{aligned}
$$				
			It follows that		
					$$
					\left|-\int_{D_1} \Delta_{\mathbb{H}^n} \phi u-\int_{D_1} a(\xi) u \phi\right| \leq C \epsilon^{-2} \int_{D_{2\varepsilon} \backslash D_{\varepsilon}}|u|+C \int_{D_{\varepsilon}}|u|=o(1)
					$$
as $\varepsilon$ tends to zero, where we have used $u(\xi)=o(\|\xi\|_{\mathbb{H}^n}^{2-Q})$ in the last step.

					We know from $u(\xi)=o(\|\xi\|_{\mathbb{H}^n}^{2-Q})$ that $u \in L_{\text {loc }}^s\left(D_1\right)$ for $s<\frac{Q}{Q-2}$. By $L^p$ estimates for the sub-Laplacian, we have $u \in S_{\text {loc }}^{2, s}\left(D_1\right)$ (see \cite{Follandstein estimates complex}). The lemma then follows from standard bootstrap methods and standard sub-elliptic estimate for the sub-Laplacian.
				\end{proof}

				\begin{lemma}\label{lem 9.2 in ccm}
					There exists some constant $\rho>0$ depending on $n$ and $\|a(\xi)\|_{L^{\infty}(D_1)}$, 
					such that the maximum principle holds for $\Delta_{\mathbb{H}^n}+a(\xi)$ on $D_{\rho}$ and there exists a unique $G(\xi) \in C^2\left(D_{\rho} \backslash\{0\}\right)$ satisfying
					\be\label{A.1}
					\begin{cases}
						-\Delta_{\mathbb{H}^n} G-a(\xi) G  =0 & \text { in } D_{\rho} \backslash\{0\}, \\
						G  =0 & \text { on } \partial D_{\rho}, \\
						\displaystyle \lim _{\|\xi\|_{\mathbb{H}^n} \rightarrow 0}\|\xi\|_{\mathbb{H}^n}^{Q-2} G(\xi) =1, 
					\end{cases}
					\ee
					Furthermore, $G(\xi)=\|\xi\|_{\mathbb{H}^n}^{2-Q}+R(\xi)$, where $R(\xi)$ satisfies for all $0<\epsilon<1$ that
					\begin{equation}\label{eq:K}
					\|\xi\|_{\mathbb{H}^n}^{Q-4+\epsilon}|R(\xi)|+\|\xi\|_{\mathbb{H}^n}^{Q-3+\epsilon}|\nabla_{\mathbb{H}^n} R(\xi)| \leq C(\epsilon)\quad \text{ for } \xi \in D_{\rho}
					\end{equation}
					and $C(\epsilon)$ is some constant depending only on $\epsilon$,  $n$, and $\|a\|_{L^{\infty}\left(D_1\right)}$.
				\end{lemma}

				\begin{proof}
					Clearly \eqref{A.1} is equivalent to
					\be\label{A.2}
					\begin{cases}
                    -\Delta_{\mathbb{H}^n} R-a(\xi) R=O(\|\xi\|_{\mathbb{H}^n}^{2-Q}) & \text { in } D_\rho, \\ R=-\rho^{2-Q} & \text { on } \partial D_\rho.
                    \end{cases}
					\ee

					Using the polar coordinates adapted to $\mathbb{H}^n$, it is easy to recognize that $\|\xi\|_{\mathbb{H}^n}^{2-Q} \in L^\tau\left(D_\rho\right)$ for all $\tau<Q /(Q-2)$. 
					By using the standard \(L^\tau\)-Dirichlet theory on sufficiently small balls, H\"older inequality and standard subelliptic estimates, \eqref{A.2} has a unique solution $R \in S^{2, \tau}\left(D_\rho\right)$ 
					and
					$$
					\|R\|_{L^s\left(D_\rho\right)} \leq \begin{cases}C(s) & \text{ for } s<\frac{Q}{Q-4}, n \geq 2, \\ C(s) & \text{ for } s<\infty, n=1.\end{cases}
					$$

					For $0<r \leq \rho / 5$, $\eta \in A_0=\left\{\eta: \frac{1}{5} \leq\|\eta\|_{\mathbb{H}^n}\leq 5\right\}$, 
					let
					$
					R_1(\eta)=r^{Q-2} R(\delta_r \eta).  
					$
					Then $R_1$ satisfies
					$$
					-\Delta_{\mathbb{H}^n} R_1(\eta)-a(\delta_r \eta) r^2 R_1(\eta)=O\left(r^2\right) \quad \text{ for } \eta \in A_0,
					$$
					where $\left|O(r^2)\right| \leq C r^2$ with $C$ independent of $r$. For $n \geq 2$, for all $0<\varepsilon<1$, we can choose some $s_1=$ $s_1(\varepsilon)<Q /(Q-4)$ such that
					$
					\left\|R_1\right\|_{L^{s_1}\left(A_0\right)} \leq r^{2-\varepsilon}\|R\|_{L^{s_1}\left(D_\rho\right)} \leq C(\varepsilon) r^{2-\varepsilon} .
					$

					Using $L^p$ estimates for the sub-Laplacian, Folland-Sobolev embedding and the bootstrap method finite times, we have
					$$
					\left|R_1(\eta)\right|+\left|\nabla_{\mathbb{H}^n} R_1(\eta)\right| \leq C(\varepsilon) r^{2-\varepsilon} \quad \text{ for }\frac{1}{2} \leq \|\eta\|_{\mathbb{H}^n} \leq 2,
					$$	
					which means
					$$
					\|\xi\|_{\mathbb{H}^n}^{Q-2}|R(\xi)|+\|\xi\|_{\mathbb{H}^n}^{Q-1}|\nabla_{\mathbb{H}^n} R(\xi)| \leq C(\varepsilon)\|\xi\|_{\mathbb{H}^n}^{2-\varepsilon} \quad \text{ for } \|\xi\|_{\mathbb{H}^n} \leq \frac{\rho}{5} .
					$$
                    
                    For $n=1$, choose any $s_1>\frac{4}{\varepsilon}$.
                    When \(n=1\), we have \(Q=4\). By the preceding estimate,
                    \[
                    \|R\|_{L^s(D_\rho)}\leq C(s)
                    \qquad\text{for every }1<s<\infty.
                    \]
                    For any \(0<\varepsilon<1\), choose
                    \[
                    s_1=s_1(\varepsilon)>\frac{4}{\varepsilon}.
                    \]
                    Since
                    \[
                    R_1(\eta)=r^{Q-2}R(\delta_r\eta)=r^2R(\delta_r\eta),
                    \]
                    a change of variables gives
                \[
                \begin{aligned}
                \|R_1\|_{L^{s_1}(A_0)}=
                r^{2-\frac{4}{s_1}}
                \|R\|_{L^{s_1}(\delta_rA_0)}  \leq
                C(s_1)r^{2-\frac{4}{s_1}}
                \leq
                C(\varepsilon)r^{2-\varepsilon}.
                \end{aligned}
                \]

                Moreover, since \(s_1>4=Q\), the interior \(L^{s_1}\)-estimates
                for the sub-Laplacian and the Folland--Stein embedding theorem still
                yield
                \[
                |R_1(\eta)|
                +
                |\nabla_{\mathbb H^n}R_1(\eta)|
                \leq
                C(\varepsilon)r^{2-\varepsilon}
                \]
                for $\frac12\leq\|\eta\|_{\mathbb H^n}\leq 2$.
				\end{proof}

				\begin{lemma}\label{lem 9.3 in ccm}
					Assume $u(\xi) \in C^2(D_1 \backslash\{0\})$ satisfies
					\be\label{66 in ccm}
					-\Delta_{\mathbb{H}^n} u-a(\xi) u=0, \quad u>0 \quad \text { in } D_1 \backslash\{0\},
					\ee
					then
					$$
					\alpha:=\varlimsup_{\rho \rightarrow 0} \max _{\|\xi\|_{\mathbb{H}^n}=\rho} u(\xi)\|\xi\|_{\mathbb{H}^n}^{Q-2}<+\infty .
					$$
				\end{lemma}
				
				\begin{proof}
					It follows from the Harnack inequality that there exists $C>0$ such that for $0<\rho<1$,
					$$
					\max _{\|\xi\|_{\mathbb{H}^n}=\rho} u(\xi) \leq C \min _{\|\xi\|_{\mathbb{H}^n}=\rho} u(\xi) .
					$$
					
					On the contrary, suppose that $\alpha=+\infty$. Therefore for all $A>0$, there exists $\rho_i \rightarrow 0^{+}$ such that
					$$
					u(\xi)>A\|\xi\|_{\mathbb{H}^n}^{2-Q} \quad \text{ for }\|\xi\|_{\mathbb{H}^n}=\rho_i .
					$$
					Set $v_A(\xi)=\frac{A}{2} G(\xi)$, where $G(\xi)$ was defined in Lemma \ref{lem 9.2 in ccm}. Using 
					$$\lim _{\|\xi\|_{\mathbb{H}^n} \rightarrow 0}\|\xi\|_{\mathbb{H}^n}^{Q-2} G(\xi) =1,
					$$
					 we have $u(\xi) \geq v_A(\xi)$ on $\|\xi\|_{\mathbb{H}^n} = \rho_i$.

					Moreover, $G(\xi) = 0$ on the boundary $\|\xi\|_{\mathbb{H}^n} = \rho$, hence $v_A(\xi) = 0$ on the boundary $\|\xi\|_{\mathbb{H}^n} = \rho$. Since $u(\xi) > 0$, it follows that $u(\xi) \geq v_A(\xi)$ on $\|\xi\|_{\mathbb{H}^n} = \rho$. Thus, it follows from the maximum principle that for large $i$,
					$$
					u(\xi) \geq v_A(\xi) \quad \text{ for }\rho_i \leq\|\xi\|_{\mathbb{H}^n} \leq \rho .
					$$
					Sending $i$ to $+\infty$, we have
					$$
					u(\xi) \geq v_A(\xi)=\frac{A}{2} G(\xi) \quad \text{ for } 0<\|\xi\|_{\mathbb{H}^n}<\rho ,
					$$	
					which contradicts the fact that $u\in C^2(D_1 \backslash\{0\})$ as $A\rightarrow +\infty$.
				\end{proof}
				
				\begin{proposition}\label{bocher type}
					Suppose $u \in C^2(D_1 \backslash\{0\})$ satisfies
					$$
					-\Delta_{\mathbb{H}^n} u(\xi)-a(\xi) u(\xi)=0, \quad u(\xi)>0 \quad \text { in } D_1 \backslash\{0\} .
					$$
					Then there exists some constant $b \geq 0$ such that
					$$
					u(\xi)=b G(\xi)+E(\xi) \quad \text { in } D_{\rho} \backslash\{0\},
					$$
					where $G(\xi)$, $\rho$ are defined in Lemma \ref{lem 9.2 in ccm}, and $E(\xi) \in C^2\left(D_\rho\right)$ satisfies
					$$
					-\Delta_{\mathbb{H}^n} E(\xi)-a(\xi) E(\xi)=0 \quad \text { in } D_\rho .
					$$
				\end{proposition}
				
				\begin{proof}
					Set
					$$
					b=b(u)=\sup \left\{\lambda \geq 0 \mid \lambda G \leq u \text { in } D_{\rho} \backslash\{0\}\right\} .
					$$
					Combining with Lemma \ref{lem 9.3 in ccm}, we know from the previous definition  that $0 \leq b \leq \alpha<\infty$.
					
					Case 1: $b=0$.\\
					In this case we claim: for any $\varepsilon>0$, there exists $\rho_\varepsilon \in\left(0, \rho\right)$ such that
					$$
					\min _{\|\xi\|_{\mathbb{H}^n}=r}\{u(\xi)-\epsilon G(\xi)\} \leq 0 \quad \text{ for } 0<r<\rho_\varepsilon .
					$$

					If the above claim were false, then there would exist some $\epsilon_0>0$ and $r_j \rightarrow 0^{+}$ such that
					$$
					\min _{\|\xi\|_{\mathbb{H}^n}=r_j}\left\{u(\xi)-\varepsilon_0 G(\xi)\right\}>0 .
					$$
					Notice that $u(\xi)-\varepsilon_0 G(\xi) \geq 0$ for $\|\xi\|_{\mathbb{H}^n}=\rho$. We derive from the maximum principle that $u(\xi)-\varepsilon_0 G(\xi) \geq 0$ on $D_{\rho} \backslash D_{r_j}$ which means $u(\xi)-\varepsilon_0 G(\xi) \geq 0$ in $D_{\rho} \backslash\{0\}$. From the definition of $b(u)$, we know  that $b \geq \varepsilon_0>0$, a contradiction.
					
					Therefore, for any $\varepsilon>0$, and $0<r<\rho_\varepsilon$, there exists $\xi_\varepsilon$ with $\|\xi_\varepsilon\|_{\mathbb{H}^n}=r$, such that $u\left(\xi_\varepsilon\right) \leq \varepsilon G\left(\xi_\varepsilon\right)$. By the Harnack inequality, we have
					$$
					\max _{\|\xi\|_{\mathbb{H}^n}=r} u(\xi) \leq C u\left(\xi_\varepsilon\right) \leq C \varepsilon G\left(\xi_\varepsilon\right) .
					$$
					It follows that
					$$
					u(\xi)=o(\|\xi\|_{\mathbb{H}^n}^{2-Q}) \quad \text { as }\|\xi\|_{\mathbb{H}^n} \rightarrow 0 .
					$$
					Setting $E(\xi)=u(\xi)$, our result in this case follows from Lemma \ref{lem 9.1 in ccm}.
					
					Case 2: $b>0$.\\
					We consider $v(\xi)=u(\xi)-b G(\xi)$. From the definition of $b(u)$, we know that $v(\xi) \geq 0$. By the maximum principle (see \cite{Bony}) , 
					we know that either $v(\xi)=0$ or $v(\xi)>0$ in $D_{\rho} \backslash\{0\}$. In the former case we are done by choosing $E(\xi)=0$. In the latter case, $v(\xi)$ satisfies \eqref{66 in ccm}. Set
					$$
					b(v)=\sup \left\{\mu \geq 0 \mid \mu G \leq v  \text { in } D_{\rho} \backslash\{0\}\right\} .
					$$
					It is easy to see that $b(v)=0$. As in Case 1, we know $v(\xi)=o(\|\xi\|_{\mathbb{H}^n}^{2-Q})$. Letting $E(\xi)=v(\xi)$, our result in this case follows from Lemma \ref{lem 9.1 in ccm}.
				\end{proof}

\section{Analysis of Isolated Blow Up Points}\label{section 4}

	Let $\tau_i\geq 0$ satisfy $\lim_{i\rightarrow \infty}\tau_i=0$, $p_i=(Q+2)/(Q-2)-\tau_i$, and $a_i\geq 0$ be a sequence of functions converging to $a$ in  $C^2(D_3)$.
	Let $u_i>0$ be a sequence of $C^2(D_3)$ solutions of
	\be\label{equation2}
		-\Delta_{\mathbb{H}^{n}} u_i -a_i u_i=  u_i^{p_i}\quad \text { in } D_3.
	\ee

	\begin{definition}\label{isolated blow-up point}
		Suppose that $\left\{u_i\right\}$ satisfies \eqref{equation2}. We say a point $\bar{\xi} \in D_2$ is a blow-up point of $\{u_i\}$ if $u_i(\xi_i) \rightarrow \infty$ for some $\xi_i \rightarrow \bar{\xi}$. 
		And a point $\bar{\xi} \in D_2$ is called an isolated blow-up point of $\left\{u_i\right\}$ if there exists $0<\bar{r}<\operatorname{dist}_{\mathbb{H}^n}(\bar{\xi}, \partial D_3)$, $C>0$ and a sequence $\xi_i \rightarrow \bar{\xi}$ such that $\xi_i$ is a local maximum of $u_i$, $u_i\left(\xi_i\right) \rightarrow \infty$ and
		$$
		u_i(\xi) \leq C d_{\mathbb{H}^n}(\xi, \xi_i)^{-2 /(p_i-1)} \quad \text { for }  \xi \in D_{\bar{r}} (\xi_i).
		$$
		
	\end{definition}
	
The usual Euclidean surface measures on the Kor\'anyi spheres are not
homogeneous under the anisotropic Heisenberg dilations. We therefore use
the weighted spherical mean introduced by Garofalo and Lanconelli, which is
adapted to the sub-Laplacian and has the correct homogeneity under left
translations and Heisenberg dilations.

For $\xi_0=(z_0,t_0)\in\mathbb H^n$, let
\[
\psi_{\xi_0}(\xi)
:=|\nabla_{\mathbb H^n}d_{\mathbb H^n}(\xi,\xi_0)|^2
=\frac{|z-z_0|^2}{d_{\mathbb H^n}(\xi,\xi_0)^2}
\quad \text{ for }\xi\neq\xi_0.
\]
Here and below, we define the normalizing constant
\[
\begin{aligned}
\omega_Q:=
\int_{\partial D_1}
\frac{\psi_0(\theta)}
{|\nabla d_{\mathbb H^n}(\theta,0)|}
\,\mathrm d\mathcal H_{Q-2}(\theta)=
\int_{\partial D_1}
\frac{|z|^2}
{|\nabla\|\theta\|_{\mathbb H^n}|}
\,\mathrm d\mathcal H_{Q-2}(\theta),
\end{aligned}
\]
where $\theta=(z,t)\in\partial D_1$.
For a continuous function $v$ on $\partial D_r(\xi_0)$, define
\be\label{Garofalo weighted spherical mean}
\begin{aligned}
\mathcal M_{\xi_0}(v,r)
&:=
\frac{1}{\omega_Qr^{Q-1}}
\int_{\partial D_r(\xi_0)}
 v(\xi)
 \frac{\psi_{\xi_0}(\xi)}{|\nabla d_{\mathbb H^n}(\xi,\xi_0)|}
 \,\mathrm d\mathcal H_{Q-2}(\xi).
\end{aligned}
\ee

\begin{lemma}\label{properties of Garofalo mean}
Let $\mathcal M_{\xi_0}$ be defined as above.

\begin{enumerate}
\item If $h\in C^2(D_R(\xi_0))$ satisfies
\[
\Delta_{\mathbb H^n}h=0
\qquad\text{in }D_R(\xi_0),
\]
then
\[
\mathcal M_{\xi_0}(h,r)=h(\xi_0),
\qquad 0<r<R.
\]

\item If
\[
v_\mu(\xi)
=
\mu^{\alpha}
v(\xi_0\circ\delta_\mu\xi),
\]
then, whenever $\mu r<R$,
\[
\mathcal M_0(v_\mu,r)
=
\mu^{\alpha}
\mathcal M_{\xi_0}(v,\mu r).
\]
\end{enumerate}
\end{lemma}

\begin{proof}
The first assertion follows directly from the weighted mean-value
formula of Garofalo-Lanconelli
\cite[Theorem 2.1]{GL}.
For the second assertion, by the coarea formula and the homogeneity of the Heisenberg dilation, under the change of variables
$
\xi=\xi_0\circ\delta_\mu\eta,
$
we have
$$
\begin{aligned}
&\int_{\partial D_{\mu r}(\xi_0)}
v(\xi)
\frac{\psi_{\xi_0}(\xi)}
{|\nabla d_{\mathbb H^n}(\xi,\xi_0)|}
\mathrm d\mathcal H_{Q-2}(\xi)\\
&\qquad=
\mu^{Q-1}
\int_{\partial D_r}
v(\xi_0\circ\delta_\mu\eta)
\frac{\psi_0(\eta)}
{|\nabla d_{\mathbb H^n}(\eta,0)|}
\mathrm d\mathcal H_{Q-2}(\eta).
\end{aligned}
$$
Therefore,
$$
\begin{aligned}
\mathcal M_0(v_\mu,r)
&=
\frac{\mu^\alpha}{\omega_Qr^{Q-1}}
\int_{\partial D_r}
v(\xi_0\circ\delta_\mu\eta)
\frac{\psi_0(\eta)}
{|\nabla d_{\mathbb H^n}(\eta,0)|}
\mathrm d\mathcal H_{Q-2}(\eta)\\
&=
\frac{\mu^{\alpha+1-Q}}{\omega_Qr^{Q-1}}
\int_{\partial D_{\mu r}(\xi_0)}
v(\xi)
\frac{\psi_{\xi_0}(\xi)}
{|\nabla d_{\mathbb H^n}(\xi,\xi_0)|}
\mathrm d\mathcal H_{Q-2}(\xi)\\
&=
\mu^\alpha\mathcal M_{\xi_0}(v,\mu r).
\end{aligned}
$$

\end{proof}



Let $\xi_i\to\bar\xi$ be an isolated blow-up point of $u_i$. For
$0<r<\bar r$, define
\be\label{weighted averaged profiles}
\bar u_i^{\,G}(r):=\mathcal M_{\xi_i}(u_i,r),
\qquad
\bar w_i^{\,G}(r):=r^{\frac{2}{p_i-1}}\bar u_i^{\,G}(r).
\ee

\begin{definition}\label{isolated simple blow-up point}
We say that $\xi_i\to\bar\xi$ is an isolated simple blow-up point if it
is an isolated blow-up point and there exists $\rho\in(0,\bar r)$,
independent of $i$, such that $\bar w_i^{\,G}$ has precisely one critical point
in $(0,\rho)$ for all sufficiently large $i$.
\end{definition}


	If $\xi_i \rightarrow 0$ is an isolated blow-up point, then we will have the following Harnack inequality in the annulus centered at $0$.

	\begin{lemma}\label{Harnack ineq 2}
		Suppose that $u_i$ satisfies \eqref{equation2} and $\xi_i \rightarrow 0$ is an isolated blow-up point of $u_i$, that is, for some positive constants $A_1$ and $\bar{r}$ independent of $i$,
		\begin{align}\label{eq:5.21 as a san}
			d_{\mathbb{H}^n}(\xi,\xi_i)^{2  /(p_i-1)} u_i(\xi) \leq A_1 \quad \text { for } \xi \in D_{\bar{r}}(\xi_i) \subset D_3 .
		\end{align}
		Then for any $0<r<\bar{r} / 3$, we have the following Harnack inequality:
		$$
		\sup _{D_{2 r}(\xi_i) \backslash \overline{D_{r / 2}(\xi_i)}} u_i \leq C \inf _{D_{2 r}(\xi_i) \backslash \overline{D_{r / 2}(\xi_i)}} u_i,
		$$
		where $C>0$ depends only on $n$, $A_1$, $\bar{r}$, and $\sup_i\|a_i\|_{L^{\infty}(D_{\bar{r}}(\xi_i))}$.
	\end{lemma}
	
	\begin{proof}
		For $0<r<\bar{r} / 3$, define
		$
			w_i(\xi):=r^{2  /(p_i-1)} u_i(\xi_i \circ \delta_r\xi).
		$
		By the equation of $u_i$, we have
		\begin{equation*}
			-\Delta_{\mathbb{H}^{n}} w_i(\xi)= r^2a_i(\xi_i \circ \delta_r\xi)w_i(\xi)+w_i(\xi)^{p_i}\quad \text { in } D_3.
		\end{equation*}
		Since $\xi_i \rightarrow 0$ is an isolated blow-up point of $u_i$, we have
		$w_i(\xi) \leq A_1\|\xi\|_{\mathbb{H}^n}^{-2  /(p_i-1)}$ for $\xi \in D_3$.
		Applying  the Harnack inequality in \cite{Citti} in the annulus $D_{9/4}\backslash D_{1/4}\subset D_3$, we have
		$$
 \text{ ess } \sup _{D_{9 / 4} \backslash D_{1 / 4}} w_i \leq C\text { ess } \inf _{D_{9 / 4} \backslash D_{1 / 4}} w_i ,
$$
where $C=C(n, A_1, \bar{r}, \sup \| a_i \|_{L^\infty (D_{\bar{r}}(\xi_i))})$. And the constant $C$ is independent of $i$ and $r$. Then, the lemma follows after rescaling back to $u_i$.
	\end{proof}

	\begin{proposition}\label{solution is bubble}
		Suppose that the hypotheses of Lemma \ref{Harnack ineq 2} hold. Suppose also that $\|a_i\|_{C^2(D_3)} \leq A_0$. 
		Then for any $R_i \rightarrow \infty$ and $\varepsilon_i \rightarrow 0^{+}$, we have, after passing to a subsequence (still denoted as $u_i,\xi_i$, etc.), that
		
		\be\label{equivalent}
		\|u_i(\xi_i)^{-1} u_i(\xi_i \circ \delta_{u_i(\xi_i)^{-(p_i-1) / 2}\ } \cdot)-\Lambda_{0,1}(\cdot)\|_{\Gamma^{2,\alpha}(D_{2 R_i})} \leq \varepsilon_i,
		\ee
		\be\label{19 of rnac}
		r_i:= R_i u_i(\xi_i)^{-(p_i-1) / 2} \rightarrow 0 \quad \text { as }  i \rightarrow \infty,
		\ee
		where $\Lambda$ is defined as in \eqref{talanti bubble}.
	\end{proposition}

	\begin{proof}
		Define
		$$
		w_i(\xi):=u_i(\xi_i)^{-1} u_i(\xi_i \circ \delta_{u_i(\xi_i)^{-(p_i-1) / 2}} \xi) \quad \text { for }\|\xi\|_{\mathbb{H}^n}<\bar{r} u_i(\xi_i)^{\left(p_i-1\right) / 2 }.
		$$
		It satisfies the equations
		\begin{equation}	\label{eq:5.7 in asan}
		\begin{aligned}
		-\Delta_{\mathbb{H}^{n}}w_i(\xi)  =u_i(\xi_i)^{1-p_i}a_i(\xi_i \circ \delta_{u_i(\xi_i)^{-(p_i-1) / 2}} \xi)w_i(\xi)+w_i(\xi)^{p_i}\\ \text { for }\|\xi\|_{\mathbb{H}^n}<\bar{r} u_i(\xi_i)^{\left(p_i-1\right) / 2 } ,
		\end{aligned}
		\end{equation}
		\be\label{eq:5.8 5.9 in asan}
		w_i(0)  =1,\quad \nabla_{\mathbb{H}^n} w_i(0)  =0 , {\quad \partial_t w_i(0)=0},
		\ee
		\be\label{def of isolated}
		0<w_i(\xi)  <\bar{C}\|\xi\|_{\mathbb{H}^n}^{-2 /(p_i-1)} \quad \text { for }\|\xi\|_{\mathbb{H}^n}<\bar{r} u_i(\xi_i)^{\left(p_i-1\right) / 2 } ,
		\ee
	where \eqref{def of isolated} follows from the definition of isolated blow-up point.

		Now, for every fixed \(R>0\), we claim that
\[
\sup_{D_R}w_i\leq C_R
\]
for all sufficiently large \(i\).

Indeed, by Lemma \ref{Harnack ineq 2}, for every
\(0<r<1\) and all sufficiently large \(i\), we have
\[
\begin{aligned}
\max_{\partial D_r} w_i
\leq
\sup_{D_{2r}\setminus\overline{D_{r/2}}}w_i \leq
C_0\inf_{D_{2r}\setminus\overline{D_{r/2}}}w_i \leq
C_0\min_{\partial D_r}w_i,
\end{aligned}
\]
where \(C>0\) is independent of \(i\) and \(r\).

On the other hand, since \(a_i\geq0\) and \(w_i>0\),
\[
-\Delta_{\mathbb H^n}w_i
=
u_i(\xi_i)^{1-p_i}
a_i(
\xi_i\circ
\delta_{u_i(\xi_i)^{-(p_i-1)/2}}\xi
)w_i
+
w_i^{p_i}
\geq0.
\]
Hence, by the maximum principle,
\[
\min_{\partial D_r}w_i\leq C_1\min_{ D_r}w_i\leq
C_1 w_i(0)=C_1.
\]
Consequently,
\[
\max_{\partial D_r}w_i\leq C
\qquad\text{for every }0<r<1.
\]
It follows that
\[
\sup_{D_1}w_i\leq C.
\]
for any $i$ and some $C$. 
For \(1\leq\|\xi\|_{\mathbb H^n}\leq R\), the isolated blow-up
estimate gives
\[
w_i(\xi)
\leq
\bar{C}
\|\xi\|_{\mathbb H^n}^{-2/(p_i-1)}
\leq C.
\]
Therefore,
\[
\sup_{D_R}w_i\leq C_R.
\] 
By the interior subelliptic estimates and a standard bootstrap
argument, for every fixed \(R>0\),
\[
\|w_i\|_{\Gamma^{2,\alpha}(D_R)}\leq C_R.
\]
Therefore, after passing to a subsequence and using a diagonal
argument,
\[
w_i\rightarrow w
\qquad\text{in }\Gamma_{\mathrm{loc}}^{2,\alpha}(\mathbb H^n).
\]

To continue the proof, we need the following lemma. 
		
		\begin{lemma}\label{w's converges}
			There exists a subsequence of $\{w_i\}$ which converges in $C_{\text {loc }}^2(\mathbb{H}^{n})$ to a positive function $w$.
		\end{lemma}
		
		\begin{proof}
			See \cite[Proposition 4.5]{Afeltra} for proof.
		\end{proof}

		It is easy to see that $w$ satisfies
		$
		-\Delta_{\mathbb{H}^{n}}w(\xi)=w(\xi)^p.
		$
		By 
		Li-Liu-Ma \cite[Corollary 1.3]{LLM} or Li \cite[theorem 1.2]{LiJungangArxiv}, we obtain that $w$ is a Jerison-Lee bubble.

 Now, given $R_i \rightarrow \infty$ and $\varepsilon_i \rightarrow 0^{+}$, one can always choose a subsequence of $\{w_i\}$, such that \eqref{equivalent} holds and the proposition follows. 
	\end{proof}

	\begin{remark}
		Note that since passing to subsequences does not affect our proofs, in the rest of the paper we will always choose $R_i \rightarrow \infty$ with $R_i^{\tau_i}=1+o(1)$ first, and then $\varepsilon_i \rightarrow 0^{+}$as small as we wish (depending on $R_i$) and then choose our subsequence $\{u_i\}$ to work with. 
	\end{remark}

	\begin{proposition}\label{2.2 in jde}
		Under the hypotheses of Lemma \ref{Harnack ineq 2}, there exists some positive constant $C=C(n, A_0, A_1)$, such that,
		\begin{align*}
			u_i(\xi) \geq C u_i(\xi_i) \Lambda_{0,1}(\delta_{u_i(\xi_i)^{(p_i-1) / 2}} (\xi_i^{-1}\circ \xi)) 
		\end{align*}
		for all $d_{\mathbb{H}^n}(\xi, \xi_i) \leq 1$. In particular, for any $e \in \mathbb{H}^n$ with $\|e\|_{\mathbb{H}^n}=1$, we have
		$$
		u_i(\xi_i \circ e) \geq C^{-1} u_i(\xi_i)^{-1+((Q-2) / 2) \tau_i}.
		$$
	\end{proposition}
	
	\begin{proof}
		Let us denote $r_i=R_i u_i(\xi_i)^{-(p_i-1) / 2}$. It follows from Proposition \ref{solution is bubble} that for all $d_{\mathbb{H}^n}(\xi, \xi_i)=r_i$,
		$$
		\begin{aligned}
			u_i(\xi) & \geq C^{-1} u_i(\xi_i)(u_i(\xi_i)^{2(p_i-1)}(t-\hat{t}+2 \sum_{i=1}^n (y_i \hat{x}_i-x_i \hat{y}_i))^2\\
			&\quad +(1+u_i(\xi_i)^{p_i-1} |z-\hat{z}|^2)^2)^{(2-Q) / 4} , \\
			&=C^{-1}u_i(\xi_i)(u_i(\xi_i)^{2(p_i-1)}r_i^4+1+2u_i(\xi_i)^{p_i-1} |z-\hat{z}|^2)^{(2-Q) / 4}\\
			& =C^{-1}u_i(\xi_i)(R_i^4+1+2u_i(\xi_i)^{p_i-1} |z-\hat{z}|^2)^{(2-Q) / 4}\\
			& \geq C^{-1} u_i(\xi_i) R_i^{2-Q},
		\end{aligned}
		$$
		where $\xi=(z, t)=\left(x_1, \cdots, x_n, y_1, \cdots, y_n, t\right)$ and $\xi_i=(\hat{z}, \hat{t})=(\hat{x}_1, \cdots, \hat{x}_n, \hat{y}_1, \cdots, \hat{y}_n, \hat{t})$.
		Set
$$
\varphi_i(\xi)=C^{-1} R_i^{2-Q} r_i^{Q-2} u_i(\xi_i)(d_{\mathbb{H}^n}(\xi,\xi_i)^{2-Q}-(\frac{3}{2})^{2-Q}) \quad \text{ for } r_i \leq d_{\mathbb{H}^n}(\xi,\xi_i)\leq  \frac{3}{2}.
$$
Clearly, $\varphi_i$ satisfies
$$
\begin{cases}
\Delta_{\mathbb{H}^n} \varphi_i(\xi)  =0 \geq \Delta_{\mathbb{H}^n} u_i(\xi)  & \text{ for } r_i \leq d_{\mathbb{H}^n}(\xi,\xi_i) \leq \frac{3}{2}, \\
\varphi_i(\xi)  =0 \leq u_i(\xi)  & \text{ for } d_{\mathbb{H}^n}(\xi,\xi_i)=\frac{3}{2}, \\
\varphi_i(\xi)  \leq u_i(\xi)  & \text{ for }d_{\mathbb{H}^n}(\xi,\xi_i)=r_i .
\end{cases}
$$
It follows from Bony's maximum principle \cite{Bony} that

$$
u_i(\xi) \geq \varphi_i(\xi) \quad \text { for }  r_i \leq d_{\mathbb{H}^n}(\xi,\xi_i) \leq \frac{3}{2}.
$$
Hence the proposition follows immediately from above and Proposition \ref{solution is bubble}.
	\end{proof}
	
	\begin{lemma}\label{lemma 2.2 in jde}
		Under the hypotheses of Lemma \ref{Harnack ineq 2}, and assume in addition that $\xi_i \rightarrow 0$ is also an isolated simple blow-up point with the constant $\rho$. Then, for any fixed $0<\delta\leq \frac{Q-2}{2}$ sufficiently small, we have
		\be\label{5.23 in asan}
		u_i(\xi) \leq C u_i(\xi_i)^{-\lambda_i} d_{\mathbb{H}^n}(\xi,\xi_i)^{2-Q+\delta} \quad \text { for } r_i \leq d_{\mathbb{H}^n}(\xi,\xi_i) \leq 1,
		\ee
		where $\lambda_i=(Q-2 -\delta)(p_i-1) / 2 -1$ and $C>0$ depends only on $n$, $A_0$, $A_1$, $\delta$ and $\rho$.
	\end{lemma}

	\begin{proof}
		Note that
		$$
		\Lambda_{0,1}(\xi)=(1+2 \bar c_n|z|^2+\bar c_n^2(|z|^4+t^2))^{(2-Q) / 4} \leq C\|\xi\|_{\mathbb{H}^n}^{2-Q}
		$$
		and from Proposition \ref{solution is bubble}, we have
		\begin{align}\label{5.24 in sasan}
			u_i(\xi) & \leq C u_i(\xi_i) \Lambda(u_i(\xi_i)^{(p_i-1) / 2}(\xi_i^{-1} \circ \xi))\nonumber \\
			& \leq C u_i(\xi_i)(u_i(\xi_i)^{(p_i-1) / 2} r_i)^{2-Q}\nonumber \\
			& =C u_i(\xi_i) R_i^{2-Q}\quad \text{ for }d_{\mathbb{H}^n} (\xi,\xi_i)=r_i=R_i u_i(\xi_i)^{-(p_i-1) / 2 }.
		\end{align}
		
Let $\bar u_i^{,G}(r)$ be defined as in \eqref{weighted averaged profiles}. It follows from the assumption that $\xi_i\to0$ is an isolated simple blow-up point and Proposition \ref{solution is bubble} that there exists $\rho>0$, independent of $i$, such that, for all sufficiently large $i$,
\be\label{weighted profile is decreasing}
r^{\frac{2}{p_i-1}}\bar u_i^{,G}(r)
=\bar w_i^{,G}(r)
\quad\text{is strictly decreasing for}\quad
r_i<r<\rho.
\ee

Let
\[
s=d_{\mathbb H^n}(\xi,\xi_i)
\quad \text{ for }r_i<s<\rho.
\]
By Lemma \ref{Harnack ineq 2} and the definition of the weighted mean,
\[
\max_{\partial D_s(\xi_i)}u_i
\leq C\min_{\partial D_s(\xi_i)}u_i
\leq C\bar u_i^{\,G}(s).
\]
Therefore, using \eqref{5.24 in sasan},
\[
\begin{aligned}
s^{\frac{2}{p_i-1}}u_i(\xi)
&\leq C s^{\frac{2}{p_i-1}}\bar u_i^{\,G}(s)
=C\bar w_i^{\,G}(s)\\
&\leq C\bar w_i^{\,G}(r_i)
=C r_i^{\frac{2}{p_i-1}}\bar u_i^{\,G}(r_i)\\
&\leq C r_i^{\frac{2}{p_i-1}}
 u_i(\xi_i)R_i^{2-Q}\\
&=C R_i^{\frac{2}{p_i-1}+2-Q}
\leq C R_i^{\frac{2-Q}{2}+o(1)}.
\end{aligned}
\]
Here we used $R_i^{\tau_i}=1+o(1)$ in the last inequality. Hence
\be\label{eq: new}
u_i(\xi)^{p_i-1}
\leq
O(R_i^{-2+o(1)})
 d_{\mathbb H^n}(\xi,\xi_i)^{-2}
\quad\text{for }
 r_i\leq d_{\mathbb H^n}(\xi,\xi_i)\leq\rho.
\ee

		Consider the operator
		$$
		\mathcal{L}_i =\Delta_{\mathbb{H}^{n}} +a_i + u_i^{p_i-1} .
		$$
		Now, we look for a supersolution $\varphi_i$ of the operator $\mathcal{L}_i$ such that $\varphi_i \geq u_i$ on the boundary of the annulus $\mathscr{A}_i$, where
		$
		\mathscr{A}_i:=\{\xi \in \mathbb{H}^{n}: r_i<d_{\mathbb{H}^n}(\xi, \xi_i)<\rho_\delta\}.
		$
		Consider the function $d_{\mathbb{H}^n}(\xi,\xi_i)^{-\mu}-\varepsilon_0 |z-z_i|^2d_{\mathbb{H}^n}(\xi,\xi_i)^{-\mu-2}$ for $\xi=(z,t)\in\mathscr{A}_i$, $\xi_i=(z_i,t_i)$ and $\varepsilon_0$ is a small constant to be chosen later. By direct calculations, we have
		$$
		\begin{aligned}
			&\quad \mathcal{L}_i(d_{\mathbb{H}^n}(\xi,\xi_i)^{-\mu}-\varepsilon_0 |z-z_i|^2d_{\mathbb{H}^n}(\xi,\xi_i)^{-\mu-2}) \\
			&=\Delta_{\mathbb{H}^n}(d_{\mathbb{H}^n}(\xi, \xi_i)^{-\mu}-\varepsilon_0 |z-z_i|^2d_{\mathbb{H}^n}(\xi,\xi_i)^{-\mu-2})\\
			&\quad + (a_i+u_i^{p_i-1})(d_{\mathbb{H}^n}(\xi, \xi_i)^{-\mu}-\varepsilon_0 |z-z_i|^2d_{\mathbb{H}^n}(\xi,\xi_i)^{-\mu-2})\\
			& \leq-\mu(Q-2-\mu)|z-z_i|^2 d_{\mathbb{H}^n}(\xi, \xi_i)^{-\mu-4}\\
			& \quad -4\varepsilon_0nd_{\mathbb{H}^n}(\xi, \xi_i)^{-\mu-2}
			 +\varepsilon_0(\mu+2)(Q-\mu)|z-z_i|^4 d_{\mathbb{H}^n}(\xi, \xi_i)^{-\mu-6}\\
			& \quad +(a_i+O( R_i^{-2+o(1)}) d_{\mathbb{H}^n}(\xi, \xi_i)^{-2})(d_{\mathbb{H}^n}(\xi,\xi_i)^{-\mu}-\varepsilon_0 |z-z_i|^2d_{\mathbb{H}^n}(\xi,\xi_i)^{-\mu-2}),
		\end{aligned}
		$$
		where we have used \eqref{eq: new} in the inequality above.
		
		We first consider the case $|z-z_i|=0$. It follows that
		$$
		\begin{aligned}
			&\quad \mathcal{L}_i(d_{\mathbb{H}^n}(\xi,\xi_i)^{-\mu}-\varepsilon_0 |z-z_i|^2d_{\mathbb{H}^n}(\xi,\xi_i)^{-\mu-2}) \\
			& \leq-4\varepsilon_0nd_{\mathbb{H}^n}(\xi, \xi_i)^{-\mu-2}+(a_i+C R_i^{-2+o(1)} d_{\mathbb{H}^n}(\xi, \xi_i)^{-2})d_{\mathbb{H}^n}(\xi,\xi_i)^{-\mu},
		\end{aligned}
		$$
		Then, we consider the case $|z-z_i|\neq 0$. Noting that $\frac{|z-z_i|^2}{d_{\mathbb{H}^n}(\xi,\xi_i)^2}\leq 1$, it follows that
		$$
\begin{aligned}
& \mathcal{L}_i(d_{\mathbb{H}^n}(\xi, \xi_i)^{-\mu}-\varepsilon_0|z-z_i|^2 d_{\mathbb{H}^n}(\xi, \xi_i)^{-\mu-2}) \\
\leq\ & |z-z_i|^2d_{\mathbb{H}^n}(\xi, \xi_i)^{-\mu-6} (\varepsilon_0(\mu+2)(Q-\mu)|z-z_i|^2 -\mu(Q-2-\mu) d_{\mathbb{H}^n}(\xi, \xi_i)^{2})\\
& -4 \varepsilon_0 n d_{\mathbb{H}^n}(\xi, \xi_i)^{-\mu-2}  +(a_i+u_i^{p_i-1})(d_{\mathbb{H}^n}(\xi, \xi_i)^{-\mu}-\varepsilon_0|z-z_i|^2 d_{\mathbb{H}^n}(\xi, \xi_i)^{-\mu-2}).
\end{aligned}
$$

Fix \(0<\delta\ll 1\). Choose \(\varepsilon_0=\kappa\delta>0\) with \(\kappa\) sufficiently small. 
Shrinking \(\rho\) if necessary, we may find a fixed \(0<\rho_\delta\le \rho\), independent of \(i\), such that for all sufficiently large \(i\),
\be\label{supersoluton 1}
\mathcal{L}_i(d_{\mathbb H^n}(\xi,\xi_i)^{-\delta}-\varepsilon_0 |z-z_i|^2d_{\mathbb H^n}(\xi,\xi_i)^{-\delta-2})\le 0
\ee
for \(r_i\le d_{\mathbb H^n}(\xi,\xi_i)\le \rho_\delta\). Similarly,
\be\label{supersoluton 2}
\mathcal{L}_i(d_{\mathbb H^n}(\xi,\xi_i)^{2-Q+\delta}-\varepsilon_0 |z-z_i|^2d_{\mathbb H^n}(\xi,\xi_i)^{-Q+\delta})\le 0
\ee
for \(r_i\le d_{\mathbb H^n}(\xi,\xi_i)\le \rho_\delta\).

	Now set $\mathfrak{M}_i=\max _{\partial D_{\rho_{\delta}}(\xi_i)}  u_i$, $\lambda_i=((Q-2-\delta)(p_i-1))/2-1$ and
		$$
		\begin{aligned}
		\varphi_i(\xi)&=\frac{\mathfrak{M}_i}{1-\varepsilon_0} \rho_{\delta}^{\delta} (d_{\mathbb{H}^n}(\xi, \xi_i)^{-\delta}-\varepsilon_0 \frac{|z-z_i|^2}{d_{\mathbb{H}^n}(\xi,\xi_i)^2}d_{\mathbb{H}^n}(\xi,\xi_i)^{-\delta})\\
		&\quad +B u_i(\xi_i)^{-\lambda_i} (d_{\mathbb{H}^n}(\xi, \xi_i)^{2-Q+\delta}-\varepsilon_0 \frac{|z-z_i|^2}{d_{\mathbb{H}^n}(\xi,\xi_i)^2}d_{\mathbb{H}^n}(\xi,\xi_i)^{2-Q+\delta}) \quad \text { in }  \mathscr{A}_i,
		\end{aligned}
		$$
		 where $B>1$ will be chosen later. By \eqref{supersoluton 1} and \eqref{supersoluton 2}, we see that $\varphi_i$ is a supersolution of $\mathcal{L}_i$ in $\mathscr{A}_i$. Furthermore,
\begin{equation}\label{5.28}
\varphi_i(\xi) \geq \mathfrak{M}_i \geq u_i(\xi) \quad \text { for } d_{\mathbb{H}^n}\left(\xi, \xi_i\right)=\rho_{\delta} .
\end{equation}
Also,
\begin{equation}\label{5.29}
\begin{aligned}
\varphi_i(\xi) & \geq B(1-\varepsilon_0) u_i\left(\xi_i\right)^{-\lambda_i} d_{\mathbb{H}^n}\left(\xi, \xi_i\right)^{2-Q+\delta} \\
& \geq B(1-\varepsilon_0) u_i\left(\xi_i\right) R_i^{2-Q} \quad \text { for } d_{\mathbb{H}^n}\left(\xi, \xi_i\right)=r_i.
\end{aligned}
\end{equation}
Comparing \eqref{5.29} with \eqref{5.24 in sasan}, we choose $B$ large enough such that $B(1-\varepsilon_0) \geq C$, where $C$ is the constant occurring in equation \eqref{5.24 in sasan}. With this choice of $B$, we have
$$
\varphi_i(\xi) \geq u_i(\xi)\quad \text { for } d_{\mathbb{H}^n}\left(\xi, \xi_i\right)=r_i .
$$
From \eqref{5.28}, \eqref{5.29} and the maximum principle, it follows that
\begin{equation}\label{ui phi}
u_i(\xi) \leq \varphi_i(\xi) \quad\text { for all } r_i \leq d_{\mathbb{H}^n}\left(\xi, \xi_i\right) \leq \rho_{\delta} .
\end{equation}
By Lemma \ref{Harnack ineq 2} and the definition of $\bar u_i^{\,G}$,
\[
\mathfrak M_i
\leq C\bar u_i^{\,G}(\rho_\delta).
\]
Since $\bar w_i^{\,G}$ is decreasing on $(r_i,\rho_\delta)$, for every
$r_i<s<\rho_\delta$,
\[
\begin{aligned}
\rho_\delta^{\frac{2}{p_i-1}}\mathfrak M_i
&\leq C\rho_\delta^{\frac{2}{p_i-1}}
\bar u_i^{\,G}(\rho_\delta)
=C\bar w_i^{\,G}(\rho_\delta)\\
&\leq C\bar w_i^{\,G}(s)
=C s^{\frac{2}{p_i-1}}\bar u_i^{\,G}(s).
\end{aligned}
\]
By \eqref{ui phi},
\[
\bar u_i^{\,G}(s)
\leq
\frac{\mathfrak M_i}{1-\varepsilon_0}
\rho_\delta^\delta s^{-\delta}
+B u_i(\xi_i)^{-\lambda_i}s^{2-Q+\delta}.
\]
Thus,
\[
\begin{aligned}
\rho_\delta^{\frac{2}{p_i-1}}\mathfrak M_i
&\leq
C s^{\frac{2}{p_i-1}}
\left[
\frac{\mathfrak M_i}{1-\varepsilon_0}
\rho_\delta^\delta s^{-\delta}
+B u_i(\xi_i)^{-\lambda_i}s^{2-Q+\delta}
\right].
\end{aligned}
\]
Choose $s_0=s_0(\rho,n,A_0,A_1)>0$ sufficiently small so that
\[
C s_0^{\frac{2}{p_i-1}}
\rho_\delta^\delta s_0^{-\delta}
<
\frac{1-\varepsilon_0}{2}
\rho_\delta^{\frac{2}{p_i-1}}.
\]
Taking $s=s_0$, we obtain
\be\label{2.12}
\mathfrak M_i\leq C u_i(\xi_i)^{-\lambda_i}.
\ee
Combining \eqref{ui phi}, \eqref{2.12}, and Lemma
\ref{Harnack ineq 2}, we complete the proof of Lemma
\ref{lemma 2.2 in jde}.

	\end{proof}

	\begin{lemma}\label{lemma 5.9 in asan}
		Under the hypothesis of Lemma \ref{lemma 2.2 in jde}, we have
		
		$$
		\tau_i=O(u_i(\xi_i)^{-\min{(\frac{4}{Q-2},1)}+o(1)}),
		$$
		and therefore
		$$
		u_i(\xi_i)^{\tau_i}=1+o(1).
		$$
	\end{lemma}
	
	\begin{proof}
		First of all, observe that the generator of the one-parameter family of dilations centered at 
\(\xi_i=(\hat{x},\hat{y},\hat{t})\) is given by
$$
\begin{aligned}
\mathcal{X}_i&=\sum_{j=1}^n
\Big((\bar{x}-\hat{x})_j X_j+(\bar{y}-\hat{y})_j Y_j\Big)+
2\left(\bar{t}-\hat{t}+2(\hat{x}\cdot\bar{y}-\hat{y}\cdot\bar{x})\right)T .
\end{aligned}
$$
Equivalently, in the Euclidean coordinate frame, this vector field can be written as
$$
\mathcal{X}_i=\nu_i(\eta)\cdot \nabla,
$$
where
$$
\nu_i(\eta)=\left(\bar{x}-\hat{x},\bar{y}-\hat{y},2\left(\bar{t}-\hat{t}+\hat{x}\cdot\bar{y}-\hat{y}\cdot\bar{x}\right)\right).
$$
		We can notice that
		\begin{align}\label{eq:XN}
			\mathcal{X}_iI \cdot N=\mathcal{X}_iI \cdot \frac{\nabla d_{\mathbb{H}^n}(\eta,\xi_i)}{|\nabla d_{\mathbb{H}^n}(\eta,\xi_i)|}=\frac{\mathcal{X}_i d_{\mathbb{H}^n}(\eta,\xi_i)}{|\nabla d_{\mathbb{H}^n}(\eta,\xi_i)|}=\frac{d_{\mathbb{H}^n}(\eta,\xi_i)}{|\nabla d_{\mathbb{H}^n}(\eta,\xi_i)|}. 
		\end{align}

		Now, we define 
		$$
		\bar{u}_i:=u_i(\xi_i \circ \xi):=u_i(\eta).
		$$ 
		and
		$$
		\bar{a}_i:=a_i(\xi_i \circ \xi):=a_i(\eta).
		$$ 
		
		By Corollary \ref{pohozaev on heisenberg} with $R=1$, we have
		\begin{equation}\label{D}
		\begin{aligned}
			&\quad\int_{\partial D_1 }\mathcal{D}( \xi, \bar{u}_i, \nabla_{\mathbb{H}^{n}} \bar{u}_i)\ \mathrm{d}\mathcal{H}_{Q-2}\\
			&= \int_{D_1} \bar{a}_i \bar{u}_i^2 \ \mathrm{d} z \mathrm{d} t+\Big(\frac{Q}{p_i+1}-\frac{Q-2}{2}\Big) \int_{ D_1 } \bar{u}_i^{p_i+1}\ \mathrm{d} z \mathrm{d} t +\frac{1}{2} \int_{ D_1 } \mathcal{X}(\bar{a}_i) \bar{u}_i^{2}\ \mathrm{d} z \mathrm{d} t \\
			&\quad -\frac{1}{2} \int_{\partial  D_1 } \bar{a}_i \bar{u}_i^2 \mathcal{X}I \cdot N\ \mathrm{d} \mathcal{H}_{Q-2} -\frac{1}{p_i+1} \int_{\partial  D_1 } \bar{u}_i^{p_i+1} \mathcal{X}I \cdot N\ \mathrm{d} \mathcal{H}_{Q-2}\\
			&=\int_{D_1(\xi_i)} a_i(\eta) u_i(\eta)^2 \ \mathrm{d} \bar{z} \mathrm{d} \bar{t}+\Big(\frac{Q}{p_i+1}-\frac{Q-2}{2}\Big) \int_{ D_1 (\xi_i)} u_i(\eta)^{p_i+1}\ \mathrm{d} \bar{z} \mathrm{d} \bar{t} \\
			& \quad +\frac{1}{2} \int_{ D_1 (\xi_i)} \mathcal{X}_i(a_i) u_i(\eta)^{2}\ \mathrm{d} \bar{z} \mathrm{d} \bar{t} 
			-\frac{1}{2} \int_{\partial  D_1(\xi_i) } a_i(\eta) u_i(\eta)^2 \mathcal{X}_iI \cdot N\ \mathrm{d} \mathcal{H}_{Q-2}\\
			&\quad -\frac{1}{p_i+1} \int_{\partial  D_1 (\xi_i)} u_i(\eta)^{p_i+1} \mathcal{X}_iI \cdot N\ \mathrm{d} \mathcal{H}_{Q-2},
		\end{aligned}
			\end{equation}
		where
		$$
		\begin{aligned}
			\mathcal{D}( \xi, \bar{u}_i, \nabla_{\mathbb{H}^{n}} \bar{u}_i)=\frac{Q-2}{2}  (A \nabla \bar{u}_i \cdot N) \bar{u}_i-\frac{1}{2}|\nabla_{\mathbb{H}^{n}} \bar{u}_i|^2 \mathcal{X}I \cdot N 
			+(A \nabla \bar{u}_i \cdot N) \mathcal{X} \bar{u}_i
		\end{aligned}.
		$$

		Notice that
		$$|\mathcal{X}_i(a_i)|\leq Cd_{\mathbb{H}^n}(\eta,\xi_i)|\nabla_{\mathbb{H}^n} a_i|+Cd_{\mathbb{H}^n}(\eta,\xi_i)^2|T a_i|\leq Cd_{\mathbb{H}^n}(\eta,\xi_i)$$
		for $d_{\mathbb{H}^n}(\eta,\xi_i)\leq 1$.

		Hence, through variable substitution and using \eqref{D}, we have
		$$
		\begin{aligned}
			\tau_i &\int_{D_1(\xi_i)} u_i^{p_i+1}
			\leq C\bigg(\int_{D_1(\xi_i)} a_i u_i^2+\int_{D_1(\xi_i)}d_{\mathbb{H}^n}(\eta,\xi_i)  u_i^{2}+\int_{\partial  D_1 (\xi_i)} a_i u_i^2 \mathcal{X}_iI \cdot N\ \mathrm{d} \mathcal{H}_{Q-2} \\
			&+\int_{\partial  D_1 (\xi_i)} u_i^{p_i+1} \mathcal{X}_iI \cdot N\ \mathrm{d} \mathcal{H}_{Q-2} +\left|\int_{\partial  D_1 (\xi_i)} \mathcal{D}( \xi, u_i, \nabla_{\mathbb{H}^{n}} u_i)\ \mathrm{d} \mathcal{H}_{Q-2}\right|\bigg).
		\end{aligned}
		$$
		By Proposition \ref{solution is bubble}, we have
		$$
		\begin{aligned}
			&\quad\int_{D_1 (\xi_i)} u_i(\eta)^{p_i+1}\, \mathrm{d} \bar{z} \mathrm{d} \bar{t}\\
			& \geq C \int_{D_{r_i}(\xi_i)} u_i(\xi_i)^{p_i+1}[\bar c_n^2u_i(\xi_i)^{2(p_i-1) }(\bar{t}-\hat{t}+2  (\bar{y} \cdot \hat{x}-\bar{x}\cdot \hat{y}))^2\\& 
            \quad +(1+ \bar c_nu_i(\xi_i)^{p_i-1 }|\bar{z}-\hat{z}|^2)^2 ]^{(2-Q )(p_i+1) / 4} \ \mathrm{d} \bar{z} \mathrm{d} \bar{t}\\
			& \geq C u_i(\xi_i)^{\tau_i(Q/ 2-1)} \int_{D_{R_i}} \big(\bar c_n^2\tilde{t}^2+(1+\bar c_n|\tilde{z}|^2)^2\big)^{(2-Q )(p_i+1) / 4}\ \mathrm{d} \tilde{z} \mathrm{d} \tilde{t}\\
			& \geq C u_i(\xi_i)^{\tau_i(Q/ 2-1)},
		\end{aligned}
		$$
		and
		$$
		\begin{aligned}
			&\int_{D_{r_i}(\xi_i)} u_i(\eta)^2  \,\mathrm{d} \bar{z} \mathrm{d} \bar{t}
			\leq C u_i(\xi_i)^{-4/(Q-2)+o(1)}.
		\end{aligned}
		$$
		Similarly, 
		$$
		\begin{aligned}
			&\quad \int_{D_{r_i}(\xi_i)}d_{\mathbb{H}^n}(\eta,\xi_i) u_i(\eta)^{2}\ \mathrm{d} \bar{z}\mathrm{d} \bar{t}\\
			& \leq C u_i(\xi_i)^{-6 /(Q-2 )+o(1)} \int_{D_{R_i}}  \|\tilde{\xi}\|_{\mathbb{H}^n}{\big(\bar c_n^2\tilde{t}^2+(1+\bar c_n|\tilde{z}|^2)^2\big)^{(2-Q ) /2}}\ \mathrm{d} \tilde{z} \mathrm{d} \tilde{t}\\
			& \leq \begin{cases}
				O({u_i(\xi_i)^{-6 /(Q-2)+o(1)}}) & \text{ if } Q> 5,\\
				O(u_i(\xi_i)^{-2+o(1)}) & \text{ if } Q< 5 .
			\end{cases}
		\end{aligned}
		$$
		It follows from the exterior estimate and standard subelliptic estimates that 
		$$\left|\int_{\partial D_1\left(\xi_i\right)} \mathcal{D}\left(\xi, u_i, \nabla_{\mathbb{H}^n} u_i\right)\right| \leq C u_i\left(\xi_i\right)^{-2+\frac{4\delta}{Q-2}+o(1)}.$$

		By Lemma \ref{lemma 2.2 in jde}, as $R_i \rightarrow \infty$, we have
		$$
		\int_{D_1(\xi_i) \backslash D_{r_i}(\xi_i)} u_i(\eta)^{2}\ \mathrm{d}\bar{z}\mathrm{d} \bar{t} \leq  \begin{cases}O(u_i(\xi_i)^{-2 \lambda_i}) & \text{ if } Q<4 +2\delta, \\ 
			O(u_i(\xi_i)^{-4 /(Q-2)+o(1)}) & \text { if } Q>4+2\delta,\end{cases}
		$$
		$$
		\int_{D_1(\xi_i) \backslash D_{r_i}(\xi_i)} d_{\mathbb{H}^n}(\eta,\xi_i)u_i(\eta)^{2}\ \mathrm{d}\bar{z}\mathrm{d} \bar{t} \leq  \begin{cases}O(u_i(\xi_i)^{-2 \lambda_i}) & \text{ if } Q<5 +2\delta, \\ 
			O(u_i(\xi_i)^{-6 /(Q-2)+o(1)}) & \text { if } Q>5+2\delta.\end{cases}
		$$
		Also, we have
		
		$$
		\int_{\partial D_1(\xi_i)} u_i(\eta)^2\ \mathrm{d}\mathcal{H}_{Q-2} \leq C u_i(\xi_i)^{-2+\frac{4\delta}{Q-2}+o(1)},
		$$
		and
		$$
		\int_{\partial D_1(\xi_i)} u_i(\eta)^{p_i+1}\ \mathrm{d}\mathcal{H}_{Q-2} \leq \int_{\partial D_1(\xi_i)} u_i(\eta)^2\ \mathrm{d}\mathcal{H}_{Q-2} \leq C u_i(\xi_i)^{-2+\frac{4\delta}{Q-2}+o(1)} .
		$$
		Combining the above estimates, using $\tau_i=o(1)$ and choosing $\delta$ sufficiently small such that $-\frac{4\delta}{Q-2} +2>\min(\frac{4}{Q-2},1)$, we complete the proof.
	\end{proof}
	
	\begin{proposition}\label{pro 2.3 in jde}
		Under the assumptions of Lemma \ref{lemma 2.2 in jde}, we have 
		
		\be\label{5.22 in asan}
		u_i(\xi) \leq C u_i(\xi_i)^{-1}d_{\mathbb{H}^n}(\xi,\xi_i)^{2-Q}\quad \text{ for }d_{\mathbb{H}^n}(\xi,\xi_i) \leq 1.
		\ee
		
	\end{proposition}
	
	\begin{proof}
		Without loss of generality, we assume that $\rho$ occurring in the Definition \ref{isolated simple blow-up point} is less than $1 / 2$ and the proof follows from Proposition \ref{solution is bubble} and Lemma \ref{lemma 5.9 in asan} for $d_{\mathbb{H}^n}(\xi,\xi_i)<r_i$.
		
		Fix $\theta_0 \in \mathbb{H}^{n}$ with $\|\theta_0\|_{\mathbb{H}^n}=1$ and set $v_i(\xi)=u_i(\xi_i \circ \theta_0)^{-1} u_i(\xi)$. Then $v_i$ satisfies
		\be\label{vi}
		-\Delta_{\mathbb{H}^{n}} v_i= a_i(\xi)v_i+u_i(\xi_i \circ \theta_0)^{p_i-1}v_i^{p_i} \quad \text { in } D_2.
		\ee
		From Lemma \ref{Harnack ineq 2} and arguing as in Lemma \ref{w's converges}, after passing to a subsequence, $\{v_i\}$ converges in $C_{\text {loc }}^2(D_2 \backslash\{0\})$ to a positive function $v \in C^2(D_2 \backslash\{0\})$. By Lemma \ref{lemma 2.2 in jde}, $u_i(\xi_i \circ \theta_0) \rightarrow 0$,  and passing to the limit as $i \rightarrow \infty$ in \eqref{vi}, we see that $v$ satisfies
		$$
		\Delta_{\mathbb{H}^{n}} v + a(\xi)v=0 \quad \text { in }  D_2 \backslash\{0\}.
		$$
		
Define the weighted spherical means centered at $\xi_i$ by
\[
\bar v_i^{\,G}(s):=\mathcal M_{\xi_i}(v_i,s).
\]
Since $v_i=u_i(\xi_i\circ\theta_0)^{-1}u_i$, we have
\[
s^{\frac{2}{p_i-1}}\bar v_i^{\,G}(s)
=
 u_i(\xi_i\circ\theta_0)^{-1}\bar w_i^{\,G}(s).
\]
By the proof of Lemma \ref{lemma 2.2 in jde}, this function is strictly
decreasing on $(r_i,\rho)$. On every compact interval
$[s_1,s_2]\Subset(0,\rho)$, the fixed-unit-sphere expression in
\eqref{Garofalo weighted spherical mean}, the convergence
$v_i\to v$ in $C^2_{\rm loc}(D_2\setminus\{0\})$, and $\xi_i\to0$
give
\[
\bar v_i^{\,G}(s)\rightarrow
\bar v(s):=\mathcal M_0(v,s)
\]
uniformly for $s\in[s_1,s_2]$. Passing to the limit, we conclude that
\[
s^{\frac{Q-2}{2}}\bar v(s)
\]
is nonincreasing on $(0,\rho)$.

If $v$ were regular at the origin, then, by positivity and the strong
maximum principle, $v(0)>0$, and hence
\[
\bar v(s)\rightarrow v(0)
\qquad\text{as }s\to0^+.
\]
It would follow that
\[
s^{\frac{Q-2}{2}}\bar v(s)\rightarrow0
\qquad\text{as }s\to0^+,
\]
which is impossible for a positive nonincreasing function on
$(0,\rho)$. Therefore, $v$ is singular at the origin. Hence, by
Proposition \ref{bocher type},
\[
v(\xi)=cG(\xi)+k(\xi)
\quad \text{ for }c>0,
\]
where
\[
\Delta_{\mathbb H^n}k+a(\xi)k=0
\qquad\text{in }D_1.
\]

		We first prove the inequality \eqref{5.22 in asan} for $d_{\mathbb{H}^n}(\xi,\xi_i)=1$,
		i.e.,
		\be\label{boundness of eigenfunction}
		u_i(\xi_i \circ \theta_0) \leq C u_i(\xi_i)^{-1} .
		\ee
		We assume by contradiction that we have
		\be\label{contradiction 4.15}
		u_i(\xi_i) u_i(\xi_i \circ \theta_0) \rightarrow \infty \quad \text { as }  i \rightarrow \infty.
		\ee

      Fix \(0<d<\min\{\rho,1/2\}\) sufficiently small such that
the first eigenvalue \(\lambda_1=\lambda_1(D_d)\) of
\(-\Delta_{\mathbb H^n}\) satisfies
\[
\lambda_1>
\sup_i\|a_i\|_{L^\infty(D_d)}+1.
\]
Let \(\varphi>0\) be the corresponding first eigenfunction,
normalized by \(\varphi(0)=1\), namely,
\be\label{eigenvalue problem}
\begin{cases}
-\Delta_{\mathbb H^n}\varphi=\lambda_1\varphi
&\text{in }D_d,\\
\varphi=0
&\text{on }\partial D_d.
\end{cases}
\ee
Thus,
\[
-\Delta_{\mathbb H^n}\varphi\geq a_i\varphi
\qquad\text{in }D_d
\]
for every \(i\). Moreover, by the positivity of \(\varphi\),
\[
m_d:=\min_{\overline{D_{d/2}}}\varphi>0.
\]

Fix \(0<\theta\leq d/2\), to be chosen sufficiently small below.
Since \(\xi_i\to0\), we have
\[
D_\theta(\xi_i)\subset D_d
\]
for all sufficiently large \(i\). Set
\[
\bar v_i(\xi):=\varphi(\xi)^{-1}v_i(\xi).
\]
 Integrating by parts over  $D_{\theta} (\xi_i)$, by the boundedness of eigenfunction, \eqref{contradiction 4.15} and Proposition \ref{solution is bubble}, we obtain
\[
\begin{aligned}
&\quad
-\int_{\partial D_\theta(\xi_i)}
\varphi^2A\nabla\bar v_i\cdot N\,
\mathrm d\mathcal H_{Q-2}\\
&=
-\int_{\partial D_\theta(\xi_i)}
A(\varphi\nabla v_i-v_i\nabla\varphi)\cdot N\,
\mathrm d\mathcal H_{Q-2}\\
&=
\int_{D_\theta(\xi_i)}
\left[
(a_i-\lambda_1)\varphi v_i
+
\varphi
u_i(\xi_i\circ\theta_0)^{p_i-1}v_i^{p_i}
\right]\,\mathrm d\xi\\
&\leq
u_i(\xi_i\circ\theta_0)^{-1}
\int_{D_\theta(\xi_i)}
\varphi u_i^{p_i}\,\mathrm d\xi\\
&\leq
C(d,\theta)
u_i(\xi_i\circ\theta_0)^{-1}u_i(\xi_i)^{-1}
\rightarrow0.
\end{aligned}
\]

On the other hand, set $\Gamma(\xi):=\|\xi\|_{\mathbb H^n}^{2-Q}$ and
\[
\gamma_Q:=
-\int_{\partial D_\theta}
A\nabla\Gamma\cdot N\,
\mathrm d\mathcal H_{Q-2}>0.
\]


By the symmetry of the first eigenfunction, the regularity of $k$, and the estimate \eqref{eq:K} for the regular part of the Green function, we have, for \(i\) sufficiently large,
\[
\begin{aligned}
&\quad
-\int_{\partial D_\theta(\xi_i)}
\varphi^2A\nabla\bar v_i\cdot N\,
\mathrm d\mathcal H_{Q-2}\\
&\geq
-c\int_{\partial D_\theta}
\varphi A\nabla\Gamma\cdot N\,
\mathrm d\mathcal H_{Q-2}
-C_d\theta^2\\
&\geq
-c\min_{\overline{D_{d/2}}}\varphi
\int_{\partial D_\theta}
A\nabla\Gamma\cdot N\,
\mathrm d\mathcal H_{Q-2}
-C_d\theta^2\\
&=
c\gamma_Q\min_{\overline{D_{d/2}}}\varphi
-C_d\theta^2\\
&=:m(\theta)>0.
\end{aligned}
\]

Combining this estimate with the preceding integral identity, we
obtain
\[
\begin{aligned}
0<m(\theta)
&\leq
-\int_{\partial D_\theta(\xi_i)}
\varphi^2A\nabla\bar v_i\cdot N\,
\mathrm d\mathcal H_{Q-2}\\
&\leq
C(d,\theta)
u_i(\xi_i\circ\theta_0)^{-1}
u_i(\xi_i)^{-1}
\rightarrow0,
\end{aligned}
\]
which is a contradiction. Therefore,
\[
u_i(\xi_i\circ\theta_0)
\leq Cu_i(\xi_i)^{-1}.
\]

		We can now proceed as in the proof of Proposition 2.3 in \cite{Li JDE} to complete the proof of Proposition \ref{pro 2.3 in jde}.
	\end{proof}

	\begin{lemma}\label{lem 2.4 in jde}
		Under the hypothesis of Proposition \ref{pro 2.3 in jde}, we have
		$$
		\begin{aligned}
			\int_{d_{\mathbb{H}^n}(\xi,\xi_i) \leq r_i} d_{\mathbb{H}^n}(\xi,\xi_i)^s u_i(\xi)^{2} 
			=\begin{cases}
				O(u_i(\xi_i)^{-(4+2s)/(Q-2)})  & \text{ if } s+4<Q, \\
				O(u_i(\xi_i)^{-2} \ln u_i(\xi_i)) & \text{ if }s+4=Q, \\
				o(u_i(\xi_i)^{-2}) \ &\text{ if }s+4>Q .
			\end{cases}
		\end{aligned}
		$$
		Meanwhile, we have
		$$
		\begin{aligned}
			& \int_{r_i \leq d_{\mathbb{H}^n}(\xi,\xi_i) \leq 1} d_{\mathbb{H}^n}(\xi,\xi_i)^s u_i(\xi)^{2} \leq\begin{cases}
				o(u_i(\xi_i)^{-(4+2s)/(Q-2)}) & \text{ if }s+4<Q, \\
				O(u_i(\xi_i)^{-2} \ln u_i(\xi_i)) & \text{ if }s+4=Q, \\
				O(u_i(\xi_i)^{-2}) & \text{ if }s+4>Q.
			\end{cases}
		\end{aligned}
		$$
	\end{lemma}

	\begin{proof}
		  We first prove the case
$d_{\mathbb{H}^n}(\xi,\xi_i) \leq r_i$.
By Proposition \ref{solution is bubble}, we have
\[
\begin{aligned}
	&\quad
	\int_{d_{\mathbb{H}^n}(\xi,\xi_i) \leq r_i}
	d_{\mathbb{H}^n}(\xi,\xi_i)^s u_i(\xi)^{2}\\
	&\leq
	C\int_{d_{\mathbb{H}^n}(\xi,\xi_i) \leq r_i}
	\|\xi_i^{-1}\circ \xi\|_{\mathbb{H}^n}^s
	u_i(\xi_i)^2
	\Lambda_{0,1}\left(
	u_i(\xi_i)^{\frac{p_i-1}{2}}
	(\xi_i^{-1}\circ\xi)
	\right)^2\\
	&=
	C\int_{D_{R_i}}
	\|\eta\|_{\mathbb{H}^n}^s
	u_i(\xi_i)^{\frac{s(1-p_i)}{2}}
	u_i(\xi_i)^2
	\Lambda_{0,1}(\eta)^2
	u_i(\xi_i)^{\frac{Q(1-p_i)}{2}}\\
	&=
	Cu_i(\xi_i)^{-\frac{4+2s}{Q-2}}
	\int_{D_{R_i}}
	\left(|\bar z|^4+\bar t^2\right)^{\frac{s}{4}}
	\left[
	\bar c_n^2\bar t^2+
	\left(1+\bar c_n|\bar z|^2\right)^2
	\right]^{\frac{2-Q}{2}}
	\,\mathrm{d}\bar z\,\mathrm{d}\bar t\\
	&\leq
	Cu_i(\xi_i)^{-\frac{4+2s}{Q-2}}
	\int_{D_{R_i}}
	\max\left\{|\bar z|^4,\bar t^2\right\}^{\frac{s}{4}}
	\left[
	\bar c_n^2\bar t^2+
	\left(1+\bar c_n|\bar z|^2\right)^2
	\right]^{\frac{2-Q}{2}}
	\,\mathrm{d}\bar z\,\mathrm{d}\bar t .
\end{aligned}
\]

        Set
\[
\widetilde{\eta}
=
(\widetilde z,\widetilde t)
:=
\delta_{\sqrt{\bar c_n}}\eta
=
\left(
\sqrt{\bar c_n}\,\bar z,
\bar c_n\bar t
\right).
\]
Then
\[
\mathrm{d}\bar z\,\mathrm{d}\bar t
=
\bar c_n^{-\frac Q2}
\,\mathrm{d}\widetilde z\,\mathrm{d}\widetilde t,
\]
and
\[
\max\left\{|\bar z|^4,\bar t^2\right\}^{\frac{s}{4}}
=
\bar c_n^{-\frac s2}
\max\left\{
|\widetilde z|^4,\widetilde t^2
\right\}^{\frac{s}{4}}.
\]
Moreover,
\[
\delta_{\sqrt{\bar c_n}}(D_{R_i})
=
D_{\sqrt{\bar c_n}R_i}
\]
and
\[
\bar c_n^2\bar t^2+
\left(1+\bar c_n|\bar z|^2\right)^2
=
\widetilde t^2+
\left(1+|\widetilde z|^2\right)^2.
\]

Using the formula in \cite[Lemma 5.5]{JL intrinsic} and Lemma \ref{lemma 5.9 in asan}, for $s<Q-4$, we obtain
		$$
		\begin{aligned}
		&\int_{D_{R_i}} \max\{|\bar{z}|^4,\bar{t}^2\}^{\frac{s}{4}}(\bar c_n^2\bar{t}^2+(1+\bar c_n|\bar{z}|^2)^2)^{\frac{2-Q}{2}}\\
        &\leq C(n)\left(\frac{\Gamma(\frac{Q-2+s}{2})\Gamma(\frac{Q-4-s}{2})\Gamma(\frac{1}{2})\Gamma(\frac{Q-3}{2})}{2\Gamma(Q-3)\Gamma(\frac{Q-2}{2})}+\frac{\Gamma\left(\frac{Q-4-s}{2}\right) \Gamma\left(\frac{s+2}{4}\right) \Gamma\left(\frac{2 Q-6-s}{4}\right)}{2 \Gamma\left(Q-3-\frac{s}{2}\right)}\right),
		\end{aligned}
		$$
		which means
		$$\int_{d_{\mathbb{H}^n}(\xi,\xi_i) \leq r_i} d_{\mathbb{H}^n}(\xi,\xi_i)^s u_i(\xi)^{2}= O(u_i(\xi_i)^{-\frac{4+2s}{Q-2}}).$$
		Similarly, for $s=Q-4$, we have  
		$$\int_{d_{\mathbb{H}^n}(\xi,\xi_i) \leq r_i} d_{\mathbb{H}^n}(\xi,\xi_i)^s u_i(\xi)^{2}= O(u_i(\xi_i)^{-2}\ln(u_i(\xi_i))).$$

		For $s>Q-4$, we have
\[
\begin{aligned}
	&\quad
	u_i(\xi_i)^{-\frac{4+2s}{Q-2}}
	\int_{D_{R_i}}
	\left(|\bar z|^4+\bar t^2\right)^{\frac{s}{4}}
	\left[
	\bar c_n^2\bar t^2+
	\left(1+\bar c_n|\bar z|^2\right)^2
	\right]^{\frac{2-Q}{2}}
	\,\mathrm{d}\bar z\,\mathrm{d}\bar t\\
	&=
	u_i(\xi_i)^{-\frac{4+2s}{Q-2}}
	\int_{D_{R_i}}
	\|\eta\|_{\mathbb{H}^n}^{s}
	\left[
	\bar c_n^2\|\eta\|_{\mathbb{H}^n}^4
	+1+2\bar c_n|\bar z|^2
	\right]^{\frac{2-Q}{2}}
	\,\mathrm{d}\bar z\,\mathrm{d}\bar t\\
	&\leq
	\bar c_n^{2-Q}
	u_i(\xi_i)^{-\frac{4+2s}{Q-2}}
	\int_{D_{R_i}}
	\|\eta\|_{\mathbb{H}^n}^{s+4-2Q}
	\,\mathrm{d}\bar z\,\mathrm{d}\bar t\\
	&\leq
	Cu_i(\xi_i)^{
	-\frac{4+2s}{Q-2}
	+\frac{(p_i-1)(s+4-Q)}{2}
	}
r_i^{s+4-Q}\\
	&=
	o\left(u_i(\xi_i)^{-2}\right).
\end{aligned}
\]
which means
		$$\int_{d_{\mathbb{H}^n}(\xi,\xi_i)\leq r_i} d_{\mathbb{H}^n}(\xi,\xi_i)^s u_i(\xi)^{2}= o(u_i(\xi_i)^{-2}).$$
		The proof of the case $r_i \leq d_{\mathbb{H}^n}(\xi,\xi_i) \leq 1$ is similar by using Proposition \ref{pro 2.3 in jde}. We omit it here.

	\end{proof}

	\begin{lemma}\label{lem 2.5 in jde}
		Suppose that $u_i$ satisfies \eqref{equation2}, $\xi_i \rightarrow 0$ is an isolated simple blow-up point 
		for some constant $A_1$ and $\|a_i\|_{C^2(D_3)} \leq A_0$. 
		Then we have
		
		$$
		\begin{aligned}
			\tau_i \leq  C
			\begin{cases}
				(\|\nabla_{\mathbb{H}^n} a_i\|_{L^{\infty}(D_1)}+\|\nabla_{\mathbb{H}^n}^2 a_i\|_{L^{\infty}(D_1)})u_i(\xi_i)^{-2} &\text{ if }Q=4 , \\
				\|\nabla_{\mathbb{H}^n} a_i\|_{L^{\infty}(D_1)}u_i(\xi_i)^{-\frac{6}{Q-2}}+\|\nabla_{\mathbb{H}^n}^2 a_i\|_{L^{\infty}(D_1)}O(u_i(\xi_i)^{-2}\ln u_i(\xi_i)) & \text{ if } Q=6, \\
				\|\nabla_{\mathbb{H}^n} a_i\|_{L^{\infty}(D_1)}u_i(\xi_i)^{-\frac{6}{Q-2}}+\|\nabla_{\mathbb{H}^n}^2 a_i\|_{L^{\infty}(D_1)}O(u_i(\xi_i)^{-\frac{8}{Q-2}})  & \text{ if } Q\geq 8,
			\end{cases}
		\end{aligned}
		$$	
		where $C=C(n, A_0, A_1, \rho)>0$.
	\end{lemma}
	
	\begin{proof}
		By the Pohozaev identity in Corollary \ref{pohozaev on heisenberg}, we have
		\begin{equation}\label{3.6.1}
		\begin{aligned}
			& \quad\int_{\partial D_1(\xi_i)} \mathcal{D}(\xi_i^{-1} \circ \xi, u_i, \nabla_{\mathbb{H}^{n}} u_i)\ \mathrm{d} \mathcal{H}_{Q-2}\\
			&=\frac{1}{2} \int_{D_1(\xi_i)} \mathcal{X}_i( a_i) u_i^{2}\ \mathrm{d} z \mathrm{d} t
			+\Big(\frac{Q}{p_i+1}-\frac{Q-2}{2}\Big) \int_{D_1(\xi_i)}  u_i^{p_i+1}\ \mathrm{d} z \mathrm{d} t +\int_{D_1(\xi_i)}  a_i u_i^{2}\ \mathrm{d} z \mathrm{d} t\\
			& \quad -\frac{1}{2} \int_{\partial D_1(\xi_i)}a_i u_i^2 \mathcal{X}_i I\cdot N \ \mathrm{d} \mathcal{H}_{Q-2}-\frac{1}{p_i+1} \int_{\partial D_1(\xi_i)} u_i^{p_i+1} \mathcal{X}_i I \cdot N\ \mathrm{d} \mathcal{H}_{Q-2}.
		\end{aligned}
		\end{equation}
		It follows from \eqref{3.6.1}, Propositions \ref{solution is bubble} and \ref{pro 2.3 in jde} that
		$$
		\begin{aligned}
			\tau_i \leq & C u_i(\xi_i)^{-2}+C \int_{D_1(\xi_i)} \mathcal{X}_i( a_i)  u_i(\xi)^2\ \mathrm{d} z \mathrm{d} t.
		\end{aligned}
		$$
Using the definition of $X_i$, $Y_i$ and Lemma \ref{lem 2.4 in jde}, we have
		$$
		\begin{aligned}\label{2.24 in jde}
			&\quad\int_{D_1(\xi_i)} \mathcal{X}_i( a_i)  u_i(\xi)^2\\
			& \leq\int_{D_1(\xi_i)}d_{\mathbb{H}^n}(\xi,\xi_i) |\nabla_{\mathbb{H}^n} a_i(\xi)|u_i(\xi)^{2}+d_{\mathbb{H}^n}(\xi,\xi_i)^2|\partial_t a_i(\xi)| u_i(\xi)^{2}\\
			& \leq C\|\nabla_{\mathbb{H}^n} a_i\|_{L^{\infty}(D_1)} \int_{D_1(\xi_i)}d_{\mathbb{H}^n}(\xi,\xi_i) u_i(\xi)^{2} +C\|\nabla_{\mathbb{H}^n}^2 a_i\|_{L^{\infty}(D_1)} \int_{D_1(\xi_i)}d_{\mathbb{H}^n}(\xi,\xi_i)^2 u_i(\xi)^{2}\\
			&\leq {C
			\begin{cases}
				\|\nabla_{\mathbb{H}^n} a_i\|_{L^{\infty}(D_1)}u_i(\xi_i)^{-2}+\|\nabla_{\mathbb{H}^n}^2 a_i\|_{L^{\infty}(D_1)}u_i(\xi_i)^{-2} \ &\text{ if }Q=4, \\
				\|\nabla_{\mathbb{H}^n} a_i\|_{L^{\infty}(D_1)}u_i(\xi_i)^{\frac{-6}{Q-2}}+\|\nabla_{\mathbb{H}^n}^2 a_i\|_{L^{\infty}(D_1)}u_i(\xi_i)^{-2}\ln u_i(\xi_i) & \text{ if } Q=6, \\
			\|\nabla_{\mathbb{H}^n} a_i\|_{L^{\infty}(D_1)}u_i(\xi_i)^{\frac{-6}{Q-2}}+\|\nabla_{\mathbb{H}^n}^2 a_i\|_{L^{\infty}(D_1)}u_i(\xi_i)^{\frac{-8}{Q-2}}  & \text{ if } Q\geq 8	 ,
			\end{cases}}\\
		\end{aligned}
		$$
where we used the definition of $X_i$, $Y_i$ in the first inequality. This completes the proof.
	\end{proof}
	
	\begin{lemma}\label{lem 2.6 in jde}
		Under the hypothesis of Lemma \ref{lem 2.5 in jde}, we have,
		
		$$
		\begin{aligned}
			|a_i(\xi_i)| \leq  C 
			\begin{cases}
				(\ln u_i(\xi_i))^{-1}(1+\|\nabla_{\mathbb{H}^n}^2 a_i\|_{L^{\infty}(D_1)}) &\text{ if }Q=4,  \\
				u_i(\xi_i)^{-2+\frac{4}{Q-2}}(1+\|\nabla_{\mathbb{H}^n}^2 a_i\|_{L^{\infty}(D_1)} \ln u_i(\xi_i))  & \text{ if } Q=6, \\
			u_i(\xi_i)^{-2+\frac{4}{Q-2}}(1+\|\nabla_{\mathbb{H}^n}^2 a_i\|_{L^{\infty}(D_1)} u_i(\xi_i)^{\frac{2Q-12}{Q-2}})  & \text{ if } Q\geq 8,	
			\end{cases}
		\end{aligned}
		$$
	and
		$$
		\begin{aligned}
			|\nabla_{\mathbb{H}^n}  a_i(\xi_i)|  \leq  C 
			\begin{cases}
				u_i(\xi_i)^{-2+\frac{4}{Q-2}}(\ln u_i(\xi_i))^{-1}(1+\|\nabla_{\mathbb{H}^n}^2 a_i\|_{L^{\infty}(D_1)}) \ &\text{ if }Q=4 , \\
				u_i(\xi_i)^{-2+\frac{4}{Q-2}}(1+\|\nabla_{\mathbb{H}^n}^2 a_i\|_{L^{\infty}(D_1)}u_i(\xi_i)^{\frac{2Q-10}{Q-2}})  & \text{ if } Q\geq 6,
			\end{cases}
		\end{aligned}
		$$
		where $C=C(n, A_0, A_1, \rho)$.
	\end{lemma}

\begin{proof}
	Let $\eta\in C_c^\infty(D_{1/2})$ be a cut-off function such that
	\[
	\begin{aligned}
		\eta(\xi)&=1
		\qquad\text{for }\|\xi\|_{\mathbb H^n}\leq\frac14,\\
		\eta(\xi)&=0
		\qquad\text{for }\|\xi\|_{\mathbb H^n}\geq\frac12.
	\end{aligned}
	\]
	We use the cut-off function centered at $\xi_i$, namely,
	\[
	\eta_i(\xi):=\eta(\xi_i^{-1}\circ\xi).
	\]
	Then
	\[
	\eta_i=1
	\quad\text{in }D_{1/4}(\xi_i),
	\qquad
	\operatorname{supp}\eta_i\subset D_{1/2}(\xi_i).
	\]
	Since $\xi_i\to0$, for all sufficiently large $i$,
	\[
	D_{1/2}(\xi_i)\subset D_1.
	\]

	We next derive the lower bound for the weighted $L^2$-mass
	which will be used throughout the proof. Since
	$\eta_i=1$ in $D_{1/4}(\xi_i)$, Proposition
	\ref{2.2 in jde} gives
	\begin{equation}\label{54 in imrn}
	\begin{aligned}
		\int_{D_1}u_i(\xi)^2\eta_i(\xi)\,\mathrm d\xi
		&\geq
		\int_{D_{1/4}(\xi_i)}u_i(\xi)^2\,\mathrm d\xi\\
		&\geq
		Cu_i(\xi_i)^2
		\int_{D_{1/4}(\xi_i)}
		\Lambda_{0,1}\left(
		\delta_{u_i(\xi_i)^{(p_i-1)/2}}
		(\xi_i^{-1}\circ\xi)
		\right)^2
		\,\mathrm d\xi\\
		&=
		Cu_i(\xi_i)^{
		2-\frac Q2(p_i-1)}
		\int_{D_{\frac14u_i(\xi_i)^{(p_i-1)/2}}}
		\Lambda_{0,1}(\bar\xi)^2
		\,\mathrm d\bar\xi\\
		&=
		Cu_i(\xi_i)^{
		-\frac4{Q-2}+\frac Q2\tau_i}
		\int_{D_{\frac14u_i(\xi_i)^{(p_i-1)/2}}}
		\Lambda_{0,1}(\bar\xi)^2
		\,\mathrm d\bar\xi\\
		&\geq
		C
		\begin{cases}
			u_i(\xi_i)^{-\frac4{Q-2}}
			\ln u_i(\xi_i),
			&\text{if }Q=4,\\[1ex]
			u_i(\xi_i)^{-\frac4{Q-2}},
			&\text{if }Q\geq6.
		\end{cases}
	\end{aligned}
	\end{equation}

    We now introduce the translated right-invariant vector fields
centered at $\xi_i$. Writing
\[
\xi_i=(\hat x,\hat y,\hat t),
\]
we define
\[
\begin{aligned}
\overline X_j^{\,i}
&:=
\overline X_j+4\hat y_j\partial_t
=
X_j-4(y_j-\hat y_j)\partial_t,\\
\overline Y_j^{\,i}
&:=
\overline Y_j-4\hat x_j\partial_t
=
Y_j+4(x_j-\hat x_j)\partial_t.
\end{aligned}
\]
They are divergence-free, commute with all the left-invariant
horizontal vector fields, and satisfy
\[
\overline X_j^{\,i}\big|_{\xi_i}=X_j\big|_{\xi_i},
\qquad
\overline Y_j^{\,i}\big|_{\xi_i}=Y_j\big|_{\xi_i}.
\]
Moreover, $\overline X_j^{\,i}\eta_i$ and
$\overline Y_j^{\,i}\eta_i$ are uniformly bounded and supported
in
\[
D_{1/2}(\xi_i)\setminus\overline{D_{1/4}(\xi_i)}.
\]

	We first estimate $X_ja_i(\xi_i)$. Multiplying
	\eqref{equation2} by
	$(\overline X_j^{\,i}u_i)\eta_i$ and integrating by parts on
	$D_1$, we have
	\[
	\begin{aligned}
	&\frac12
	\int_{D_1}
	(\overline X_j^{\,i}a_i)u_i^2\eta_i\\
	&=
	-\int_{D_1}
	\eta_i A\nabla u_i\cdot
	\nabla(\overline X_j^{\,i}u_i)-
	\int_{D_{1/2}(\xi_i)\setminus
	\overline{D_{1/4}(\xi_i)}}
	(\overline X_j^{\,i}u_i)
	A\nabla u_i\cdot\nabla\eta_i\\
	&\quad-
	\frac12
	\int_{D_{1/2}(\xi_i)\setminus
	\overline{D_{1/4}(\xi_i)}}
	a_i u_i^2\overline X_j^{\,i}\eta_i-
	\frac1{p_i+1}
	\int_{D_{1/2}(\xi_i)\setminus
	\overline{D_{1/4}(\xi_i)}}
	u_i^{p_i+1}\overline X_j^{\,i}\eta_i .
	\end{aligned}
	\]
	Since $\overline X_j^{\,i}$ commutes with all the
	left-invariant horizontal vector fields,
	\[
	\begin{aligned}
	-\int_{D_1}
	\eta_i A\nabla u_i\cdot
	\nabla(\overline X_j^{\,i}u_i)
	&=
	-\frac12
	\int_{D_1}
	\eta_i\overline X_j^{\,i}
	|\nabla_{\mathbb H^n}u_i|^2\\
	&=
	\frac12
	\int_{D_{1/2}(\xi_i)\setminus
	\overline{D_{1/4}(\xi_i)}}
	|\nabla_{\mathbb H^n}u_i|^2
	\overline X_j^{\,i}\eta_i.
	\end{aligned}
	\]
	Therefore,
	\[
	\begin{aligned}
	&\frac12
	\int_{D_1}
	(\overline X_j^{\,i}a_i)u_i^2\eta_i\\
	&=
	\frac12
	\int_{D_{1/2}(\xi_i)\setminus
	\overline{D_{1/4}(\xi_i)}}
	|\nabla_{\mathbb H^n}u_i|^2
	\overline X_j^{\,i}\eta_i-
	\int_{D_{1/2}(\xi_i)\setminus
	\overline{D_{1/4}(\xi_i)}}
	(\overline X_j^{\,i}u_i)
	A\nabla u_i\cdot\nabla\eta_i\\
	&\quad-
	\frac12
	\int_{D_{1/2}(\xi_i)\setminus
	\overline{D_{1/4}(\xi_i)}}
	a_i u_i^2\overline X_j^{\,i}\eta_i-
	\frac1{p_i+1}
	\int_{D_{1/2}(\xi_i)\setminus
	\overline{D_{1/4}(\xi_i)}}
	u_i^{p_i+1}\overline X_j^{\,i}\eta_i .
	\end{aligned}
	\]

	By Young's inequality,
	\[
	\begin{aligned}
	\left|
	\int_{D_{1/2}(\xi_i)\setminus
	\overline{D_{1/4}(\xi_i)}}
	(\overline X_j^{\,i}u_i)
	A\nabla u_i\cdot\nabla\eta_i
	\right|\leq
	C
	\int_{D_{1/2}(\xi_i)\setminus
	\overline{D_{1/4}(\xi_i)}}
	\left(
	|\nabla_{\mathbb H^n}u_i|^2
	+
	|\overline X_j^{\,i}u_i|^2
	\right).
	\end{aligned}
	\]
	Consequently,
	\begin{equation}\label{5.49 in asan}
	\begin{aligned}
	\left|
	\int_{D_1}
	(\overline X_j^{\,i}a_i)u_i^2\eta_i
	\right|&\leq C\int_{D_{1/2}(\xi_i)\setminus
	\overline{D_{1/4}(\xi_i)}}
	|\nabla_{\mathbb H^n}u_i|^2
    +C\int_{D_{1/2}(\xi_i)\setminus
	\overline{D_{1/4}(\xi_i)}}
	|\overline X_j^{\,i}u_i|^2\\
    &+Cu_i(\xi_i)^{-p_i-1}
	+Cu_i(\xi_i)^{-2}.
	\end{aligned}
	\end{equation}

	Using
	\[
	\overline X_j^{\,i}u_i
	=
	X_ju_i
	-
	4(y_j-\hat y_j)\partial_tu_i,
	\]
	we have
	\[
	\begin{aligned}
	\int_{D_{1/2}(\xi_i)\setminus
	\overline{D_{1/4}(\xi_i)}}
	|\overline X_j^{\,i}u_i|^2\leq
	C\int_{D_{1/2}(\xi_i)\setminus
	\overline{D_{1/4}(\xi_i)}}
	\left(
	|\nabla_{\mathbb H^n}u_i|^2
	+
	|\partial_tu_i|^2
	\right)
	\end{aligned}
	\]

    Indeed, by Proposition \ref{pro 2.3 in jde}, for any fixed
	$\theta\in\partial D_1$, after translating the coordinates by
	$\xi_i$, the sequence
	\[
	u_i(\xi_i\circ\theta)^{-1}
	u_i(\xi_i\circ\xi)
	\]
	converges in
	$C_{\mathrm{loc}}^2(D_2\setminus\{0\})$ to a function of the
	form
	\[
	v(\xi)=a_1\|\xi\|_{\mathbb H^n}^{2-Q}+k(\xi).
	\]
	Consequently, in the annular region $D_{1/2}(\xi_i)\setminus
	\overline{D_{1/4}(\xi_i)}$, we have
	\[
	u_i(\xi)\leq Cu_i(\xi_i)^{-1}.
	\]
	The local subelliptic estimates therefore give
	\begin{equation}\label{5.50 in asan}
	\int_{D_{1/2}(\xi_i)\setminus
	\overline{D_{1/4}(\xi_i)}}
	|\nabla_{\mathbb H^n}u_i|^2
	\leq
	Cu_i(\xi_i\circ\theta)^2
	\leq
	Cu_i(\xi_i)^{-2}.
	\end{equation}
	By applying the local second-order subelliptic estimate in a
	slightly larger annular region, we also have
	\[
	\begin{aligned}
	\int_{D_{1/2}(\xi_i)\setminus\overline{D_{1/4}(\xi_i)}}
	\left(|\nabla_{\mathbb H^n}^2u_i|^2
	+|\partial_tu_i|^2\right)\leq Cu_i(\xi_i)^{-2},
	\end{aligned}
	\]
	where we have used
	\[
	[X_j,Y_j]=-4\partial_t.
	\]

	Repeating the local estimate used in
	\eqref{5.50 in asan}, we obtain
	\begin{equation}\label{5.50(2) in asan}
	\begin{aligned}
	\int_{D_{1/2}(\xi_i)\setminus
	\overline{D_{1/4}(\xi_i)}}
	|\overline X_j^{\,i}u_i|^2
	\leq
	Cu_i(\xi_i\circ\theta)^2
	\leq
	Cu_i(\xi_i)^{-2}.
	\end{aligned}
	\end{equation}
	It follows from
	\eqref{5.49 in asan}, \eqref{5.50 in asan}, and
	\eqref{5.50(2) in asan} that
	\[
	\left|
	\int_{D_1}
	(\overline X_j^{\,i}a_i)u_i^2\eta_i
	\right|
	\leq
	Cu_i(\xi_i)^{-2}.
	\]

	At the center $\xi_i$, we have
	\[
	\overline X_j^{\,i}a_i(\xi_i)
	=
	X_ja_i(\xi_i).
	\]
	Therefore,
	\[
	\begin{aligned}
	&|X_ja_i(\xi_i)|
	\int_{D_1}u_i(\xi)^2\eta_i(\xi)\,\mathrm d\xi\\
	&\leq
	\left|
	\int_{D_1}
	(\overline X_j^{\,i}a_i)(\xi)
	u_i(\xi)^2\eta_i(\xi)\,\mathrm d\xi
	\right|+
	\int_{D_1}
	\left|
	(\overline X_j^{\,i}a_i)(\xi)
	-
	X_ja_i(\xi_i)
	\right|
	u_i(\xi)^2\eta_i(\xi)\,\mathrm d\xi.
	\end{aligned}
	\]
	Using
	\[
	\overline X_j^{\,i}
	=
	X_j-4(y_j-\hat y_j)\partial_t,
	\]
	we have
	\[
	\begin{aligned}
	\left|
	(\overline X_j^{\,i}a_i)(\xi)
	-
	X_ja_i(\xi_i)
	\right|\leq
	|X_ja_i(\xi)-X_ja_i(\xi_i)|
	+
	4|y_j-\hat y_j||\partial_ta_i(\xi)|.
	\end{aligned}
	\]
	By the first-order Taylor estimate on the Heisenberg group,
	see \cite[Theorem 20.3.1]{BEU},
	\[
	|X_ja_i(\xi)-X_ja_i(\xi_i)|
	\leq
	C
	\|\nabla_{\mathbb H^n}^2a_i\|_{L^\infty(D_1)}
	d_{\mathbb H^n}(\xi,\xi_i).
	\]
	Moreover, since
	\[
	[X_j,Y_j]=-4\partial_t,
	\]
	we have
	\[
	|\partial_ta_i(\xi)|
	\leq
	C
	\|\nabla_{\mathbb H^n}^2a_i\|_{L^\infty(D_1)}.
	\]
	Also,
	\[
	|y_j-\hat y_j|
	\leq
	d_{\mathbb H^n}(\xi,\xi_i).
	\]
	It follows that
	\[
	\left|
	(\overline X_j^{\,i}a_i)(\xi)
	-
	X_ja_i(\xi_i)
	\right|
	\leq
	C
	\|\nabla_{\mathbb H^n}^2a_i\|_{L^\infty(D_1)}
	d_{\mathbb H^n}(\xi,\xi_i).
	\]
	Consequently,
	\[
	\begin{aligned}
	&|X_ja_i(\xi_i)|
	\int_{D_1}u_i(\xi)^2\eta_i(\xi)\,\mathrm d\xi\\
	&\leq
	Cu_i(\xi_i)^{-2}+C\|\nabla_{\mathbb H^n}^2a_i\|_{L^\infty(D_1)}
	\int_{D_1}d_{\mathbb H^n}(\xi,\xi_i)u_i(\xi)^2\eta_i(\xi)\,\mathrm d\xi.
	\end{aligned}
	\]
	By Lemma \ref{lem 2.4 in jde}, this gives
	\[
	\begin{aligned}
	|X_ja_i(\xi_i)|
	\int_{D_1}u_i^2\eta_i\leq
	C
	\begin{cases}
		u_i(\xi_i)^{-2}
		\left(
		1+
		\|\nabla_{\mathbb H^n}^2a_i\|_{L^\infty(D_1)}
		\right)
		&\text{if }Q=4,\\[1ex]
		u_i(\xi_i)^{-2}
		(
		1+
		\|\nabla_{\mathbb H^n}^2a_i\|_{L^\infty(D_1)}
		u_i(\xi_i)^{\frac{2Q-10}{Q-2}}
		)
		&\text{if }Q\geq6.
		\end{cases}
	\end{aligned}
	\]
	Using \eqref{54 in imrn}, we obtain
	\[
	\begin{aligned}
	|X_ja_i(\xi_i)|
	\leq
	C
	\begin{cases}
	u_i(\xi_i)^{-2+\frac4{Q-2}}
	(\ln u_i(\xi_i))^{-1}
	\left(
	1+
	\|\nabla_{\mathbb H^n}^2a_i\|_{L^\infty(D_1)}
	\right)
	&\text{if }Q=4,\\[1ex]
	u_i(\xi_i)^{-2+\frac4{Q-2}}
	(
	1+
	\|\nabla_{\mathbb H^n}^2a_i\|_{L^\infty(D_1)}
	u_i(\xi_i)^{\frac{2Q-10}{Q-2}}
	)
	&\text{if }Q\geq6.
	\end{cases}
	\end{aligned}
	\]

    Similarly, applying the same argument with
$(\overline Y_j^{\,i}u_i)\eta_i$ in place of
$(\overline X_j^{\,i}u_i)\eta_i$, we obtain
	\[
	\begin{aligned}
	|Y_ja_i(\xi_i)|
	\leq
	C
	\begin{cases}
	u_i(\xi_i)^{-2+\frac4{Q-2}}
	(\ln u_i(\xi_i))^{-1}
	(
	1+
	\|\nabla_{\mathbb H^n}^2a_i\|_{L^\infty(D_1)}
	)
	&\text{if }Q=4,\\[1ex]
	u_i(\xi_i)^{-2+\frac4{Q-2}}
	(
	1+
	\|\nabla_{\mathbb H^n}^2a_i\|_{L^\infty(D_1)}
	u_i(\xi_i)^{\frac{2Q-10}{Q-2}}
	)
	&\text{if }Q\geq6.
	\end{cases}
	\end{aligned}
	\]
	Summing the estimates for $X_ja_i(\xi_i)$ and
	$Y_ja_i(\xi_i)$ over $j=1,\ldots,n$, we conclude that
	\[
	\begin{aligned}
	|\nabla_{\mathbb H^n}a_i(\xi_i)|
	\leq
	C
	\begin{cases}
	u_i(\xi_i)^{-2+\frac4{Q-2}}
	(\ln u_i(\xi_i))^{-1}
	(
	1+
	\|\nabla_{\mathbb H^n}^2a_i\|_{L^\infty(D_1)}
	)
	&\text{if }Q=4,\\[1ex]
	u_i(\xi_i)^{-2+\frac4{Q-2}}
	(
	1+
	\|\nabla_{\mathbb H^n}^2a_i\|_{L^\infty(D_1)}
	u_i(\xi_i)^{\frac{2Q-10}{Q-2}}
	)
	&\text{if }Q\geq6.
	\end{cases}
	\end{aligned}
	\]

	We next improve the estimate for $\tau_i$. Applying the
	Pohozaev identity \eqref{3.6.1} with $D_1(\xi_i)$ replaced by
	$D_{1/2}(\xi_i)$, we have
	\[
	\begin{aligned}
	\tau_i
	\leq
	Cu_i(\xi_i)^{-2}
	+
	C
	\int_{D_{1/2}(\xi_i)}
	|\mathcal X_i(a_i)(\xi)|u_i(\xi)^2\,\mathrm d\xi.
	\end{aligned}
	\]

	By the definition of $\mathcal X_i$,
	\[
	\begin{aligned}
	|\mathcal X_i(a_i)(\xi)|
	&\leq
	Cd_{\mathbb H^n}(\xi,\xi_i)
	|\nabla_{\mathbb H^n}a_i(\xi)|+
	Cd_{\mathbb H^n}(\xi,\xi_i)^2
	|\partial_ta_i(\xi)|.
	\end{aligned}
	\]
	The Taylor formula on the Heisenberg group give
	\[
	\begin{aligned}
	|\nabla_{\mathbb H^n}a_i(\xi)|
	&\leq
	|\nabla_{\mathbb H^n}a_i(\xi_i)|+
	C
	\|\nabla_{\mathbb H^n}^2a_i\|_{L^\infty(D_1)}
	d_{\mathbb H^n}(\xi,\xi_i),
	\end{aligned}
	\]
	and
	\[
	|\partial_ta_i(\xi)|
	\leq
	C
	\|\nabla_{\mathbb H^n}^2a_i\|_{L^\infty(D_1)}.
	\]
	Consequently,
	\[
	\begin{aligned}
	|\mathcal X_i(a_i)(\xi)|
	&\leq
	Cd_{\mathbb H^n}(\xi,\xi_i)
	|\nabla_{\mathbb H^n}a_i(\xi_i)|+
	C\|\nabla_{\mathbb H^n}^2a_i\|_{L^\infty(D_1)}
	d_{\mathbb H^n}(\xi,\xi_i)^2.
	\end{aligned}
	\]
	Hence,
	\[
	\begin{aligned}
	\tau_i
	&\leq
	Cu_i(\xi_i)^{-2}+
	C|\nabla_{\mathbb H^n}a_i(\xi_i)|
	\int_{D_{1/2}(\xi_i)}
	d_{\mathbb H^n}(\xi,\xi_i)u_i(\xi)^2
	\,\mathrm d\xi\\
	&\quad+
	C
	\|\nabla_{\mathbb H^n}^2a_i\|_{L^\infty(D_1)}
	\int_{D_{1/2}(\xi_i)}
	d_{\mathbb H^n}(\xi,\xi_i)^2u_i(\xi)^2
	\,\mathrm d\xi.
	\end{aligned}
	\]
    Thus, combining Lemma \ref{lem 2.4 in jde} and the estimate of  $|\nabla_{\mathbb H^n}a_i(\xi_i)|$, we have
	\[
	\tau_i
	\leq
	C
	\begin{cases}
	u_i(\xi_i)^{-2}
	(
	1+
	\|\nabla_{\mathbb H^n}^2a_i\|_{L^\infty(D_1)}
	)
	&\text{if }Q=4,\\[1ex]
	u_i(\xi_i)^{-2}
	(
	1+
	\|\nabla_{\mathbb H^n}^2a_i\|_{L^\infty(D_1)}
	\ln u_i(\xi_i)
	)
	&\text{if }Q=6,\\[1ex]
	u_i(\xi_i)^{-2}
	(
	1+
	\|\nabla_{\mathbb H^n}^2a_i\|_{L^\infty(D_1)}
	u_i(\xi_i)^{\frac{2Q-12}{Q-2}}
	)
	&\text{if }Q\geq8.
	\end{cases}
	\]

    Using Pohozaev identity \eqref{3.6.1}, the estimates for $\left|\nabla_{\mathbb{H}^n} a_i\left(\xi_i\right)\right|$, 
	 Lemma \ref{lem 2.5 in jde} and \eqref{54 in imrn}, the estimates of $a_i\left(\xi_i\right)$ follows immediately.

\end{proof}

\section{Expansions of Blow Up Solutions}\label{section 5}

In the following we will adapt some arguments from Marques \cite{Marques} for the Yamabe equation; see also Li-Zhang \cite{Li Zhang}, Niu-Peng-Xiong \cite{Niu Peng Xiong JFA}, and Jin-Li-Xiong \cite{Li Xiong}.

\begin{lemma}\label{lem 4.2 in imrn}
	Assume as in Lemma \ref{lemma 2.2 in jde}. Suppose $\rho=1$, we have,
	$$
	\begin{aligned}
	\left|\Phi_i(\xi)-\Lambda_{0,1}(\xi)\right| \leq C 
			\begin{cases}
				u_i(\xi_i)^{-2} (1+\|\nabla_{\mathbb{H}^n}^2 a_i\|_{L^{\infty}(D_1)}) & \text{ if } Q=4, \\
				u_i(\xi_i)^{-2} (1+\|\nabla_{\mathbb{H}^n}^2 a_i\|_{L^{\infty}(D_1)}\ln u_i(\xi_i)) & \text{ if } Q=6, \\
				u_i(\xi_i)^{-2} (1+\|\nabla_{\mathbb{H}^n}^2 a_i\|_{L^{\infty}(D_1)}u_i(\xi_i)^{\frac{2Q-12}{Q-2}}) \ &\text{ if }Q\geq 8 .
			\end{cases}
	\end{aligned}
	$$
where $\Phi_i(\xi):=u_i(\xi_i)^{-1} u_i(\xi_i \circ \delta_{u_i(\xi_i)^{-(p_i-1) / 2}} \xi)$, and $C>0$ depends only on $n$, $A_0$ and $A_1$.
\end{lemma}

\begin{proof}
	For brevity, let $\ell_i:=u_i(\xi_i)^{(p_i-1)/2}$. By the equation which $u_i$ satisfies, we have
	\be\label{4.48 in imrn}
	-\Delta_{\mathbb{H}^{n}}\Phi_i(\xi)  =\ell_i^{-2}\tilde{a}_i(\xi)\Phi_i(\xi)+\Phi_i(\xi)^{p_i} ,
	\ee
	where $\widetilde{a}_i(\xi)=a_i(\xi_i \circ \delta_{\ell_i^{-1}} \xi)$. 

	Let
	\be\label{4.52 in imrn}
	T_i(\xi)=\Lambda_{0, 1}(\xi)^{\frac{Q+2}{Q-2 }}-\Lambda_{0, 1}(\xi)^{p_i}
	\ee
	and
	\be\label{V in imrn}
	V_i(\xi):=\frac{\Phi_i(\xi)-\Lambda_{0, 1}(\xi)}{\varphi_i},
	\ee
	where 
	$$
	\varphi_i=\max _{\|\xi\|_{\mathbb{H}^n} \leq \ell_i}|\Phi_i(\xi)-\Lambda_{0, 1}(\xi)| .
	$$
	By Lemma \ref{Harnack ineq 2} and Proposition \ref{pro 2.3 in jde}, we have for any $0<\varepsilon<1$ and $\varepsilon \ell_i\leq\|\xi\|_{\mathbb{H}^n} \leq \ell_i$,
	$$
	|\Phi_i(\xi)-\Lambda_{0, 1}(\xi)| \leq |\Phi_i(\xi)|+|\Lambda_{0, 1}(\xi)| \leq C(\varepsilon) u_i(\xi_i)^{-2},
	$$
	where we use $u_i(\xi_i)^{\tau_i}=1+o(1)$. Hence, we may assume that $\varphi_i$ is achieved at some point $\|\zeta_i\|_{\mathbb{H}^n} \leq \ell_i/2$, otherwise the proof is finished. In fact, taking \(\varepsilon=1/2\), if the maximum defining
\(\varphi_i\) is attained in
\(D_{\ell_i}\setminus D_{\ell_i/2}\), then
\[
\varphi_i
\leq
C u_i(\xi_i)^{-2}
\leq
C\alpha_i,
\]
and the desired estimate follows. 
	
	By \eqref{4.48 in imrn}, we have
	\be\label{4.54(1) in imrn}
	-\Delta_{\mathbb{H}^{n}}V_i(\xi)  =\frac{1}{\varphi_i}(\ell_i^{-2}\tilde{a}_i(\xi)\Phi_i(\xi)+\Phi_i(\xi)^{p_i}-\Lambda_{0, 1}(\xi)^{p_i}-T_i(\xi)) .
	\ee
	Denoting that 
	\be\label{4.55 in imrn}
	c_i(\xi)=\frac{\Phi_i(\xi)^{p_i}-\Lambda_{0, 1}(\xi)^{p_i}}{\Phi_i(\xi)-\Lambda_{0, 1}(\xi)},
	\ee
	we have
	\be\label{4.54 in imrn}
	-\Delta_{\mathbb{H}^{n}}V_i(\xi)  =F_i+c_iV_i(\xi)\qquad\text{in }D_{\ell_i}
	\ee
    where $F_i=\frac{\ell_i^{-2}\widetilde{a}_i(\xi)\Phi_i(\xi)}{\varphi_i} -\frac{T_i(\xi)}{\varphi_i}$.

Let \(G_i(\xi,\eta)\) be the Dirichlet Green function of
\(-\Delta_{\mathbb H^n}\) in \(D_{\ell_i}\); namely, for each
fixed \(\eta\in D_{\ell_i}\),
\[
\begin{cases}
-\Delta_{\mathbb H^n,\xi}G_i(\xi,\eta)=\delta_\eta
&\text{in }D_{\ell_i},\\
G_i(\xi,\eta)=0
&\text{on }\partial D_{\ell_i}.
\end{cases}
\]

Let \(H_i\) be the harmonic extension of the boundary values of
\(V_i\), namely,
\[
\begin{cases}
-\Delta_{\mathbb H^n}H_i=0
&\text{in }D_{\ell_i},\\
H_i=V_i
&\text{on }\partial D_{\ell_i}.
\end{cases}
\]
By Green's representation formula, we have
\begin{equation}\label{eq:green 3}
\begin{aligned}
V_i(\xi)
&=
H_i(\xi)
+\int_{D_{\ell_i}}
G_i(\xi,\eta)
\bigl(
F_i(\eta)+c_i(\eta)V_i(\eta)
\bigr)
\mathrm{~d}\eta
\\
&=
H_i(\xi)
+\int_{D_{\ell_i}}
G_i(\xi,\eta)
\left(
\frac{\ell_i^{-2}\tilde a_i(\eta)\Phi_i(\eta)}
{\varphi_i}
-\frac{T_i(\eta)}{\varphi_i}
+c_i(\eta)V_i(\eta)
\right)
\mathrm{~d}\eta .
\end{aligned}
\end{equation}

By the second-order Taylor estimate on \(\mathbb H^n\)(see \cite{BEU}), and using
the fact that the \(t\)-derivative is controlled by ordered
second horizontal derivatives, we have
	\be\label{eq:Taylor}
	\begin{aligned}
		\tilde{a}_i(\eta):&=a_i(\xi_i \circ \delta_{\ell_i^{-1}} \eta) \\
		&\leq a_i(\xi_i)+\ell_i^{-1} \nabla_{\mathbb{H}^n} a_i(\xi_i)(\bar{x},\bar{y})+\ell_i^{-2}\partial_t a_i(\xi_i)t\\
		& \quad+\ell_i^{-2}\|\nabla^2_{\mathbb{H}^n} a_i\|_{L^{\infty}(D_1)}(|\bar{x}|^2+|\bar{y}|^2) \\
		&\leq a_i(\xi_i)+\ell_i^{-1} |\nabla_{\mathbb{H}^n} a_i(\xi_i)|\|\eta\|_{\mathbb{H}^n}+\ell_i^{-2}\|\nabla^2_{\mathbb{H}^n} a_i\|_{L^{\infty}(D_1)}\|\eta\|_{\mathbb{H}^n}^2.		
	\end{aligned}
	\ee
	Since
\[
0<\Phi_i(\eta)
\leq
C\Lambda_{0,1}(\eta)
\leq
C(1+\|\eta\|_{\mathbb H^n})^{2-Q},
\]
it follows from \eqref{eq:Taylor} and Lemma \ref{lem 2.6 in jde} that
	\be\label{estimate 1}
	\int_{D_{\ell_i}}  G_i(\xi,\eta)  \ell_i^{-2}\tilde{a} (\eta)\Phi_i(\eta) \leq C \alpha_i,
	\ee
	where
	\be\label{65 in imrn}
	\begin{aligned}
		\alpha_i=
			\begin{cases}
				u_i(\xi_i)^{-2} (1+\|\nabla_{\mathbb{H}^n}^2 a_i\|_{L^{\infty}(D_1)}) & \text{ if } Q=4, \\
				u_i(\xi_i)^{-2} (1+\|\nabla_{\mathbb{H}^n}^2 a_i\|_{L^{\infty}(D_1)}\ln u_i(\xi_i)) & \text{ if } Q=6, \\
				u_i(\xi_i)^{-2} (1+\|\nabla_{\mathbb{H}^n}^2 a_i\|_{L^{\infty}(D_1)}u_i(\xi_i)^{\frac{2Q-12}{Q-2}}) \ &\text{ if }Q\geq 8 .
			\end{cases}
	\end{aligned}
	\ee

	Meanwhile, we have	
	\be\label{estimate 2}
	c_i(\eta)\leq C \Lambda_{0, 1}(\eta)^{p_i-1}
	\ee
	and
	\be\label{estimate 3}
	|T_i(\eta)| \leq C \tau_i|\log \Lambda_{0, 1}(\eta)|\Lambda_{0, 1}(\eta)^{p_i}.
	\ee

We claim that
\[
\varphi_i\leq C\alpha_i.
\]
Suppose otherwise. After passing to a subsequence, we may assume
that
\be\label{contradiction alpha phi}
\frac{\alpha_i}{\varphi_i}\rightarrow0.
\ee

Since \(\|V_i\|_{L^\infty(D_{\ell_i})}=1\), it follows from
\eqref{estimate 1}, \eqref{estimate 2}, \eqref{estimate 3},
and the corresponding potential estimates that
\be\label{estimate Wi}
|V_i(\xi)-H_i(\xi)|
\leq
\frac{C}{1+\|\xi\|_{\mathbb H^n}}
+
C\frac{\alpha_i+\tau_i}{\varphi_i}
\quad \text{ for }
\xi\in D_{\ell_i/2}.
\ee

On the other hand, taking \(\varepsilon=1/2\) in the estimate on
the outer annulus, we have
\[
\sup_{\eta\in\partial D_{\ell_i/2}}
|V_i(\eta)|
\leq
C\frac{u_i(\xi_i)^{-2}}{\varphi_i}.
\]
Therefore, by \eqref{estimate Wi},
\[
\begin{aligned}
\sup_{\eta\in\partial D_{\ell_i/2}}
|H_i(\eta)|
&\leq
\sup_{\eta\in\partial D_{\ell_i/2}}
|V_i(\eta)|
+
\sup_{\eta\in\partial D_{\ell_i/2}}
|V_i(\eta)-H_i(\eta)|
\\
&\leq
C\ell_i^{-1}
+
C\frac{
u_i(\xi_i)^{-2}+\alpha_i+\tau_i
}{\varphi_i}.
\end{aligned}
\]
Applying the maximum principle to \(H_i\) and \(-H_i\) in
\(D_{\ell_i/2}\), we obtain
\be\label{Vi minus Wi small}
\begin{aligned}
\|H_i\|_{L^\infty(D_{\ell_i/2})}\leq
C\ell_i^{-1}+C\frac{u_i(\xi_i)^{-2}+\alpha_i+\tau_i}{\varphi_i}.
\end{aligned}
\ee
Since
\[
u_i(\xi_i)^{-2}\leq\alpha_i,
\qquad
\tau_i\leq C\alpha_i,
\]
and \(\ell_i\to\infty\), it follows from
\eqref{contradiction alpha phi} that
\be\label{Vi minus Wi o1}
\|H_i\|_{L^\infty(D_{\ell_i/2})}
=o(1).
\ee

Combining \eqref{estimate Wi} and
\eqref{Vi minus Wi o1}, we obtain
\be\label{pointwise estimate for Vi}
|V_i(\xi)|
\leq
o(1)
+
\frac{C}{1+\|\xi\|_{\mathbb H^n}}
\quad \text{ for }
\xi\in D_{\ell_i/2},
\ee
where \(o(1)\) is independent of
\(\xi\in D_{\ell_i/2}\).

 Let \(R>0\) be fixed.
Since \(\ell_i\to\infty\), we have
$
D_R\subset D_{\ell_i/2}
$
for all sufficiently large \(i\). By the  estimate
\eqref{eq:Taylor} and  \eqref{estimate 3}, we have
\be\label{Fi converges to zero}
\|F_i\|_{L^\infty(D_R)}
\leq
C_R\frac{\alpha_i+\tau_i}{\varphi_i}
\rightarrow0.
\ee

On the other hand, since
\[
\|V_i\|_{L^\infty(D_{\ell_i})}=1,
\]
the local subelliptic estimates applied to
\[
-\Delta_{\mathbb H^n}V_i
=
F_i+c_iV_i
\]
show that, after passing to a subsequence,
\[
V_i\rightarrow V
\qquad\text{in }C_{\mathrm{loc}}^2(\mathbb H^n),
\]
where \(V\) satisfies
\be\label{linearized equation limit V}
-\Delta_{\mathbb H^n}V
=
\frac{Q+2}{Q-2}
\Lambda_{0,1}^{\frac{4}{Q-2}}V
\qquad\text{in }\mathbb H^n.
\ee

By
\[
\Phi_i(0)=\Lambda_{0,1}(0)=1,
\]
we have
\[
V_i(0)=0.
\]
Moreover, since \(\xi_i\) is a local maximum point of \(u_i\), we have
\[
\nabla_{\mathbb H^n}\Phi_i(0)=0,
\qquad
\partial_t\Phi_i(0)=0,
\]
It follows that
\[
\nabla_{\mathbb H^n}V_i(0)=0,
\qquad
\partial_tV_i(0)=0.
\]
Passing to the limit, we obtain
\be\label{normalization limit V}
V(0)=0,
\qquad
\nabla_{\mathbb H^n}V(0)=0,
\qquad
\partial_tV(0)=0.
\ee

Meanwhile, applying Green's representation formula, \eqref{pointwise estimate for Vi}, and then iterating the potential estimate finitely many times, we obtain
\be\label{sharp decay limit V}
|V(\xi)|
\leq
C(1+\|\xi\|_{\mathbb H^n})^{2-Q}
\qquad\text{in }\mathbb H^n.
\ee

It follows from the non-degeneracy result in \cite{Mal} (see also \cite{Qiang}) that
	$$
	V(\xi)=\bigg[\sum_{j=1}^n (c_j \overline{X}_j \Lambda_{\xi_0,\lambda}(\xi)+c_{j+n} \overline{Y}_j \Lambda_{\xi_0,\lambda}(\xi))+c_{2n+1}\partial_t \Lambda_{\xi_0,\lambda}(\xi) +c_{2n+2}\partial_\lambda \Lambda_{\xi_0,\lambda}(\xi)\bigg]\bigg|_{\xi_0=0,\lambda=1},
	$$
where $c_j$ for $j=1,\cdots, 2n+2$ are constants. Therefore,
\[
V\equiv0
\qquad\text{in }\mathbb H^n.
\]

Recall that, by the choice of \(\zeta_i\),
\[
|V_i(\zeta_i)|=1,
\qquad
\|\zeta_i\|_{\mathbb H^n}\leq\frac{\ell_i}{2}.
\]
We claim that
\[
\|\zeta_i\|_{\mathbb H^n}\rightarrow\infty.
\]
Indeed, otherwise, after passing to a subsequence,
\[
\zeta_i\rightarrow\zeta_\infty
\qquad\text{in }\mathbb H^n.
\]
The local uniform convergence \(V_i\to V\equiv0\) would then give
\[
1
=
\lim_{i\to\infty}|V_i(\zeta_i)|
=
|V(\zeta_\infty)|
=
0,
\]
which is impossible.

Finally, evaluating \eqref{pointwise estimate for Vi} at
\(\xi=\zeta_i\), we obtain
\[
1
=
|V_i(\zeta_i)|
\leq
o(1)
+
\frac{C}{1+\|\zeta_i\|_{\mathbb H^n}}
\rightarrow0,
\]
again a contradiction. Therefore,
\[
\varphi_i\leq C\alpha_i.
\]
This completes the proof.

	\end{proof}

	\begin{lemma}\label{lem 4.2}
		Assume as in Lemma \ref{lem 4.2 in imrn}, we have, for every $\|\xi\|_{\mathbb{H}^n} \leq \ell_i$,
		$$
		\begin{aligned}
			& |\Phi_i(\xi)-\Lambda_{0,1}(\xi)| \\
			& \leq C \begin{cases}u_i(\xi_i)^{-2} & \text { if } Q=4, \\ 
				\max \{u_i(\xi_i)^{-2}, (1+\|\xi\|_{\mathbb{H}^n})^{-1}\|\nabla_{\mathbb{H}^n}^2 a_i\|_{L^{\infty}\left(D_1\right)} u_i\left(\xi_i\right)^{-2+2 /(Q-2)}\} & \text { if } Q=6, \\ \max \{u_i(\xi_i)^{-2},
				(1+\|\xi\|_{\mathbb{H}^n})^{6-Q}\|\nabla_{\mathbb{H}^n}^2 a_i\|_{L^{\infty}\left(D_1\right)} u_i\left(\xi_i\right)^{-2+(2 Q-12) /(Q-2)} \}& \text { if } Q\geq 8,\end{cases}
		\end{aligned}
		$$
	where $C>0$ depends only on $n$, $A_0$ and $A_1$.
	\end{lemma}
			
	\begin{proof}
		In the case $Q = 4$, the conclusion is obvious. We will only prove the case when $Q\geq 6$ in the following. Assume that $u_i(\xi_i)^{-2} / \alpha_i \rightarrow 0$ as $i \rightarrow \infty$; otherwise, there exists a subsequence $i_l$ of $\{i\}$
		 such that $u_{i_l}(\xi_{i_l})^{-2} \geq C \alpha_{i_l}$ for some $C>0$ and the lemma follows from Lemma \ref{lem 4.2 in imrn}. 
         Define
		$$
		\begin{aligned}
			\alpha^{\prime}_i=  \begin{cases}u_i(\xi_i)^{-2} & \text { if } Q=4, \\ 
				\|\nabla_{\mathbb{H}^n}^2 a_i\|_{L^{\infty}\left(D_1\right)} u_i\left(\xi_i\right)^{-2+2 /(Q-2)} & \text { if } Q=6, \\ 
				\|\nabla_{\mathbb{H}^n}^2 a_i\|_{L^{\infty}\left(D_1\right)} u_i\left(\xi_i\right)^{-2+(2 Q-12) /(Q-2)} & \text { if } Q\geq 8.\end{cases}
		\end{aligned}		
		$$
         and let
		$$
		V_i^{\prime}(\xi):=\frac{\Phi_i(\xi)-\Lambda_{0,1}(\xi)}{\alpha_i^{\prime}} \quad \text{ for }\|\xi\|_{\mathbb{H}^n} \leq \ell_i,
		$$
		where $\ell_i=u_i(\xi_i)^{(p_i-1)/2}$. Since
\[
\frac{u_i(\xi_i)^{-2}}{\alpha_i}\rightarrow0
\]
and \(\alpha_i\leq C\alpha_i'\), it follows from Lemma
\ref{lem 4.2 in imrn} that
\[
|V_i'(\xi)|\leq C
\qquad\text{for }
\|\xi\|_{\mathbb H^n}\leq\ell_i.
\]
Since $0<\Phi_i\leq C\Lambda_{0,1}$ and $u_i(\xi_i)^{\tau_i}=1+o(1)$, we have
\[
|V_i'(\xi)|
\leq
\begin{cases}
C u_i(\xi_i)^{-\frac{2}{Q-2}},
&\text{if }Q=6,\\[2mm]
C u_i(\xi_i)^{-\frac{2(Q-6)}{Q-2}},
&\text{if }Q\geq8.
\end{cases}
\]
Thus we only need to prove the proposition when $\|\xi\|_{\mathbb{H}^n} \leq \ell_i / 2$. 

          By a calculation similar to \eqref{4.54 in imrn}, $V_i^{\prime}(\xi)$ satisfies
		\be
		-\Delta_{\mathbb{H}^{n}}V_i^{\prime}(\xi)  =\frac{\ell_i^{-2}\tilde{a}(\xi)\Phi_i(\xi)}{\alpha_i^{\prime}} -\frac{T_i(\xi)}{\alpha_i^{\prime}}+c_iV_i^{\prime}(\xi),
		\ee
		where $\tilde{a}_i$, $T_i$ and $c_i$ are given by \eqref{4.48 in imrn}, \eqref{4.52 in imrn} and \eqref{4.55 in imrn}, respectively. By Green representation, we have
		\begin{equation}\label{eq:green 4}
		\begin{split}
			V_i^\prime (\xi)
			=H_i(\xi)+\int_{D_{\ell_i}}G_i( \xi,\eta) \Big(\frac{\ell_i^{-2}\tilde{a}(\eta)\Phi_i(\eta)}{\alpha^{\prime}_i}-\frac{T_i(\eta)}{\alpha^{\prime}_i}+c_i(\eta)V^{\prime}_i(\eta)\Big) \ \mathrm{d} \bar{z}\mathrm{d} \bar{t}.
		\end{split}
	\end{equation}
	By Taylor expansion of $a_i$ at $\xi_i$, we have
	$$
	\begin{aligned}
		a_i(\xi_i\circ \delta_{\ell_i^{-1}} \eta) \leq 
		a_i(\xi_i)+\ell_i^{-1} \|\nabla_{\mathbb{H}^n} a_i(\xi_i)\|_{L^{\infty}(D_1)}\|\eta\|_{\mathbb{H}^n}
		+\ell_i^{-2 }\|\nabla^2_{\mathbb{H}^n} a_i\|_{L^{\infty}(D_1)}\|\eta\|^2_{\mathbb{H}^n}.		
	\end{aligned}
	$$
	Since $\Phi_i(\eta) \leq C \Lambda_{0,1}(\eta)$, similar to Lemma \ref{lem 2.6 in jde} and Lemma 4.1 in \cite{Niu Tang Zhou IMRN}, we have the following: for $Q=6$,
	$$
	\begin{aligned}
	\int_{D_{\ell_i}}G_i( \xi,\eta)\frac{  \ell_i^{-2}\tilde{a} (\eta) \Phi_i(\eta)}{\alpha_i^\prime}\ \mathrm{d} \bar{z}\mathrm{d} \bar{t}&\leq C\int_{D_{\ell_i}}\frac{u_i(\xi_i)^{\frac{-2}{Q-2}}}{\|\eta^{-1}\circ\xi\|^{Q-2}_{\mathbb{H}^n}(1+\|\eta\|_{\mathbb{H}^n})^{Q-4}}\ \mathrm{d} \bar{z}\mathrm{d} \bar{t}\\
	&\leq C\int_{D_{\ell_i}}\frac{1}{\|\eta^{-1}\circ\xi\|^{Q-2}_{\mathbb{H}^n}(1+\|\eta\|_{\mathbb{H}^n})^{Q-3}}\ \mathrm{d} \bar{z}\mathrm{d} \bar{t}\\
	&\leq C(1+\|\xi\|_{\mathbb{H}^n})^{-1};
	\end{aligned}
	$$
for $Q\geq 8$,
	$$
	\begin{aligned}
	\int_{D_{\ell_i}}G_i( \xi,\eta) \frac{ \ell_i^{-2}\tilde{a} (\eta) \Phi_i(\eta)}{\alpha^\prime_i}\ \mathrm{d} \bar{z}\mathrm{d} \bar{t}
	&\leq C\int_{D_{\ell_i}}\frac{1}{\|\eta^{-1}\circ\xi\|^{Q-2}_{\mathbb{H}^n}(1+\|\eta\|_{\mathbb{H}^n})^{Q-4}}\ \mathrm{d} \bar{z}\mathrm{d} \bar{t}\\
	&\leq C(1+\|\xi\|_{\mathbb{H}^n})^{6-Q}.
	\end{aligned}
	$$
	On the other hand, it follows from Lemma \ref{lem 4.2 in imrn} that
	$$c_i(\xi)\leq C \Lambda_{0, 1}(\xi)^{p_i-1}\leq C(\bar c_n^2 t^2+(1+\bar c_n|z|^2)^2)^{O(\tau_i)-1}$$
	and
	$$
	|T_i(\xi)| \leq C \tau_i|\log \Lambda_{0, 1}(\xi)|(\bar c_n^2 t^2+(1+\bar c_n|z|^2)^2)^{\frac{-Q-2}{4} }.
	$$
	Repeating the process of Lemma \ref{lem 4.2 in imrn}, we complete the proof.
	\end{proof}

    Considering the equation of \(W_i:=\Phi_i-\Lambda_{0,1}\), we obtain the following corollary. In fact, The horizontal gradient estimate follows by differentiating the Green representation formula and using
the standard pointwise estimates for the Green kernel on
\(\mathbb H^n\). Combining these estimates with the scale-invariant
interior subelliptic Schauder estimate and the identity
\([X_j,Y_j]=-4\partial_t\), we also obtain the estimate for
\(\partial_tW_i\). The argument for the horizontal derivatives is
parallel to the Euclidean case
\cite{Marques,Niu Peng Xiong JFA,Niu Tang Zhou IMRN}, with \(Q\)
replacing \(n\).
	\begin{corollary}\label{corollary 4.4 in imrn}
		Assume as in Lemma \ref{lem 4.2 in imrn}, we have
		$$
		\begin{aligned}
			& |\nabla_{\mathbb{H}^n}(\Phi_i(\xi)-\Lambda_{0,1}(\xi))| \leq C(1+\|\xi\|_{\mathbb{H}^n})^{-1} \\
			&\quad  \times \begin{cases}u_i(\xi_i)^{-2} & \text { if }  Q=4, \\
				\max \{u_i(\xi_i)^{-2}, \left\|\nabla^2_{\mathbb{H}^n} a_i\right\|_{L^{\infty}(D_1)} u_i(\xi_i)^{-2+2 /(Q-2 )}(1+\|\xi\|_{\mathbb{H}^n})^{-1}\} & \text { if } Q=6, \\
				\max \{u_i(\xi_i)^{-2},\|\nabla^2_{\mathbb{H}^n} a_i\|_{L^{\infty}(D_1)} u_i(\xi_i)^{-2+(2 Q-12) /(Q-2 )}(1+\|\xi\|_{\mathbb{H}^n})^{6-Q}\} & \text { if } Q\geq 8,\end{cases}
		\end{aligned}
	$$
	and
	$$
		\begin{aligned}
			& |\partial_t(\Phi_i(\xi)-\Lambda_{0,1}(\xi))| \leq C(1+\|\xi\|_{\mathbb{H}^n})^{-2} \\
			&\quad  \times \begin{cases}u_i(\xi_i)^{-2} & \text { if }  Q=4, \\
				\max \{u_i(\xi_i)^{-2}, \left\|\nabla^2_{\mathbb{H}^n} a_i\right\|_{L^{\infty}(D_1)} u_i(\xi_i)^{-2+2 /(Q-2 )}(1+\|\xi\|_{\mathbb{H}^n})^{-1}\} & \text { if } Q=6, \\
				\max \{u_i(\xi_i)^{-2},\|\nabla^2_{\mathbb{H}^n} a_i\|_{L^{\infty}(D_1)} u_i(\xi_i)^{-2+(2 Q-12) /(Q-2 )}(1+\|\xi\|_{\mathbb{H}^n})^{6-Q}\} & \text { if } Q\geq 8,\end{cases}
		\end{aligned}
	$$
	where $C>0$ depends only on $n$, $A_0$ and $A_1$.
	\end{corollary}

\section{Local results}\label{section 6}
	In this section we prove some local results regarding isolated blow-up points, namely that an isolated blow-up point is a critical point for the function $a=\lim _{i \rightarrow \infty} a_i$ and we give sufficient conditions for an isolated blow-up point to be an isolated simple blow-up point.

	\begin{proposition}\label{pro 5.1 in imrn}
		Assume as in Lemma \ref{lemma 2.2 in jde}. Assume further that $\left\|a_i\right\|_{C^4(D_1)} \leq A_0$. Then for $0<r<\rho$ there holds
		$$
		\begin{aligned}
			& u_i(\xi_i)^2 \int_{\partial D_r(\xi_i)}\mathcal{D}( \xi_i^{-1}\circ \xi, u_i, \nabla_{\mathbb{H}^{n}} u_i)\geq-C_0 r^{-Q} u_i(\xi_i)^{-\frac{4 }{Q-2 }}-C_0\|a_i\|_{L^{\infty}(D_1)} r^{4-Q} \\
			& \quad +C_1 \begin{cases} a_i(\xi_i) \ln (r u_i(\xi_i)^{\frac{2}{Q-2 }}) & \text{ if }Q=4, \\
			a_i(\xi_i) u_i(\xi_i)^{\frac{2(Q-4 )}{Q-2 }}+\frac{1}{2 n} \Delta_{\mathbb{H}^n} a_i(\xi_i) \ln (r u_i(\xi_i)^{\frac{2}{Q-2 }}) & \text{ if }Q=6, \\
			\beta_i-C_0\|a_i\|_{D_1} \ln (r u_i(\xi_i)^{\frac{2}{Q-2 }}) & \text{ if }Q=8, \\
			\beta_i-C_0\|a_i\|_{D_1} u_i(\xi_i)^{\frac{2(Q-8)}{Q-2 }} & \text{ if }Q\geq 10, \end{cases}
		\end{aligned}
		$$
		where 
		\begin{equation*}
			\begin{aligned}
					\beta_i&:= a_i(\xi_i) u_i(\xi_i)^{\frac{2(Q-4 )}{Q-2 }}+\frac{1}{2n} \Delta_{\mathbb{H}^n} a_i(\xi_i) u_i(\xi_i)^{\frac{2(Q-6)}{Q-2 }},\\ \|a_i\|_{D_1}&:=\|a_i\|_{L^{\infty}(D_1)} \times\|\nabla_{\mathbb{H}^n}^2 a_i\|_{L^{\infty}(D_1)}+\|\nabla_{\mathbb{H}^n}^4 a_i\|_{L^{\infty}\left(D_1\right)},
			\end{aligned}\end{equation*} 
			$C_0>0$ depends only on $n$,  $A_0$, $A_1$, $\rho$ and is independent of $r$ if $i$ is sufficiently large, and $C_1>0$ depends only on $n$.
		\end{proposition}
			
		\begin{proof}
			The corresponding Pohozaev identity centered at $\xi_i$ is
			$$
			\begin{aligned}
				& \int_{\partial D_r(\xi_i)} \mathcal{D}(\xi_i^{-1} \circ \xi, u_i, \nabla_{\mathbb{H}^{n}} u_i)\ \mathrm{d} \mathcal{H}_{Q-2}\\
				&=\frac{1}{2} \int_{D_r(\xi_i)} \mathcal{X}_i( a_i) u_i^{2}\ \mathrm{d} z \mathrm{d} t
				+\Big(\frac{Q}{p_i+1}-\frac{Q-2}{2}\Big) \int_{D_r(\xi_i)}  u_i^{p_i+1}\ \mathrm{d} z \mathrm{d} t +\int_{D_r(\xi_i)}  a_i u_i^{2}\ \mathrm{d} z \mathrm{d} t\\
				& \quad -\frac{1}{2} \int_{\partial D_r(\xi_i)}a_i u_i^2 \mathcal{X}_i I\cdot N\ \mathrm{d} \mathcal{H}_{Q-2}-\frac{1}{p_i+1} \int_{\partial D_r(\xi_i)} u_i^{p_i+1} \mathcal{X}_i I \cdot N\ \mathrm{d} \mathcal{H}_{Q-2},
			\end{aligned}
			$$
			where $\mathcal{X}_i $ is defined as in Lemma \ref{lemma 5.9 in asan}.
			Note that
			$	\mathcal{X}_i I\cdot N=\nu_i \cdot N.$
			Thus, we get
			$$
			\begin{aligned}
				&\quad \mathcal{G}_i(r)\\&:=  \frac{1}{2} \int_{D_r(\xi_i)} \mathcal{X}_i( a_i) u_i^{2}\ \mathrm{d} z \mathrm{d} t
				+\int_{D_r(\xi_i)}  a_i u_i^{2}\ \mathrm{d} z \mathrm{d} t -\frac{1}{2} \int_{\partial D_r(\xi_i)}a_i u_i^2 \mathcal{X}_i I\cdot N \ \mathrm{d} \mathcal{H}_{Q-2}\\
				&=  \frac{1}{2} \int_{D_r(\xi_i)} (\nu_i\cdot\nabla a_i) u_i^{2}\ \mathrm{d} z \mathrm{d} t
				+\int_{D_r(\xi_i)}  a_i u_i^{2}\ \mathrm{d} z \mathrm{d} t -\frac{1}{2} \int_{\partial D_r(\xi_i)}a_i u_i^2 \mathcal{X}_i I\cdot N \ \mathrm{d} \mathcal{H}_{Q-2}\\
				&=  \frac{1}{2} \int_{\partial D_r(\xi_i)}  a_i u_i^{2}\nu_i\cdot N\ \mathrm{d} \mathcal{H}_{Q-2}
				-\frac{1}{2} \int_{D_r(\xi_i)} Q a_i u_i^{2}\ \mathrm{d} z \mathrm{d} t\\
				&\quad -\frac{1}{2} \int_{D_r(\xi_i)} 2a_i u_i\nu_i\cdot \nabla u_i\ \mathrm{d} z \mathrm{d} t
				+\int_{D_r(\xi_i)}  a_i u_i^{2}\ \mathrm{d} z \mathrm{d} t -\frac{1}{2} \int_{\partial D_r(\xi_i)}a_i u_i^2 \mathcal{X}_i I\cdot N \ \mathrm{d} \mathcal{H}_{Q-2}\\
				&=\frac{2-Q}{2}\int_{ D_r(\xi_i)}a_i u_i^2\ \mathrm{d} z \mathrm{d} t-\int_{ D_r(\xi_i)}a_i u_i \mathcal{X}_i(u_i)\ \mathrm{d} z \mathrm{d} t.
			\end{aligned}
			$$	
			By a change of variables $\bar{\xi}=(\bar{z} ,\bar{t})=\ell_i(\xi_i^{-1}\circ \xi)$ with $\ell_i=u_i(\xi_i)^{(p_i-1) / 2 }$, we have
			$$
			\begin{aligned}
				\mathcal{G}_i(r)=-u_i(\xi_i)^{2-Q(p_i-1) / 2} \int_{D_{\ell_i r}}(\bar{\mathcal{X}}_i(\Phi_i)+\frac{Q-2}{2} \Phi_i) a_i(\xi_i\circ \delta_{\ell_i^{-1}} \bar{\xi}) \Phi_i\ \mathrm{d} \bar{z} \mathrm{d} \bar{t},
			\end{aligned}
			$$
			where $\Phi_i(\bar{\xi})=u_i(\xi_i)^{-1} u_i(\xi_i\circ \delta_{\ell_i^{-1}} \bar{\xi})$ and $\bar{\mathcal{X}}_i(\Phi_i)=(\bar{x}, \bar{y}, 2\bar{t})\nabla \Phi_i$.
			Let
			$$
			\begin{aligned}
				\mathcal{\hat{G}}_i(r)=-u_i(\xi_i)^{2-Q(p_i-1) / 2} \int_{D_{\ell_i r}}(\bar{\mathcal{X}}_i(\Lambda_{0, 1})+\frac{Q-2}{2} \Lambda_{0, 1}) a_i(\xi_i\circ\delta_{\ell_i^{-1}} \bar{\xi}) \Lambda_{0, 1}\ \mathrm{d} \bar{z} \mathrm{d} \bar{t}.
			\end{aligned}
			$$

		Notice that
		$$
		\begin{aligned}
			& |(\bar{\mathcal{X}}_i(\Phi_i)+\frac{Q-2}{2} \Phi_i) a_i(\xi_i\circ\delta_{\ell_i^{-1}} \bar{\xi}) \Phi_i-(\bar{\mathcal{X}}_i(\Lambda_{0, 1})+\frac{Q-2}{2} \Lambda_{0, 1}) a_i(\xi_i\circ\delta_{\ell_i^{-1}} \bar{\xi}) \Lambda_{0, 1}| \\
			& \leq C\|a_i\|_{L^{\infty}(D_1)}  |(\bar{\mathcal{X}}_i(\Phi_i)+\frac{Q-2}{2} \Phi_i)  \Phi_i-(\bar{\mathcal{X}}_i(\Lambda_{0, 1})+\frac{Q-2}{2} \Lambda_{0, 1})  \Lambda_{0, 1}| \\
			& = C\|a_i\|_{L^{\infty}(D_1)}  |(\bar{\mathcal{X}}_i(\Phi_i-\Lambda_{0, 1})+\frac{Q-2}{2} (\Phi_i-\Lambda_{0, 1}))  \Phi_i-(\bar{\mathcal{X}}_i(\Lambda_{0, 1})+\frac{Q-2}{2} \Lambda_{0, 1})  (\Lambda_{0, 1}-\Phi_i)| \\
&\leq C\|a_i\|_{L^{\infty}(D_1)}(\|\bar{\xi}\|_{\mathbb{H}^n}|\nabla_{\mathbb{H}^n}(\Phi_i-\Lambda_{0,1})|+\|\bar{\xi}\|_{\mathbb{H}^n}^2|\partial_t(\Phi_i-\Lambda_{0,1})|+|\Phi_i-\Lambda_{0,1}|)\Phi_i\\
&\quad +C\|a_i\|_{L^{\infty}(D_1)}(\|\bar{\xi}\|_{\mathbb{H}^n}|\nabla_{\mathbb{H}^n}\Lambda_{0,1}|+\|\bar{\xi}\|_{\mathbb{H}^n}^2|\partial_t\Lambda_{0,1}|+\Lambda_{0,1})|\Phi_i-\Lambda_{0,1}|
		\end{aligned}
		$$
and 
$$
\begin{aligned}
\left|\Lambda_{0,1}(\xi)\right| & \leq C\left(1+\|\xi\|_{\mathbb{H}^n}\right)^{2-Q}, \\
\left|\nabla_{\mathbb{H}^n} \Lambda_{0,1}(\xi)\right| & \leq C\left(1+\|\xi\|_{\mathbb{H}^n}\right)^{1-Q}, \\
\left|\partial_t \Lambda_{0,1}(\xi)\right| & \leq C\left(1+\|\xi\|_{\mathbb{H}^n}\right)^{-Q} .
\end{aligned}
$$
		By Proposition \ref{pro 2.3 in jde}, we have
		$$
		\Phi_i=u_i(\xi_i)^{-1} u_i(\xi_i\circ \delta_{\ell_i^{-1}} \bar{\xi})\leq Cu_i(\xi_i)^{-2}\ell_i^{Q-2}\|\bar{\xi}\|^{2-Q}_{\mathbb{H}^n}=Cu_i(\xi_i)^{\tau_i}\|\bar{\xi}\|^{2-Q}_{\mathbb{H}^n}.
		$$
		According to Lemma \ref{lemma 5.9 in asan}, we have 
		$
		\Phi_i\leq C(1+\|\bar{\xi}\|_{\mathbb{H}^n})^{2-Q}.
		$
		Using Lemma \ref{lemma 5.9 in asan} and Corollary \ref{corollary 4.4 in imrn}, we have
		$$
		\begin{aligned}
			& \quad u_i(\xi_i)^2|\mathcal{G}_i(r)-\hat{\mathcal{G}}_i(r)| \\
			& \leq C\|a_i\|_{L^{\infty}(D_1)}  u_i(\xi_i)^{(2Q-8) /(Q-2 )}\\
			&\quad \times \int_{D_{\ell_i r}} \sum_{j=0}^1|\nabla_{\mathbb{H}^n}^j(\Phi_i-\Lambda_{0, 1})|(1+{|\bar{\xi}|}_{\mathbb{H}^n})^{2 -Q+j}+|\partial_t(\Phi_i-\Lambda_{0, 1})|(1+{|\bar{\xi}|}_{\mathbb{H}^n})^{4-Q}\ \mathrm{d} \bar{z} \mathrm{d} \bar{t} \\
			& \leq C\|a_i\|_{L^{\infty}(D_1)} \begin{cases} r^2 & \text { if } Q=4, \\
				\max \{r^2,{\left\|\nabla_{\mathbb{H}^n}^2 a_i\right\|_{L^{\infty}\left(D_1\right)}} r\} & \text { if } Q=6, \\
				\max \{r^{2},{\|\nabla_{\mathbb{H}^n}^2 a_i\|_{L^{\infty}(D_1)}} \ln (r u_i(\xi_i)^{2 /(Q-2 )})\} & \text { if } Q=8, \\
				\max \{r^{2 },{\|\nabla_{\mathbb{H}^n}^2 a_i\|_{L^{\infty}(D_1)}} u_i(\xi_i)^{2(Q-8) /(Q-2 )}\} & \text { if } Q\geq 10 ,\end{cases}
		\end{aligned}
		$$
		where $C>0$ depends only on $n$, $A_0$ and $A_1$. Next, by a change of variables, the divergence theorem and direct computations, we see that
		$$
		\begin{aligned} 
			\hat{\mathcal{G}}_i(r) & =- \int_{D_r}((\tilde{x},\tilde{y},2\tilde{t})\cdot \nabla \Lambda_{0,\ell_i}+\frac{Q-2 }{2} \Lambda_{0,\ell_i}) a_i( \xi_i\circ\tilde{\xi}) \Lambda_{0,\ell_i}\ \mathrm{d} \tilde{z} \mathrm{d} \tilde{t} \\ 
			& =- \int_{D_r}(\frac{1}{2} a_i( \xi_i\circ\tilde{\xi}) (\tilde{x},\tilde{y},2\tilde{t}) \cdot\nabla \Lambda^2_{0,\ell_i} +\frac{Q-2 }{2}  a_i( \xi_i\circ\tilde{\xi}) \Lambda^2_{0,\ell_i})\ \mathrm{d} \tilde{z} \mathrm{d} \tilde{t} \\ 
			& =- \int_{\partial D_r}\frac{1}{2}a_i( \xi_i\circ\tilde{\xi})(\tilde{x},\tilde{y},2\tilde{t}) \Lambda^2_{0,\ell_i}\cdot N\ \mathrm{d} H_{Q-2}+\int_{D_r}\frac{Q}{2} a_i( \xi_i\circ\tilde{\xi}) \Lambda^2_{0,\ell_i}\ \mathrm{d} \tilde{z} \mathrm{d} \tilde{t} \\
			& \quad+\int_{D_r}\frac{1}{2} \nabla a_i( \xi_i\circ\tilde{\xi})\cdot (\tilde{x},\tilde{y},2\tilde{t}) \Lambda^2_{0,\ell_i}\ \mathrm{d} \tilde{z} \mathrm{d} \tilde{t}-\int_{D_r}\frac{Q-2 }{2}  a_i( \xi_i\circ\tilde{\xi}) \Lambda^2_{0,\ell_i}\ \mathrm{d} \tilde{z} \mathrm{d} \tilde{t} \\ 
			& \geq \int_{D_r}(\frac{1}{2} (\tilde{x},\tilde{y},2\tilde{t})\cdot \nabla a_i(\xi_i\circ\tilde{\xi})+ a_i(\xi_i\circ\tilde{\xi})) \Lambda_{0,\ell_i}^2\ \mathrm{d} \tilde{z} \mathrm{d} \tilde{t}-Cu_i(\xi_i)^{-2}\|a_i\|_{L^{\infty}(D_1)} r^{4 -Q}\\
			&=\int_{D_r}(\frac{1}{2} (\tilde{x},\tilde{y}) \cdot\nabla_{\mathbb{H}^n}a_i(\xi_i\circ\tilde{\xi})+\tilde{t} \partial_ta_i(\xi_i\circ\tilde{\xi})+ a_i(\xi_i\circ\tilde{\xi})) \Lambda_{0,\ell_i}^2\ \mathrm{d} \tilde{z} \mathrm{d} \tilde{t}\\
			&\quad -Cu_i(\xi_i)^{-2}\|a_i\|_{L^{\infty}(D_1)} r^{4 -Q},
		\end{aligned}
		$$
		where we used the following estimate
		$$
		\begin{aligned}
			&\quad \int_{\partial D_r}a_i( \xi_i\circ\tilde{\xi})(\tilde{x},\tilde{y},2\tilde{t}) \Lambda^2_{0,\ell_i}\cdot N\ \mathrm{d} \mathcal{H}_{Q-2}\\
			&\leq C\|a_i\|_{L^{\infty}(D_1)}\int_{\partial D_r} (\tilde{x},\tilde{y},2\tilde{t}) \Lambda^2_{0,\ell_i}\cdot N\ \mathrm{d} \mathcal{H}_{Q-2}\\
			&\leq C\|a_i\|_{L^{\infty}(D_1)}\int_{\partial D_r}  \Lambda^2_{0,\ell_i}\frac{2r^4}{(4|\tilde{z}|^6+\tilde{t}^2)^{\frac{1}{2}}}\ \mathrm{d} \mathcal{H}_{Q-2}\\
			&\leq C\|a_i\|_{L^{\infty}(D_1)}\int_{\partial D_r}  \ell_i^{2-Q}r^{4-2Q}\frac{2r^4}{(4|\tilde{z}|^6+\tilde{t}^2)^{\frac{1}{2}}}\ \mathrm{d} \mathcal{H}_{Q-2}\\
			&\leq C\|a_i\|_{L^{\infty}(D_1)}\int_{|z|\leq r} \ell_i^{2-Q}r^{8-2Q}\frac{1}{(r^4-|\tilde{z}|^4)^{\frac{1}{2}}}\ \mathrm{d} \tilde{z}\\
			&\leq C\|a_i\|_{L^{\infty}(D_1)}\int_{0}^{{\frac{\pi}{2}}} \ell_i^{2-Q}r^{8-2Q+Q-4}\cos^{n-1}\theta\ \mathrm{d} \theta\\
			&\leq Cu_i(\xi_i)^{-2}\|a_i\|_{L^{\infty}(D_1)}r^{4-Q}.
		\end{aligned}
		$$
		In the third-to-last inequality above, we employed the identity $\mathrm{d} \mathcal{H}_{Q-2} = \frac{\sqrt{4|\tilde{z}|^6+\tilde{t}^2}}{|\tilde{t}|} \mathrm{~d} \bar{z}$, 
		which is derived from the Area Formula (see \cite[Theorem 3.2.3]{Federer}).

        By the homogeneous Taylor formula on step-two Carnot groups, we have
\[
\begin{aligned}
a_i(\xi_i\circ\tilde\xi)
={}&a_i(\xi_i)
+\nabla_{\mathbb H^n}a_i(\xi_i)\cdot(\tilde x,\tilde y)
+\partial_ta_i(\xi_i)\tilde t\\
&+\frac12(\tilde x,\tilde y)\cdot
\operatorname{Hess}_{\mathrm{sym}}a_i(\xi_i)
\cdot\binom{\tilde x}{\tilde y}\\
&+\tilde t\sum_{j=1}^n
\left[
\tilde x_jX_j\partial_ta_i(\xi_i)
+\tilde y_jY_j\partial_ta_i(\xi_i)
\right]\\
&+\frac1{3!}
\left[
\sum_{j=1}^n
\left(\tilde x_jX_j+\tilde y_jY_j\right)
\right]^3a_i(\xi_i)
+R_4(\tilde\xi),
\end{aligned}
\]
where the coefficients \(\tilde x_j,\tilde y_j\) in the
third-order term are regarded as constants, and
\[
|R_4(\tilde\xi)|
\leq
C\left\|\nabla_{\mathbb H^n}^{4}a_i\right\|_{L^\infty(D_1)}
\|\tilde\xi\|_{\mathbb H^n}^{4}.
\]

		Thus, by direct computations we see that	
	
		\begin{equation}
				\begin{aligned}
					\hat{\mathcal{G}}_i(r)&\geq \int_{D_r} \big(a_i(\xi_i)+\sum_{k=1}^n X_kX_ka_i(\xi_i)\tilde{x}\tilde{x}+\sum_{k=1}^nY_kY_ka_i(\xi_i)\tilde{y}\tilde{y}\big)\Lambda_{0,\ell_i}^2 \ \mathrm{d}\tilde{z}\mathrm{d}\tilde{t} \\
					&\quad  -Cu_i(\xi_i)^{-2}\|a_i\|_{L^{\infty}(D_1)} r^{4 -Q}\\
					&\geq 	\int_{D_r}( a_i(\xi_i)+\frac{1}{2 n} \Delta_{\mathbb{H}^n} a_i(\xi_i)|\tilde{z}|^2) \Lambda_{0,\ell_i}^2\ \mathrm{d} \tilde{z}\mathrm{d} \tilde{t} \\
				& \quad-C \|\nabla_{\mathbb{H}^n}^4 a_i\|_{L^{\infty}(D_1)} \int_{D_r}\|\tilde{\xi}\|^4 \Lambda_{0,\ell_i}^2\ \mathrm{d} \tilde{z}\mathrm{d} \tilde{t}-Cu_i(\xi_i)^{-2}\|a_i\|_{L^{\infty}(D_1)} r^{4 -Q},
				\end{aligned}
			\end{equation}
			$$
			\begin{aligned}
				&\quad u_i(\xi_i)^2 \int_{D_r}( a_i(\xi_i)+\frac{1}{2 n} \Delta_{\mathbb{H}^n} a_i(\xi_i)|\tilde{z}|^2) \Lambda_{0,\ell_i}^2\ \mathrm{d} \tilde{z}\mathrm{d} \tilde{t}\\
				& \geq C \begin{cases} a_i(\xi_i) \ln (r u_i(\xi_i)^{\frac{2}{Q-2}}) & \text { if } Q=4 ,  \\
					a_i(\xi_i) u_i(\xi_i)^{\frac{2(Q-4 )}{Q-2 }}+\frac{1}{2 n} \Delta_{\mathbb{H}^n} a_i(\xi_i) \ln (r u_i(\xi_i)^{\frac{2}{Q-2 }}) & \text { if } Q=6, \\
					a_i(\xi_i) u_i(\xi_i)^{\frac{2(Q-4 )}{Q-2 }}+\frac{1}{2 n} \Delta_{\mathbb{H}^n} a_i(\xi_i) u_i(\xi_i)^{\frac{2(Q-6 )}{Q-2 }} & \text { if } Q\geq 8 ,\end{cases}
			\end{aligned}
			$$
			and
			$$
			u_i(\xi_i)^2 \int_{D_r}\|\tilde{\xi}\|^4 \Lambda_{0,\ell_i}^2\ \mathrm{d} \tilde{z}\mathrm{d} \tilde{t} \leq C \begin{cases}r^{8-Q} & \text { if } Q=4,6, \\ \ln (r 	u_i(\xi_i)^{\frac{2}{Q-2 }}) & \text { if } Q=8, \\ u_i(\xi_i)^{\frac{2(Q-8)}{Q-2 }} & \text { if } Q\geq 10, \end{cases}
			$$
			where $C>0$ depends only on $n$ and $\sup\|\nabla_{\mathbb{H}^n}^4 a_i\|_{L^{\infty}(D_1)}$.
				
			On the other hand, by Proposition \ref{pro 2.3 in jde},
			\[
\begin{aligned}
&\left(\frac{Q}{p_i+1}-\frac{Q-2}{2}\right)
\int_{D_r(\xi_i)}u_i^{p_i+1}\,\mathrm d\xi\\
&\qquad
-\frac{1}{p_i+1}
\int_{\partial D_r(\xi_i)}
u_i^{p_i+1}\mathcal X_iI\cdot N\,
\mathrm d\mathcal H_{Q-2}
\geq
-Cr^{-Q}u_i(\xi_i)^{-1-p_i}.
\end{aligned}
\]

			Thus, 
			\begin{equation*}
			u_i(\xi_i)^2 \int_{\partial D_r(\xi_i)}\mathcal{D}( \xi_i^{-1}\circ \xi, u_i, \nabla_{\mathbb{H}^{n}} u_i) \geq \ u_i(\xi_i)^2\hat{\mathcal{G}}_i(r)-u_i(\xi_i)^2|\mathcal{G}_i(r)-\hat{\mathcal{G}}_i(r)|-C  r^{-Q} u_i(\xi_i)^{1-p_i},
			\end{equation*}
			then the proposition follows immediately.
			\end{proof}

			\begin{proposition}\label{isolated=isolated simple}
				Assume as in Lemma \ref{Harnack ineq 2}. Suppose that for large $i$,
				\begin{enumerate}

				\item[$(1)$] $\beta_i \geq 0$ \quad if $n=2$,
				
				\item[$(2)$] $\beta_i \geq (C_0+1)\|a_i\|_{D_1} \ln u_i(\xi_i)$ \quad if $n=3$,
				
				\item[$(3)$] $\beta_i \geq (C_0+1)\left\|a_i\right\|_{D_1} u_i(\xi_i)^{\frac{2(Q-8)}{Q-2}}$\quad  if $n\geq 4 $,
				\end{enumerate}
				where
				$$
				\beta_i:= \begin{cases} a_i(\xi_i) u_i(\xi_i)^{\frac{2(Q-4 )}{Q-2}}+\frac{1}{2 n} \Delta_{\mathbb{H}^n} a_i(\xi_i) \ln u_i(\xi_i) & \text { if } n=2, \\  a_i(\xi_i) u_i(\xi_i)^{\frac{2(Q-4 )}{Q-2 }}+\frac{1}{2 n} \Delta_{\mathbb{H}^n} a_i(\xi_i) u_i(\xi_i)^{\frac{2(Q-6)}{Q-2 }} & \text { if } n\geq 3,\end{cases}
				$$
				$C_0$ is the constant in Proposition \ref{pro 5.1 in imrn} with $\rho=1$. Then, after passing to a subsequence, $\xi_i \rightarrow 0$ is an isolated simple blow-up point of $u_i$.
			\end{proposition}
			
			\begin{proof}
By Proposition \ref{solution is bubble}, 
$\bar w_i^G(r)$ has precisely one critical point in
\[
0<r<r_i:=R_i u_i(0)^{-\frac{p_i-1}{2}}.
\]
If $0$ were not an isolated simple blow-up point, let
$\mu_i\geq r_i$ be the second critical point of $\bar w_i^G$.
Then
\[
\lim_{i\to\infty}\mu_i=0.
\]

Define
\[
w_i(\xi)
:=
\mu_i^{\frac{2}{p_i-1}}u_i(\delta_{\mu_i}\xi)
\quad \text{ and }\quad 
\tilde a_i(\xi)
:=
\mu_i^2a_i(\delta_{\mu_i}\xi).
\]
Then
\be\label{rescaled equation at second weighted critical radius}
\begin{cases}
-\Delta_{\mathbb H^n}w_i
=\tilde a_iw_i+w_i^{p_i},\\[1mm]
w_i(0)
=\mu_i^{\frac{2}{p_i-1}}u_i(\xi_i)
\geq R_i^{\frac{2}{p_i-1}}\rightarrow\infty,\\[1mm]
\|\xi\|_{\mathbb H^n}^{\frac{2}{p_i-1}}w_i(\xi)
\leq A_1.
\end{cases}
\ee
By the definition of $w_i$ and the choice of $\mu_i$, the function
\[
\bar w_i(s)
:=
s^{\frac{2}{p_i-1}}\mathcal M_0(w_i,s)
\]
has precisely one critical point in $(0,1)$, and
\be\label{6.6 in asan}
\left.
\frac{\mathrm d}{\mathrm ds}
\left[
s^{\frac{2}{p_i-1}}\mathcal M_0(w_i,s)
\right]\right|_{s=1}
=0.
\ee
Thus, $0$ is an isolated simple blow-up point of the sequence ${w_i}$.

Applying Lemma \ref{Harnack ineq 2}, Proposition \ref{pro 2.3 in jde},
the B\"ocher-type theorem and the interior subelliptic estimates, after
passing to a subsequence, we obtain
\be\label{bocher solution2}
w_i(0)w_i(\xi)
\rightarrow
w(\xi)
=
c_0\|\xi\|_{\mathbb H^n}^{2-Q}+k(\xi)
\quad\text{in }C^2_{\rm loc}(\mathbb H^n\setminus\{0\}),
\ee
where $c_0>0$ and
\[
\Delta_{\mathbb H^n}k=0
\]
in a neighborhood of the origin.

By the mean value property (see \cite[Theorem 2.1]{GL}), we have
\be\label{weighted mean of bocher limit}
\mathcal M_0(w,s)
=
c_0s^{2-Q}+k(0).
\ee
Multiplying \eqref{6.6 in asan} by $w_i(0)$ and passing to the limit,
using the $C^2$-convergence on an annulus containing $\partial D_1$ and
$p_i\to(Q+2)/(Q-2)$, we obtain
\[
\begin{aligned}
0
&=
\left.\frac{\mathrm d}{\mathrm ds}\right|_{s=1}
\left[
 s^{\frac{Q-2}{2}}\mathcal M_0(w,s)
\right]\\
&=
\left.\frac{\mathrm d}{\mathrm ds}\right|_{s=1}
\left[
 c_0s^{\frac{2-Q}{2}}
 +k(0)s^{\frac{Q-2}{2}}
\right]\\
&=
\frac{Q-2}{2}\bigl(k(0)-c_0\bigr).
\end{aligned}
\]
Consequently,
\be\label{positive regular part from weighted mean}
k(0)=c_0>0.
\ee
Therefore,
\[
w(\xi)
=
c_0\|\xi\|_{\mathbb H^n}^{2-Q}+c_0+c_0h(\xi),
\]
where $h(\xi):=c_0(k(\xi)-k(0))$.
By Proposition \ref{pro 4.3 in asan}, for all sufficiently small
$\sigma>0$,
\be\label{7.1 in imrn}
\int_{\partial D_\sigma}
\mathcal D(\xi,w,\nabla_{\mathbb H^n}w)
\,\mathrm d\mathcal H_{Q-2}<0.
\ee

				On the other hand, if $ n=2$,
				by Proposition \ref{pro 5.1 in imrn} and item (i) in the assumption we have,
				\begin{equation*}
\liminf_{\sigma\rightarrow 0}\ \liminf _{i \rightarrow \infty} w_i(0)^2 \int_{\partial D_\sigma} \mathcal{D}\left(\xi, w_i, \nabla_{\mathbb{H}^n} w_i\right) \geq 0,
\end{equation*}
				which contradicts \eqref{7.1 in imrn}. Hence, $\xi_i \rightarrow 0$ has to be an isolated simple blow-up point of $u_i$ upon passing to a subsequence.

				If $n \geq 3$, let
				$$
					\tilde{\beta}_i:  =\tilde{a}_i(0) w_i(0)^{\frac{2(Q-4 )}{Q-2}}+\frac{1}{2 n} \Delta_{\mathbb{H}^n} \tilde{a}_i(0) w_i(0)^{\frac{2(Q-6)}{Q-2 }} 
					 =(1+o(1)) \mu_i^{Q-2} \beta_i .
				$$
				Since $\|\tilde{a}_i\|_{C^4(D_1)} \leq  \mu_i^2A_0$ and $\|\tilde{a}_i\|_{D_1} \leq \mu_i^{6}\|a_i\|_{D_1}$, for $n=3$, we see that
				$$
				\begin{aligned}
					& \quad \tilde{\beta}_i-C_0\|\tilde{a}_i\|_{D_1} \ln w_i(0)-C_0\|\tilde{a}_i\|_{L^{\infty}(D_1)} \delta^{4-Q} \\
					& \geq(1+o(1)) \mu_i^{6} \beta_i-C_0 \mu_i^{6}\|a_i\|_{D_1} \ln u_i(0)-C_0 \mu_i^{6}\|a_i\|_{D_1} \ln \mu_i^{\frac{2}{p_i-1}} \\
					&\quad -C_0 \mu_i^{2}\|a_i\|_{L^{\infty}(D_1)}\delta^{-4} \\
					& \geq(1+o(1)) \mu_i^{6}(C_0+1)\|a_i\|_{D_1} \ln u_i(0)-C_0 \mu_i^{6}\|a_i\|_{D_1} \ln u_i(0)-C_0 \mu_i^{6}\|a_i\|_{D_1} \ln \mu_i^{\frac{2}{p_i-1}} \\
					& \quad-C_0 \mu_i^{2}\|a_i\|_{L^{\infty}(D_1)} \delta^{-4} \geq 0,
				\end{aligned}
				$$
				and for $n\geq4 $,
				$$
				\begin{aligned}
					& \quad \tilde{\beta}_i-C_0\|\tilde{a}_i\|_{D_1} w_i(0)^{\frac{2(Q-8)}{Q-2}}-C_0\|\tilde{a}_i\|_{L^{\infty}(D_1)} \delta^{4 -Q} \\
					& \geq(1+o(1)) \mu_i^{Q-2} \beta_i-C_0(1+o(1)) \mu_i^{Q-2}\|a_i\|_{D_1} u_i(0)^{\frac{2(Q-8)}{Q-2}}-C_0 \mu_i^{2}\|a_i\|_{L^{\infty}(D_1)} \delta^{4 -Q} \\
					& \geq(1+o(1)) \mu_i^{Q-2}(C_0+1)\|a_i\|_{D_1} u_i(0)^{\frac{2(Q- 8)}{Q-2}}-C_0(1+o(1)) \mu_i^{Q-2}\|a_i\|_{D_1} u_i(0)^{\frac{2(Q-8)}{Q-2}} \\
					& \quad-C_0 \mu_i^{2}\|a_i\|_{L^{\infty}(D_1)} \delta^{4-Q} \geq 0.
				\end{aligned}
				$$
				 By Proposition \ref{pro 5.1 in imrn}, we have
				$$
				\liminf_{\sigma\rightarrow 0} \ \liminf _{i \rightarrow \infty} w_i(0)^2 \int_{\partial D_\sigma} \mathcal{D}(\xi, w_i, \nabla_{\mathbb{H}^{n}} w_i) \geq 0 ,
				$$
				which contradicts \eqref{7.1 in imrn} again. Hence, $\xi_i \rightarrow 0$ has to be an isolated simple blow-up point of $\left\{u_i\right\}$ upon passing to a subsequence. Therefore, we complete the proof of Proposition \ref{isolated=isolated simple}.
			\end{proof}

\section{Proof of the Main Theorems}\label{section 7}

Let \(u\in C^2(D_3)\) be a positive solution of
\be\label{final equation}
\Delta_{\mathbb H^n}u+au+u^{p_\tau}=0
\qquad\text{in }D_3,
\ee
where
\[
0\le a\in C^4(D_3),
\qquad
p_\tau=\frac{Q+2}{Q-2}-\tau,
\qquad
0\le\tau<\frac{2}{Q-2}.
\]
			
			\begin{proposition}\label{6.1 in rnac}
				Assume the conditions above, then for any $0<\varepsilon<1$ and $R>1$, there exist large positive constants \(C_1\) and \(C_2\)
                depending only on \(n\), \(\|a\|_{C^2(D_2)}\),
                \(\varepsilon\), and \(R\), such that if u solves equation \eqref{final equation} 
				then the following statement holds. If
				$$
				\max _{\bar{D}_2} \operatorname{dist}_{\mathbb{H}^n}(\xi, \partial D_2)^{\frac{Q-2}{2}} u(\xi) \geq C_1,
				$$
				then $p_\tau \geq \frac{Q+2}{Q-2}-\varepsilon$ and there exists a finite set $S$ of local maximum points of $u$ in $D_2$ such that:
				
				$(i)$ For any $\eta \in S$, it holds
				\be\label{eq:73 in IMRN}
				\|u(\eta)^{-1} u(\eta\circ\delta_{ u(\eta)^{-(p_\tau-1) / 2 }} \cdot)-\Lambda_{0,1}(\cdot)\|_{C^2(D_{2 R})}<\varepsilon,
				\ee
				where $\Lambda_{0,1}$ is defined in the introduction.
				
				$(ii)$ If $\xi_1$, $\xi_2$ in $S$ and $\xi_1 \neq \xi_2$, then
				$$
				D_{R u(\xi_1)^{-(p_\tau-1) / 2 }}(\xi_1) \cap D_{R u(\xi_2)^{-(p_\tau-1) / 2 }}(\xi_2)=\emptyset .
				$$
				
				$(iii)$ $u(\xi) \leq C_2 \operatorname{dist}_{\mathbb{H}^n}(\xi, S)^{-2  /(p_\tau-1)}\quad  \text{ for } \xi \in D_{7 / 4}$.	
			\end{proposition}

		\begin{proof}

    The proof follows the blow-up and point-selection argument in
    \cite[Lemma 6.2 and Proposition 6.1]{Niu Tang Zhou IMRN},
    with the Euclidean translations and dilations replaced by the
    corresponding left translations and nonisotropic dilations on
    \(\mathbb H^n\). We only explain the dimension-dependent
    classification step.

    In the contradiction argument, let \(u_i\) be a blow-up sequence,
    let \(\eta_i\) be a point selected by the point-selection procedure,
    and set
    \[
    v_i(\xi)
    :=
    u_i(\eta_i)^{-1}
    u_i\bigl(\eta_i\circ\delta_{u_i(\eta_i)^{\frac{1-p_{\tau_i}}{2}}}\xi\bigr).
    \]
    The point-selection property gives a uniform bound for \(v_i\) on
    Kor\'anyi balls whose radii tend to infinity. Hence, by the interior
    sub-elliptic estimates, after passing to a subsequence,
    \[
    v_i\rightarrow v
    \qquad\text{in }C_{\mathrm{loc}}^2(\mathbb H^n),
    \]
    where
    \[
    v(0)=1,\qquad v>0,
    \]
    and
    \[
    -\Delta_{\mathbb H^n}v=v^{p_\infty}
    \qquad\text{in }\mathbb H^n
    \]
    for some $1<p_\infty\leq\frac{Q+2}{Q-2}$.
    The sub-critical Liouville theorem on the Heisenberg group
    \cite{Ma Ou} excludes the case
    $p_\infty<\frac{Q+2}{Q-2}$.
    Therefore,
    \[
    p_\infty=\frac{Q+2}{Q-2},
    \]
    which implies that \(v\) is a
    Jerison--Lee bubble by \cite{LLM,LiJungangArxiv}.

    Having established the bubble profile at every selected point,
    the remainder of the point-selection and iteration argument is
    identical to that in
    \cite[Proposition 6.1]{Niu Tang Zhou IMRN}.
    This yields $(i)$--$(iii)$.
		\end{proof}

			Here, we prove that the isolated simple blow-up points are separated by a fixed, positive distance.

			\begin{proposition}\label{pro 6.3 in rnac}
				Assume the conditions above. Suppose that $0 \leq a \in C^4(D_3)$, $\Delta_{\mathbb{H}^n} a \geq 0$ on $\{\zeta: a(\zeta)<d\} \cap D_2$ for some constant $d>0$, and further that $\Delta_{\mathbb{H}^n} a>\gamma>0$ on $\{\zeta: a(\zeta)<d\} \cap D_2$ for some constant $\gamma$ if $n \geq 3 $. Then for $\varepsilon>0$, $R>1$ and any solution of \eqref{final equation} with $\max _{\bar{D}_2} \operatorname{dist}_{\mathbb{H}^n}(\zeta, \partial D_2)^{(Q-2) / 2} u(\zeta) \geq C_1$, we have
				$$
				d_{\mathbb{H}^n}(\zeta_1,\zeta_2) \geq \delta^*>0 \quad \text { for any } \zeta_1, \zeta_2 \in S \cap D_{3 / 2} \text { and } \zeta_1 \neq \zeta_2,
				$$
				where \(\delta^*>0\) depends only on
                \(n,A_0,d,\varepsilon,R\).
   
			\end{proposition}

                \begin{proof}
	The idea is similar to that of
	\cite[Proposition 5.2]{Jin Li Xiong}, but
	\cite[Lemma 2.1]{Niu Peng Xiong JFA} has to be used since
	our equation is defined in a bounded domain.

	Suppose the contrary, for some $\varepsilon$, $R$ and $d>0$, there exist sequence $\left\{p_{\tau_i}\right\}$ and nonnegative potentials 
				$a_i \rightarrow a$ in $C^4\left(D_3\right)$ with $\left\|a_i\right\|_{C^4\left(D_3\right)} \leq A_0$, 
				satisfying the assumptions for $a$, and a sequence of corresponding solutions $u_i$ such that
				$$
				\lim _{i \rightarrow \infty} \inf _{j \neq l} d_{\mathbb{H}^n}(\zeta_{i_j}, \zeta_{i_l})=0.
				$$
                where $\zeta_{i_j}$, $\zeta_{i_l} \in S_i \cap D_{3 / 2}$ associated to $u_i$ defined in Proposition \ref{6.1 in rnac}.

                Upon passing to a subsequence, we assume $\zeta_{i_j}, \zeta_{i_l} \rightarrow \bar{\zeta} \in D_{3 / 2}$. Define $f_i(\zeta): S_i \rightarrow$ $(0, \infty)$ by $f_i(\zeta)=\min _{\eta \in S_i \backslash\{\zeta\}}\|\eta^{-1}\circ\zeta\|_{\mathbb{H}^n}$. Let $\bar R_i \rightarrow \infty$ with $\bar R_i f_i\left(\zeta_{i_j}\right) \rightarrow 0$ as $i \rightarrow \infty$. By Lemma 2.1 in \cite{Niu Peng Xiong JFA}, one can find $\zeta_{i_1} \in S_i \cap D_{2\bar R_i f_i(\zeta_{i_ j})}(\zeta_{i_j})$ satisfying

$$
f_i\left(\zeta_{i_1}\right) \leq\left(2 \bar R_i+1\right) f_i\left(\zeta_{i_j}\right) \quad \text { and } \quad \min _{\zeta \in S_i \cap D_{\bar R_i f_i(\zeta_{i_1})(\zeta_{i_1})}} f_i(\zeta) \geq \frac{1}{2} f_i(\zeta_{i_1}) .
$$
	Let
	\[
	\sigma_i:=f_i(\zeta_{i_1})
	\]
	and choose $\zeta_{i_2}\in S_i$ such that
	\[
	d_{\mathbb H^n}(\zeta_{i_1},\zeta_{i_2})
	=
	\sigma_i.
	\]
	Define
	\[
	w_i(\xi)
	:=
	\sigma_i^{\frac{2}{p_i-1}}
	u_i\bigl(
	\zeta_{i_1}\circ\delta_{\sigma_i}\xi
	\bigr).
	\]
	Set
	\[
	\xi_i
	:=
	\delta_{\sigma_i^{-1}}
	\bigl(
	\zeta_{i_1}^{-1}\circ\zeta_{i_2}
	\bigr).
	\]
	Then
	\[
	\|\xi_i\|_{\mathbb H^n}=1.
	\]
	After passing to a subsequence,
	\[
	\xi_i\rightarrow\bar\xi,
	\qquad
	\|\bar\xi\|_{\mathbb H^n}=1.
	\]

	The rest of the proof is divided into three steps:

	\begin{enumerate}
		\item[(1)] Prove that $0$ and
		$\xi_i\to\bar\xi$ are two isolated blow-up points of
		$\{w_i\}$.

		\item[(2)] By Proposition
		\ref{isolated=isolated simple}, after passing to a subsequence,
		$0$ and $\xi_i\to\bar\xi$ have to be isolated simple
		blow-up points of $\{w_i\}$.

		\item[(3)] Since $w_i(0)w_i$ tends to a function with at least
		two poles, we can derive a contradiction by the Pohozaev
		identity.
	\end{enumerate}

	For Steps (1) and (3), the proof is the same as
	that of \cite[Proposition 5.2]{Jin Li Xiong}, with the
	localization property above replacing the global closest-pair
	property there, and with the Heisenberg left translations and
	dilations replacing the Euclidean ones.

	For Step (2), let
	\[
	\bar a_i(\xi)
	:=
	\sigma_i^2
	a_i\bigl(
	\zeta_{i_1}\circ\delta_{\sigma_i}\xi
	\bigr)
	\]
	and verify the assumptions in Proposition
	\ref{isolated=isolated simple}. We only show it if $n=3$;
	the other cases are similar.

	By the homogeneity of the sub-Laplacian,
	\[
	\bar a_i(0)
	=
	\sigma_i^2a_i(\zeta_{i_1}),
	\qquad
	\Delta_{\mathbb H^n}\bar a_i(0)
	=
	\sigma_i^4
	\Delta_{\mathbb H^n}a_i(\zeta_{i_1}).
	\]
	By the assumptions on $a_i$, for large $i$, we have
	\[
	\begin{aligned}
		&\quad
		\frac{
		\bar a_i(0)
		w_i(0)^{\frac{2(Q-4)}{Q-2}}
		+
		\frac1{2n}
		\Delta_{\mathbb H^n}\bar a_i(0)
		w_i(0)^{\frac{2(Q-6)}{Q-2}}
		}{
		\|\bar a_i\|_{D_{1/2}}
		}\\
		&\geq
		w_i(0)^{\frac{2(Q-6)}{Q-2}}
		\|a_i\|_{D_2}^{-1}
		\left\{
		a_i(\zeta_{i_1})
		w_i(0)^{\frac4{Q-2}}
		\sigma_i^{-2}
		+
		\frac1{2n}
		\Delta_{\mathbb H^n}a_i(\zeta_{i_1})
		\right\}\\
		&\geq
		w_i(0)^{\frac{2(Q-6)}{Q-2}}
		\|a_i\|_{D_2}^{-1}
		\frac{\gamma}{2n}.
	\end{aligned}
	\]
	Since
	\[
	\frac{
		w_i(0)^{\frac{2(Q-6)}{Q-2}}
		\|a_i\|_{D_2}^{-1}
		\frac{\gamma}{2n}
		}{
		\ln w_i(0)
		}
	\longrightarrow\infty
	\qquad\text{if }n=3,
	\]
	it follows from item (2) of Proposition
	\ref{isolated=isolated simple} that $0$ is an isolated simple
	blow-up point of $\{w_i\}$.

	Similarly, after translating $\xi_i$ to the origin, one can show
	that $\xi_i\to\bar\xi$ is another isolated simple blow-up point
	of $\{w_i\}$.

	Therefore, Step (2) follows, and Steps (1) and
	(3) give the desired contradiction. This completes the
	proof of Proposition \ref{pro 6.3 in rnac}.
\end{proof}

			After passing to a subsequence, if $\{u_i\}$ stays bounded in $L^{\infty}(\Omega)$, then sub-elliptic estimates further imply that it remains bounded in $C^{2, \alpha}$ for $0<\alpha<1$. Then, we prove our main theorems.
			
			\begin{proof}[Proof of Theorem \ref{main thm 3}]
				We first prove that $\|u\|_{L^{\infty}\left(D_{5 / 4}\right)} \leq C$. Suppose the contrary, that is, there exists a sequence of solutions $u_i$ of \eqref{final equation} satisfying all the hypotheses of the corresponding theorem
                such that
            $$
            \|u_i\|_{L^\infty(D_{5/4})}\to\infty.
            $$

            For any fixed $\varepsilon>0$ sufficiently small and $R \gg 1$, by Proposition \ref{pro 6.3 in rnac} the set $S_i$ associated to $u_i$ defined by Proposition \ref{6.1 in rnac} only consists of finite many points in $D_{3 / 2}$ with a uniform positive lower bound of distances between each two points, if $S_i \cap D_{3 / 2}$ has points more than 1. By the contradiction assumption $\left\|u_i\right\|_{L^{\infty}\left(D_{5 / 4}\right)} \rightarrow \infty$ and Proposition \ref{6.1 in rnac}, $S_i \cap D_{11 / 8}$ is not empty and has only isolated blow-up points of $u_i$ after passing to a subsequence. By Proposition \ref{isolated=isolated simple}, these isolated blow-up points have to be isolated simple blow-up points. Suppose that $\xi_i \rightarrow \bar{\xi} \in \bar{D}_{11 / 8}$ is an isolated simple blow-up point of $u_i$. By Proposition \ref{pro 2.3 in jde}, we have
				$$
				\left|u_i\left(\xi_i\right)^2 \int_{\partial D_r(\xi_i)}\mathcal{D}\left(\xi_i^{-1}\circ \xi, u_i, \nabla_{\mathbb{H}^{n}} u_i\right)\right| \leq C(r).
				$$
				
				On the other hand, by the assumption of $a$ and Proposition \ref{pro 5.1 in imrn} we have
				$$
				\liminf _{i \rightarrow \infty} u_i\left(\xi_i\right)^2 \int_{\partial D_r(\xi_i)}\mathcal{D}\left(\xi_i^{-1}\circ \xi, u_i, \nabla_{\mathbb{H}^{n}} u_i	\right)=\infty \quad\text { for some small } r>0,
				$$
				if $n \geq 1$. Hence, we obtain a contradiction and thus $\|u\|_{L^{\infty}\left(D_{5 / 4}\right)} \leq C$. The theorem then follows from interior estimates of solutions.
			\end{proof}
			
			\begin{proof}[Proof of Theorem \ref{main thm 2}]
			 For any fixed $\varepsilon>0$ sufficiently small and $R \gg 1$, let $S_i$ be the set associated to $u_i$ defined by Proposition \ref{6.1 in rnac}.
				
				If $n= 2$, by Proposition \ref{pro 6.3 in rnac} the set $S_i$ only consists of finite many points in $D_{3 / 2}$. Since $u_i\left(\xi_i\right) \rightarrow \infty$ and 
				$\xi_i \rightarrow \bar{\xi}$, by item (iii) of Proposition \ref{6.1 in rnac}, after passing to a subsequence, there exists $\xi_i^{\prime} \in S_i$ such that $\xi_i' \rightarrow \bar{\xi}$ 
				is an isolated blow-up point of $u_i$. By Proposition \ref{isolated=isolated simple}, it has to be an isolated simple blow-up point. By Proposition \ref{pro 2.3 in jde}, 
				we have
				$$
				\left|u_i\left(\xi_i^{\prime}\right)^2 \int_{\partial D_r(\xi_i^\prime)}\mathcal{D}\left({\xi_i^{\prime}}^{-1}\circ \xi, u_i, \nabla_{\mathbb{H}^{n}} u_i\right)\right| \leq C(r).
				$$
				By Proposition \ref{pro 5.1 in imrn}, this proves the theorem for $n=2$.

				If $n \geq 3$, suppose the contrary that for some subsequence, which we still denote as $i$,
				\be\label{80 in imrn}
				\begin{aligned}
					& a_i(\xi_i) u_i(\xi_i)^{\frac{4}{Q-2 }}+\frac{1}{2 n} \Delta_{\mathbb{H}^n} a_i(\xi_i)  
					\geq & \frac{1}{|o(1)|} \begin{cases}u_i\left(\xi_i\right)^{-\frac{4}{Q-2 }} \ln u_i\left(\xi_i\right) & \text { for } n=3, \\
						u_i\left(\xi_i\right)^{-\frac{4}{Q-2 }} & \text { for } n\geq 4. \end{cases}
				\end{aligned}
			\ee

				Let $\mu_i=\operatorname{dist}_{\mathbb{H}^n}\left\{\xi_i, S_i \backslash\left\{\xi_i\right\}\right\}$ and
				
				$$
				\Phi_i(\xi)=\mu_i^{(Q-2 ) / 2} u_i\left(\xi_i\circ \delta_{\mu_i} \xi\right).
				$$
				If $\xi_i \notin S_i$, we have $u_i\left(\xi_i\right) \leq C \mu_i^{-(Q-2) / 2}$. 
				Hence, $\Phi_i(0) \leq C<\infty$ and $\mu_i \rightarrow 0$. Since $\max _{D_{\bar{d}}\left(\xi_i\right)} u_i(\xi) \leq \bar{b} u_i\left(\xi_i\right)$, we conclude that $\Phi_i(\xi) \leq C \bar{b}$ for all $|\xi| \leq \bar{d} / \mu_i$. 

                By the interior subelliptic estimates, after passing to a
subsequence, we have
\[
\Phi_i\rightarrow\Phi
\qquad\text{in }C_{\mathrm{loc}}^2(\mathbb H^n),
\]
where \(\Phi\) is a nontrivial nonnegative entire solution of the
critical equation.
Hence, by the Jerison--Lee classification theorem,
\[
\Phi=\Lambda_{\xi_0,\lambda}
\]
for some \(\xi_0\in\mathbb H^n\) and \(\lambda>0\).
                
		Note that the limiting function has only one critical point. Suppose $\zeta_i \in S_i$ satisfying $d_{\mathbb{H}^n}(\xi_i,\zeta_i)=\mu_i$. Since both $\xi_i$ and $\zeta_i$ are local maximum points of $u_i$, $\nabla_{\mathbb{H}^n} \Phi_i(0)=0$, $\partial_t \Phi_i(0)=0$ and, after passing to subsequence,
		\begin{equation*}
			\begin{aligned}
			\delta_{{\mu_i}^{-1}}(\xi_i^{-1}\circ \zeta_i) \rightarrow \bar{\xi} \quad& \text { with }|\bar{\xi}|=1, \\ \nabla_{\mathbb{H}^n} \Phi_i(\delta_{{\mu_i}^{-1}}(\xi_i^{-1}\circ \zeta_i))=0,& \quad \partial_t \Phi_i(\delta_{{\mu_i}^{-1}}(\xi_i^{-1}\circ \zeta_i))=0.
		\end{aligned}
		\end{equation*}

		We obtain a contradiction. Hence, $\xi_i \in S_i$. Therefore, 0 is an isolated blow-up point of $\Phi_i$. By the assumptions of $a_i$, the contradiction assumption \eqref{80 in imrn} and Proposition \ref{isolated=isolated simple}, 0 is an isolated simple blow-up point. Making use of Propositions \ref{pro 2.3 in jde} and \ref{pro 5.1 in imrn}, we obtain a contradiction again. Therefore, we complete the proof.
			\end{proof}

		\end{document}